\documentclass[9pt, reqno]{amsart}
\usepackage{amsmath, amssymb}

\usepackage{latexsym, amssymb, amsmath, amsthm, amscd}
\usepackage[all,cmtip]{xy}
\usepackage{amsmath}
\usepackage{amsthm}
\usepackage{amssymb}
\usepackage{amsfonts}
\usepackage{amsrefs}
\usepackage{color,authblk}
\usepackage{tikz}
\usepackage{amscd} 
\usepackage{comment}
\usepackage{mathrsfs}
\usepackage{comment}
\usepackage[all]{xy}
\usepackage{graphicx}
\usepackage{authblk}

\usepackage{bbm}
\usepackage{fullpage}
\usepackage[all,cmtip]{xy}
\usepackage{relsize}
\usepackage{amsmath,amscd}
\usepackage{tikz-cd}
\usepackage{tikz}
\usetikzlibrary{matrix}
\usepackage[mathcal]{euscript}
\usepackage{txfonts}
\DeclareMathAlphabet{\matholdcal}{OMS}{cmsy}{m}{n}

\numberwithin{equation}{section} \DeclareMathSizes{2}{10}{12}{13}
\usepackage{etoolbox}
\patchcmd{\section}{\scshape}{\bfseries}{}{}
\makeatletter
\renewcommand{\@secnumfont}{\bfseries}
\makeatother

\newtheorem{thmi}{Theorem}

\makeatletter
\newcommand*{\doublerightarrow}[2]{\mathrel{
		\settowidth{\@tempdima}{$\scriptstyle#1$}
		\settowidth{\@tempdimb}{$\scriptstyle#2$}
		\ifdim\@tempdimb>\@tempdima \@tempdima=\@tempdimb\fi
		\mathop{\vcenter{
				\offinterlineskip\ialign{\hbox to\dimexpr\@tempdima+1em{##}\cr
					\rightarrowfill\cr\noalign{\kern.5ex}
					\rightarrowfill\cr}}}\limits^{\!#1}_{\!#2}}}
\newcommand*{\triplerightarrow}[1]{\mathrel{
		\settowidth{\@tempdima}{$\scriptstyle#1$}
		\mathop{\vcenter{
				\offinterlineskip\ialign{\hbox to\dimexpr\@tempdima+1em{##}\cr
					\rightarrowfill\cr\noalign{\kern.5ex}
					\rightarrowfill\cr\noalign{\kern.5ex}
					\rightarrowfill\cr}}}\limits^{\!#1}}}
\makeatother

\newtheorem{thm}{Proposition}[section]
\newtheorem{Thm}[thm]{Theorem}

\newtheorem{eg}[thm]{Example}
\newtheorem{lem}[thm]{Lemma}
\newtheorem{defn}[thm]{Definition}

\usepackage[hidelinks]{hyperref}
\newenvironment{nouppercase}{%
  \renewcommand{\uppercasenonmath}[1]{}}{}

\title{\large Galois measurings for noncommutative base change of entwined contramodule and entwined comodule categories}

\date{}

\begin{document}

\linespread{0.9}
	
	\begin{nouppercase}
\maketitle
\end{nouppercase}
	
	\centerline{Divya Ahuja \footnote{Department of Mathematics, Indian Institute of Technology, Delhi. Email: divyaahuja1428@gmail.com.} 
\footnote{DA was partially supported by CSIR Fellowship 09/086(1430)/2019-EMR-I.} $\qquad$ Abhishek Banerjee \footnote{Department of Mathematics, Indian Institute of Science, Bangalore. Email: abhishekbanerjee1313@gmail.com.} 
\footnote{\label{footcrg} AB and SK were partially supported by SERB Core Research Grant 2023/004143} $\qquad$ Surjeet Kour \textsuperscript{\ref{footcrg}} \footnote{Department of Mathematics, Indian Institute of Technology, Delhi. Email: koursurjeet@gmail.com.} }

	\begin{abstract}
		\normalsize We study the noncommutative base change of an entwining structure $(A,C,\psi)$ by a Grothendieck category $\mathfrak S$, using two module like categories. These are the categories of entwined comodule objects and entwined contramodule objects in $\mathfrak S$ over the entwining structure 
$(A,C,\psi)$. We consider criteria for maps between these noncommutative spaces, induced by generalized maps between entwining structures, known as measurings, to behave like Galois extensions. We also study conditions for extensions of these noncommutative spaces, understood as functors between module like categories, to have separability, Frobenius or Maschke type properties. 
	\end{abstract}
	
	\medskip
	{MSC(2020) Subject Classification: 16T15, 18E10, 18E50}

	
	\medskip

	\section{Introduction}
	
	The purpose of this paper is to study extensions between noncommutative spaces using properties of functors between module like categories.  The spaces that we consider are noncommutative base changes of entwining structures. We study categorical criteria for morphisms between these noncommutative spaces to behave like Galois extensions, Frobenius or separable extensions, or to have Maschke type properties. 
	
	\smallskip
	Let $k$ be a field. Let $A$ be a $k$-algebra and let $C$ be a $k$-coalgebra. An entwining structure $(A,C,\psi)$ (see Brzezi\'nski and Majid \cite{BM}) consists of a map 
	$\psi:C\otimes A\longrightarrow A\otimes C$ such that the coalgebra structure on $C$ is well behaved with respect to the multiplication on $A$ and vice versa. The data of an entwining structure can be used to formulate a number of noncommutative spaces, such as the ``coalgebra fiber bundle'' corresponding to a Galois extension, or the algebraic counterpart of the quotient produced by an affine algebraic group acting freely on a  scheme. In noncommutative geometry, a space is commonly replaced by an algebra, and an algebra is studied through its category of modules. In a similar vein, the noncommutative space corresponding to an entwining structure $(A,C,\psi)$ is studied by means of modules over it. The latter are called entwined modules (see Brzezi\'nski \cite{Tb}). This has given rise to a rich theory that binds together several notions which appear naturally in the study of comodule algebras, or in Hopf-Galois theory, such as those of Doi-Hopf modules, relative Hopf modules, and Yetter-Drinfeld modules (see, for instance, \cite{Abu}, \cite{BBR1}, \cite{BBR2}, \cite{Tb*}, \cite{Tbart}, \cite{BCT1}, \cite{BCT2}, \cite{CR}, \cite{Jia},  \cite{Sch}). 
	
	\smallskip
	The noncommutative spaces that we consider in this paper are given by the base change of an entwining structure $(A,C,\psi)$ with respect to a Grothendieck category $\mathfrak S$.  If $R$, $R'$ are commutative $k$-algebras, a module over $R\otimes_kR'$ can be expressed as an $R$-module $M$ along with a homomorphism of rings from $R'$ to the $R$-module endomorphisms of $M$. If $A$ is a $k$-algebra (not necessarily commutative), the category of $A$-modules is understood to be the noncommutative affine scheme corresponding to $A$. We also know that Grothendieck categories often play the role of noncommutative spaces (see, for instance, \cite{LGS}). We want to think about a quasi-coherent sheaf over the  noncommutative base change of the noncommutative affine scheme corresponding to $A$ by a Grothendieck category $\mathfrak S$. This can be understood as an object $\mathcal M\in \mathfrak S$ along with a map of rings from $A$ to the endomorphisms of $\mathcal M$ in $\mathfrak S$.  The idea of this category $\mathfrak S_A$ of  (right) $A$-module objects in $\mathfrak S$ is a classical notion that goes back to Popescu \cite{NP}.  These categories arise naturally in the study of noncommutative projective schemes due to Artin and Zhang \cite{AZ0}. For a graded $k$-algebra $A$, the noncommutative counterpart of the space of homogeneous prime ideals of $A$ is given by either of two categories. The first is the category $Gr(A)$ of graded $A$-modules. The second is the category $QGr(A)$ given by the quotient of $Gr(A)$ over modules which can be obtained as filtered colimits of bounded graded modules (see 
	Artin and Zhang \cite{AZ0}). Both $Gr(A)$ and $QGr(A)$ are Grothendieck categories with excellent properties with respect to base change. For instance, we recall (see  \cite{AZ}) that a Grothendieck category is said to be strongly locally noetherian if the category $\mathfrak S_R$ of $R$-module objects in $\mathfrak S$ is locally noetherian for every noetherian 
	commutative $k$-algebra $R$.  We know (see \cite[$\S$ B5]{AZ}) that the category of modules over a (not necessarily commutative) $k$-algebra $R'$ is strongly locally noetherian if and only if $R'$ is strongly locally noetherian, i.e., $R'\otimes_kR$ is noetherian for every noetherian 
	commutative $k$-algebra $R$ (see Artin, Small and Zhang \cite{ASZ}). Similarly, a graded $k$-algebra $A$ is strongly locally noetherian if and only if the Grothendieck category $Gr(A)$ is strongly locally noetherian (see \cite[$\S$ B5]{AZ}). If $A$ is a graded $k$-algebra and strongly locally noetherian, then the category $QGr(A)$ is also strongly locally noetherian 
	(see \cite[$\S$ B8]{AZ}).
	
	\smallskip
	In \cite{AZ}, Artin and Zhang developed the theory of the 
category $\mathfrak S_A$ to study Hilbert functors on   graded modules over
general noncommutative graded algebras. Further, the authors in \cite{AZ} also established extensions for a number of constructions in commutative algebra  and homological algebra, such as localizations, completions, derived functors of Hom and tensor products for $A$-module objects in $\mathfrak S$. Previously in \cite{BBK}, \cite{BKg},  we have developed counterparts for associated primes and their connection to indecomposable injectives, as well as injective envelopes in the abstract module categories $\mathfrak S_A$. The comodule version of base change by means of Grothendieck categories was introduced by  Brzezi\'nski  and Wisbauer in \cite[$\S$ 39]{BWb}, which gives the category $\mathfrak S^C$ of $C$-comodule objects in $\mathfrak S$, where $C$ is a $k$-coalgebra. In \cite{BBK}, we showed that the category $\mathfrak S^C$ embeds as a coreflective subcategory of the abstract module category $\mathfrak S_A$, where $A$ is a $k$-algebra and $C\otimes A\longrightarrow k$ is a pairing that is ``$\mathfrak S$-rational.'' In \cite{BBK}, we also introduced the category ${_A}\mathfrak S^H$ of relative 
$(A,H)$-Hopf module objects in $\mathfrak S$, where $H$ is a Hopf algebra and $A$ is an $H$-comodule algebra. We then combined this with our work in \cite{BKg} to study injective resolutions, as well as terms appearing in Grothendieck spectral sequences for computing cohomology in  ${_A}\mathfrak S^H$.

\smallskip
As mentioned before, the modules over an entwining structure $(A,C,
\psi)$ incorporate the framework of relative Hopf modules. As such, in this paper, we use two different module like categories to study the base change (with respect to a Grothendieck category $\mathfrak S$) of the noncommutative space corresponding to an entwining structure  $(A,C,
\psi)$. These are the categories of entwined comodule objects and entwined contramodule objects in $\mathfrak S$ over  $(A,C,\psi)$. An object in the former category  $\mathfrak S_A^{C}(\psi)$  consists of $\mathcal M\in \mathfrak S$ that is both a $C$-comodule object in $\mathfrak S$ and an $A$-module object in $\mathfrak S$, such that the two structures are compatible with respect to the entwining $\psi:C\otimes A\longrightarrow A\otimes C$ (see Definition \ref{D2.3} for details). Similarly, an object in the latter category $_{\hspace{1em}A}^{[C,-]}\mathfrak S(\psi)$ consists of $\matholdcal M\in \mathfrak S$  that is both a $C$-contramodule object in $\mathfrak S$ and an $A$-module object in $\mathfrak S$, such that the two structures are compatible with respect to the entwining $\psi$ (see Definition \ref{D3.8} for details). While contramodules were introduced by Eilenberg and Moore alongside comodules in \cite{EM}, they appear to have received insufficient attention in the literature compared to comodules. However, with a revival of interest in contramodules in recent years (see, for instance, \cite{BBR3}, 
\cite{Bazz}, \cite{Pos1}, \cite{Pos2}), one of the key features of our paper is that we study Galois extensions, Frobenius, separability and Maschke type properties in the context of both  entwined contramodules and  entwined comodules. 

\smallskip
We now describe the paper in more detail. For each $k$-vector space $V$, we have a pair of adjoint functors (see Artin and Zhang \cite{AZ})
\begin{equation}\label{Intr1.0}
V\otimes - : \mathfrak S\longrightarrow \mathfrak S\qquad (V,-): \mathfrak S\longrightarrow \mathfrak S
\end{equation} 
We begin in Section 2 by considering the category $\mathfrak S_A^{C}(\psi)$ of entwined comodule objects in $\mathfrak S$ over $(A,C,\psi)$, along with a pair of adjoint functors connecting $\mathfrak S_A^{C}(\psi)$ to the category $\mathfrak S^C$. Using the fact that $\mathfrak S^C$ is a Grothendieck category from our previous work in \cite{BBK}, we show in Theorem \ref{T2.6} that
$\mathfrak S_A^{C}(\psi)$ is also a Grothendieck category. In Section 3, we show that   the category  $_{\hspace{1em}A}^{[C,-]}\mathfrak S(\psi)$ of entwined contramodule objects in 
$\mathfrak S$ over $(A,C,\psi)$ is abelian, and give an explicit description of its generators. As with  ordinary contramodules over a coalgebra, the category  $_{\hspace{1em}A}^{[C,-]}\mathfrak S(\psi)$ of entwined contramodule objects in $\mathfrak S$ over 
$(A,C,\psi)$ is not in general a Grothendieck category, because it does not satisfy the (AB5) condition. 

\smallskip
From Section 4 onwards, we consider morphisms of noncommutative spaces obtained by the base change of entwining structures   with respect to the Grothendieck category
$\mathfrak S$. As with the noncommutative spaces themselves, these morphisms may not exist in an explicit sense, but are understood in terms of adjoint functors between entwined contramodule objects or entwined comodule objects in $\mathfrak S$. We start with Galois extensions. The advantage of this functorial framework is that it allows us to consider maps between noncommutative spaces induced by more general morphisms of entwining structures, known as measurings. If $A$, $A'$ are $k$-algebras, we recall (see Sweedler \cite{Sweed}) that a coalgebra measuring from $A'$ to $A$ consists of a $k$-coalgebra $C'$ and a $k$-linear map $\alpha:C'\longrightarrow Hom_k(A',A)$ such that
\begin{equation}
\alpha(c')(a'b')=\sum \alpha(c'_{(1)})(a') \alpha(c'_{(2)})(b')\qquad \alpha(c')(1_{A'})=\epsilon_{C'}(c')\cdot 1_A \qquad c'\in C,\textrm{ }a',b'\in A'
\end{equation} where the coproduct $\Delta_{C'}(c'):=\sum c'_{(1)}\otimes c'_{(2)}$ and $\epsilon_{C'}$ is the counit on $C'$. We note that when $c'\in C'$ is a grouplike element, i.e., 
$\Delta_{C'}(c')=c'\otimes c'$ and $\epsilon_{C'}(c')=1$, then $\alpha(c'):A'\longrightarrow A$ is an ordinary homomorphism of  rings. The coalgebra measurings lead to an enrichment of the category of $k$-algebras over $k$-coalgebras. This is known as the Sweedler hom of algebras (see \cite{AJ}), which is related to the Sweedler dual of an algebra (see \cite{Porst}). For more on this subject, we refer the reader for instance, to \cite{Bat}, \cite{bim1}, \cite{bim2}, \cite{Vas1}, \cite{Vas2}, \cite{Vas3}. In \cite{BKtga}, \cite{BKvda}, we have used measurings as generalized morphisms between algebras to induce maps between a number of (co)homology theories.  

\smallskip
As introduced by  Brzezi\'nski \cite{Tb}, a measuring of entwining structures from $(A',C',\psi')$ to $(A,C,\psi)$ is given by a pair $(\alpha,\gamma)$ of linear maps 
\begin{equation}
\alpha: C'\otimes A'\longrightarrow A\qquad \gamma: C'\longrightarrow A\otimes C
\end{equation} satisfying certain conditions.  These generalized maps were then used by  Brzezi\'nski in \cite{Tb} to construct adjoint functors between categories of entwined modules, as well as give a notion of Galois measuring of entwining structures. Motivated by this, we show in Section 4 that a measuring of entwining structures  from $(A',C',\psi')$ to $(A,C,\psi)$ leads to a pair of adjoint functors (see Proposition \ref{P4.4})
\begin{equation}\label{Intr1.2}
		\widetilde{Cohom_C}(C',-):~_{\hspace{1em}A}^{[C,-]}\mathfrak S(\psi)\longrightarrow ~_{\hspace{1em}A'}^{[C',-]}\mathfrak S(\psi') \qquad
		\widetilde{Hom_{A'}}(A,-):~_{\hspace{1em}A'}^{[C',-]}\mathfrak S(\psi')\longrightarrow _{\hspace{1em}A}^{[C,-]}\mathfrak S(\psi) 
\end{equation}
between categories of entwined contramodule objects in $\mathfrak S$. When the adjunction in \eqref{Intr1.2} is particularly well behaved, this leads naturally to the notion of an ``$\mathfrak S$-contra-Galois measuring,'' which we introduce in Definition \ref{D4.5}. Similarly, we show that a measuring of entwining structures induces a pair of adjoint functors (see 
Proposition \ref{P4.10})
\begin{equation}\label{Intr1.3}
-\hat{\otimes}_{A'}A:\mathfrak S_{A'}^{C'}(\psi')\longrightarrow\mathfrak S_A^C(\psi) \qquad -\hat{\square}_C C':\mathfrak S_A^C(\psi)\longrightarrow \mathfrak S_{A'}^{C'}(\psi')
\end{equation}
between categories of entwined comodule objects in $\mathfrak S$. Again, when the adjunction in \eqref{Intr1.3} is particularly well behaved, we introduce in Definition \ref{D4.11} the notion of an ``$\mathfrak S$-co-Galois measuring'' of entwining structures  from $(A',C',\psi')$ to $(A,C,\psi)$. Our main results on Galois properties in Section 4 can now be summarized as follows. 

\begin{thmi} (see Theorem \ref{T4.7} and Theorem \ref{T4.13}) Let $\mathfrak S$ be a $k$-linear Grothendieck category.	Let $(\alpha, \gamma)$ be a measuring of entwining structures from $(A',C',\psi')$ to $(A,C,\psi)$. 

\smallskip
(a) Suppose that  the functor $(A',-):\mathfrak S\longrightarrow \mathfrak S$ is exact. For categories of entwined contramodule objects in $\mathfrak S$, the following are equivalent:
		
		\smallskip
		\begin{itemize}
		\item[(i)]  $\widetilde{Cohom_C}(C',-)$ and $\widetilde{Hom_{A'}}(A,-)$ form an adjoint equivalence between the categories  $_{\hspace{1em}A}^{[C,-]}\mathfrak S(\psi)$ and  $_{\hspace{1em}A'}^{[C',-]}\mathfrak S(\psi')$.

	\item[(ii)] $(\alpha, \gamma)$ is an $\mathfrak S$-contra-Galois measuring and the functors $\widetilde{Cohom_C}(C',-)$ and $\widetilde{Hom_{A'}}(A,-)$ are both exact.
		
\item[(iii)]   $(\alpha, \gamma)$ is an $\mathfrak S$-contra-Galois measuring, $\widetilde{Hom_{A'}}(A,-)$ is exact and $\widetilde{Cohom_C}(C',-)$ is faithfully exact.
		
\item[(iv)] $(\alpha, \gamma)$ is an $\mathfrak S$-contra-Galois measuring, $\widetilde{Cohom_C}(C',-)$ is exact and $\widetilde{Hom_{A'}}(A,-)$ is faithfully exact.
\end{itemize}

\smallskip
(b)  For categories of entwined comodule objects in $\mathfrak S$, the following are equivalent:
		
\begin{itemize}
		\item[(i)]The functors $-\hat{\otimes}_{A'}A$ and $-\hat{\square}_C C'$ form an adjoint equivalence of categories between $\mathfrak S_{A}^{C}(\psi)$ and $\mathfrak S_{A'}^{C'}(\psi')$.
		
	\item[(ii)] $(\alpha, \gamma)$ is an $\mathfrak S$-co-Galois measuring and the functors $-\hat{\otimes}_{A'}A$ and $-\hat{\square}_C C'$ are exact.
		
	\item[(iii)]    $(\alpha, \gamma)$ is an $\mathfrak S$-co-Galois measuring, the functor $-\hat{\square}_C C'$ is exact and the functor $-\hat{\otimes}_{A'}A$ is faithfully exact.
		
\item[(iv)] 	 $(\alpha, \gamma)$ is an $\mathfrak S$-co-Galois measuring, the functor  $-\hat{\otimes}_{A'}A$ is exact and the functor $-\hat{\square}_C C'$ is faithfully exact.

\end{itemize}
\end{thmi}

	In the second part of the paper, we study Frobenius, separability and Maschke type properties for entwined contramodule objects and entwined comodule objects in $\mathfrak S$ by adapting the 
	methods of  Brzezi\'nski, Caenepeel, Militaru and Zhu \cite{BCMZ}. As with the noncommutative spaces that are understood in terms of the module categories $_{\hspace{1em}A}^{[C,-]}\mathfrak S(\psi)$ and $\mathfrak S_{A}^{C}(\psi)$, the extensions between these spaces are studied using categorical properties of functors between them. The definitions are inspired by  
those of separable and Frobenius extensions in ring theory. Specifically, if  $R\longrightarrow R'$ is a ring
extension, it is said to be separable if the restriction  of scalars functor  is separable. Additionally,
the extension is called Frobenius if the functors of restriction and extension of scalars form a
Frobenius pair. In the case of entwined contramodule objects in $\mathfrak S$ over $(A,C,\psi)$, one has a pair of adjoint functors 
\begin{equation}^{[C,-]}\mathcal F:~_{\hspace{1em}A}^{[C,-]}\mathfrak S(\psi)\longrightarrow~ ^{[C,-]}\mathfrak S~\qquad\qquad~^{[C,-]}\mathcal T:~^{[C,-]}\mathfrak S\longrightarrow~ _{\hspace{1em}A}^{[C,-]}\mathfrak S(\psi)
\end{equation} where $^{[C,-]}\mathcal F$ is the forgetful functor.  Similarly, for entwined comodule objects in $\mathfrak S$ over $(A,C,\psi)$, one has adjoint functors
\begin{equation}\mathcal T^C:\mathfrak S^C\longrightarrow \mathfrak S_A^C(\psi)~\qquad \qquad~\mathcal F^C:\mathfrak S_A^C(\psi)\longrightarrow \mathfrak S^C
\end{equation} where  $\mathcal F^C$ is  the forgetful functor. In \cite{ENak}, Eilenberg and Nakayama showed   that every separable algebra over a field is Frobenius. Conversely, necessary and sufficient conditions for a Frobenius extension to be separable were established
in \cite[Corollary 4.1]{Kad99}.  Further, in \cite{BCMZ},  Brzezi\'nski, Caenepeel, Militaru and Zhu  explored Frobenius and Maschke type theorems for Doi-Hopf modules and entwined modules
and gave a categorical perspective on the connection between separability and Frobenius properties.
We have explored Frobenius and separability conditions for modules over small preadditive categories
entwined with coalgebras in \cite{BBR2} and \cite{ABtra}. In recent work \cite{BKsm},  we have also studied 
 Frobenius, separability and Maschke type theorems for functors between categories
of entwined comodules and entwined contramodules over coalgebras with several objects.
For more  on separability, Frobenius, and Maschke type theorems, we refer the reader to \cite{CM},  \cite{CM99}, \cite{CMZ0}. For separability and Frobenius properties, our main results in this paper may be summarized as follows.

\begin{thmi} (see Propositions \ref{P5.1}, \ref{P5.4} and Theorems \ref{T5.2}, \ref{T5.5}) Let $(A,C,\psi)$ be an entwining structure  and $\mathfrak S$ be a $k$-linear Grothendieck category. Let $~_{\hspace{1em}A}^{[C,-]}\mathfrak S(\psi)$ be the category of entwined contramodule objects in $\mathfrak S$ over $(A,C,\psi)$. 

\smallskip
(a) The space $V:=Nat(id_{^{[C,-]}\mathfrak S}, ^{[C,-]}\mathcal F~^{[C,-]}\mathcal T)$ is isomorphic to the space $V_1$ of  natural transformations $\sigma:id_{\mathfrak S}\longrightarrow (A,C,-)$ which satisfy \begin{equation}
	(A,\Delta,\matholdcal M)\circ (\psi,C,\matholdcal M)\circ (C,\sigma_{\matholdcal M})= (A,\Delta,\matholdcal M)\circ \sigma_{(C,\matholdcal M)}:(C,\matholdcal M)
	\longrightarrow (A,C,\matholdcal M)
\end{equation}
for each $\matholdcal M\in\mathfrak S$. Then, the following are equivalent:
	\begin{itemize}
	\smallskip
	\item[(i)] The functor $^{[C,-]}\mathcal T:~^{[C,-]}\mathfrak S \longrightarrow~ _{\hspace{1em}A}^{[C,-]}\mathfrak S(\psi)$ is separable. 
	
	\smallskip
	\item[(ii)] There exists $\sigma\in V_1$ such that $(\eta,C,\matholdcal M)\circ \sigma_{\matholdcal M}=(\epsilon,\matholdcal M)$ for each $\matholdcal M\in\mathfrak S$.
	\end{itemize}
	
	\smallskip
	(b)  The space $W:=Nat(^{[C,-]}\mathcal T^{[C,-]}\mathcal F, id_{~ _{\hspace{1em}A}^{[C,-]}\mathfrak S(\psi)})$ is isomorphic to the space $W_1$   of natural transformations $\rho:(A,A,-)\longrightarrow (C,-)$ 
between endofunctors on $\mathfrak S$ which satisfy
\begin{equation}
(A,\rho_{\matholdcal M})\circ (\mu,A,\matholdcal M)=(\psi,\matholdcal M)\circ \rho_{(A,\matholdcal M)}\circ (A,\mu,\matholdcal M)\qquad (\Delta,\matholdcal M)\circ (C,\rho_{\matholdcal M})= (\Delta,\matholdcal M)\circ \rho_{(C,\matholdcal M)}
\circ (A,\psi,\matholdcal M)\circ (\psi,A,\matholdcal M)
\end{equation}  for each $\matholdcal M\in\mathfrak S$. 
Then, the following are equivalent

 \begin{itemize}
 
 \smallskip
	\item[(i)]  The functor $^{[C,-]}\mathcal F:~_{\hspace{1em}A}^{[C,-]}\mathfrak S(\psi)\longrightarrow~ ^{[C,-]}\mathfrak S$ is separable 
	
	\smallskip
	\item[(ii)]  There exists $\rho\in W_1$ such that $\rho_{\matholdcal M}\circ (\mu,\matholdcal M)=(\epsilon,\matholdcal M)\circ (\eta,\matholdcal M)$ for each $\matholdcal M\in\mathfrak S$.
\end{itemize}
\end{thmi}

\begin{thmi} (see  Theorem  \ref{T5.3} and Theorem \ref{Tlst})   Let $(A,C,\psi)$ be an entwining structure  and $\mathfrak S$ be a $k$-linear Grothendieck category. Let $\mathfrak S^C_A(\psi)$ be the category of entwined comodule objects in $\mathfrak S$ over $(A,C,\psi)$. 

 \smallskip
 
 (a) The space $V':=Nat(\mathcal F^C\mathcal T^C, id_{\mathfrak S^C})$ is isomorphic to the space $V_1'$ of natural transformations $\sigma:(-\otimes C\otimes A)\longrightarrow id_{\mathfrak S}$ such that
\begin{equation} 
	(\sigma_{\mathcal M}\otimes C)\circ(\mathcal M\otimes C\otimes \psi)\circ (\mathcal M\otimes \Delta\otimes A)=\sigma_{(\mathcal M\otimes C)}\circ (\mathcal M\otimes \Delta\otimes A) 
\end{equation}
for any $\mathcal M\in \mathfrak S$. Then, the following are equivalent:
 \begin{itemize}
	
	\smallskip
	\item[(i)] The functor $\mathcal T^C:\mathfrak S^C\longrightarrow \mathfrak S^C_A(\psi)$ is separable.
	
	\smallskip
	\item[(ii)] There exists $\sigma\in V_1'$ such that $\sigma_{\mathcal M}\circ (\mathcal M\otimes C\otimes \eta)=\mathcal M\otimes \epsilon $ for each $\mathcal M\in\mathfrak S$.
	\end{itemize}

(b) The space $W':=Nat(id_{\mathfrak S_A^C(\psi)}, \mathcal T^C\mathcal F^C)$ is isomorphic to the space $W_1'$  of natural transformations $\rho:(-\otimes C)\longrightarrow (-\otimes A\otimes A)$ between endofunctors on $\mathfrak S$ such that for each $\mathcal M\in\mathfrak S$ 
\begin{equation}
\begin{array}{c}
(\mathcal M\otimes A\otimes \psi)\circ (\mathcal M\otimes \psi\otimes A) \circ \rho_{(\mathcal M\otimes C)}\circ (\mathcal M\otimes \Delta)=(\rho_{\mathcal M}\otimes C)\circ (\mathcal M\otimes \Delta)\\
(\mathcal M\otimes \mu\otimes A)\circ \rho_{(\mathcal M\otimes A)} \circ (\mathcal M\otimes \psi)= (\mathcal M\otimes A\otimes \mu)\circ (\rho_{\mathcal M}\otimes A)\\
\end{array}
\end{equation}
  Then, the following are equivalent
 \begin{itemize}
\smallskip
	\item[(i)] The functor $\mathcal F^C:\mathfrak S^C_A(\psi)\longrightarrow \mathfrak S^C$ is separable.
	
	\smallskip
	\item[(ii)]  There exists $\rho\in W_1'$ such that $(\mu_{\mathcal M}\otimes A)\circ \rho_{\mathcal M}=(\mathcal M\otimes \eta)\circ (\mathcal M\otimes \epsilon)$ for each  $\mathcal M\in\mathfrak S$.
		
	\end{itemize}
\end{thmi}

\begin{thmi} (see  Theorem  \ref{T6.2} and Theorem \ref{tT6.3}) Let $(A,C,\psi)$ be an entwining structure  and $\mathfrak S$ be a $k$-linear Grothendieck category.

\smallskip
(a) The adjoint pair $(^{[C,-]}\mathcal F, ^{[C,-]}\mathcal T)$ is Frobenius if and only if there exist natural transformations $\sigma\in V_1$ and $\rho\in W_1$ satisfying for  each $\matholdcal M\in\mathfrak S$
\begin{equation}
			 (\epsilon,\matholdcal M)\circ (\eta,\matholdcal M)=(\Delta,\matholdcal M)\circ \rho_{(C,\matholdcal M)}\circ(A,\psi,\matholdcal M)\circ \sigma_{(A,\matholdcal M)} \qquad (\epsilon,\matholdcal M)\circ (\eta,\matholdcal M)=(\Delta,\matholdcal M)\circ \rho_{(C,\matholdcal M)}\circ (A,\sigma_{\matholdcal M})
		\end{equation} 
		
		\smallskip
		(b) The adjoint pair $(\mathcal T^C, \mathcal F^C)$ is Frobenius if and only if there exist natural transformations $\sigma\in V_1'$ and $\rho\in W_1'$ such that for any $\mathcal M\in\mathfrak S$:
		\begin{equation}
			(\mathcal M\otimes \eta)\circ (\mathcal M\otimes \epsilon)=\sigma_{(\mathcal M\otimes A)}\circ (\mathcal M\otimes\psi\otimes A)\circ \rho_{(\mathcal M\otimes C)}\circ (\mathcal M\otimes \Delta) \qquad 
			(\mathcal M\otimes \eta)\circ (\mathcal M\otimes \epsilon)=(\sigma_{\mathcal M}\otimes A)\circ \rho_{(\mathcal M\otimes C)}\circ (\mathcal M\otimes \Delta)
		\end{equation}

\end{thmi}

The well known Maschke Theorem asserts that the group ring of
a finite group is semisimple  when the characteristic of the field does not divide the order
of the group. Over time, this result has been widely extended, leading to various generalizations
for instance to Hopf algebras, or comodule algebras (see \cite{BlMo}, \cite{DoiY}). The Maschke type theorems were studied for Doi-Hopf modules by  Caenepeel, Militaru, and Zhu in \cite{CMZ} and for entwined modules by Brzezi\'nski in \cite{Tb*}.   Our final result in this paper is as follows.

\begin{thmi} (see Theorem \ref{T7.4} and Theorem \ref{tT7.5})
 Let $\mathfrak S$ be a $k$-linear Grothendieck category. Let  $(A,C,\psi)$ be an entwining structure and $\varphi:A^\ast\otimes C\longrightarrow A$ be a normalized cointegral for $(A,C,\psi)$. Then,
		$^{[C,-]}\mathcal F:~_{\hspace{1em}A}^{[C,-]}\mathfrak S(\psi)\longrightarrow~ ^{[C,-]}\mathfrak S$ and 
		$\mathcal F^C:\mathfrak S_A^C(\psi)\longrightarrow~ \mathfrak S^C$ are  semisimple and Maschke functors.
\end{thmi}

	\section{Entwined comodule objects in a Grothendieck category}
	Throughout this section and the rest of this paper, we assume that $k$ is a field of characteristic zero and that $\mathfrak S$ is a $k$-linear Grothendieck category.  Let $(A,\mu,\eta)$ be a $k$-algebra, equipped with multiplication map  $\mu: A\otimes_{k} A\longrightarrow A$ and unit map $\eta:k \longrightarrow A$. Let $(C,\Delta,\epsilon)$ be a $k$-coalgebra, where $\epsilon:C \longrightarrow k$ is the counit and $\Delta:C\longrightarrow C\otimes_{k} C$ is the comultiplication map. 
	
	\smallskip
	Following Popescu  (see \cite[p 108]{NP}), a  right $A$-module object $(\mathcal M,  \rho_{\mathcal M})$ in $\mathfrak S$ consists of an object $\mathcal M\in\mathfrak S$ and a $k$-algebra homomorphism $\rho_{\mathcal M}:A\longrightarrow \mathfrak S(\mathcal M, \mathcal M)$. A morphism between such objects $(\mathcal M,  \rho_{\mathcal M})$, $(\mathcal M',  \rho_{\mathcal M'})$ is a morphism $\phi:\mathcal M\longrightarrow \mathcal M'$ in $\mathfrak S$ that commutes with the actions of $A$ in the obvious manner. We denote by $\mathfrak S_A$ the category of right $A$-module objects 
	in $\mathfrak S$.  We will often denote an object $(\mathcal M,  \rho_{\mathcal M})\in \mathfrak S_A$ simply by $\mathcal M$.  
	
	\smallskip  Let $Mod_k$ denote the category of $k$-vector spaces. By \cite[\S B3]{AZ}, there is a bifunctor $-\otimes_k-:\mathfrak S \times Mod_k\longrightarrow \mathfrak S$ that preserves colimits in both variables. Since $\mathfrak S$ is a Grothendieck category, it follows from the construction (see \cite[$\S$ B3.1]{AZ}) that the bifunctor $-\otimes_k-$ is also exact in both variables. 
	From \cite[\S B3.14]{AZ}, we know that a right $A$-module object $(\mathcal M,  \rho_{\mathcal M})\in \mathfrak S_A$    can also be described by a pair $(\mathcal M, \mu_{\mathcal M})$, where $\mu_{\mathcal M}: \mathcal M \otimes_{k} A \longrightarrow \mathcal M$ is a $k$-linear map such that the composition   $\mathcal M=\mathcal M\otimes_{k} k \xrightarrow{\mathcal M\otimes_{k} \eta} \mathcal M \otimes_{k} A \xrightarrow{\mu_{\mathcal M}} \mathcal M$ is the identity on $\mathcal M$, and the following diagram commutes 
	\[\begin{tikzcd}
		\mathcal M \otimes_{k} A \otimes_{k} A  \arrow{r}{\mathcal M \otimes_{k} \mu} \arrow[swap]{d}{\mu_{\mathcal M} \otimes_{k} A} & \mathcal M \otimes_{k} A \arrow{d}{\mu_{\mathcal M}} \\%
		\mathcal M \otimes_{k} A \arrow{r}{\mu_{\mathcal M}}& \mathcal M
	\end{tikzcd}
	\]
	
	Accordingly, a morphism $\phi:(\mathcal M,\mu_{\mathcal M})\longrightarrow(\mathcal M',\mu_{\mathcal M'})$ of right $A$-module objects is a morphism $\phi:\mathcal M\longrightarrow \mathcal M$ in $\mathfrak S$ that satisfies $\phi\circ \mu_{\mathcal M}=\mu_{\mathcal M'}\circ (\phi\otimes_k A)$.  
	In a similar manner, one has a category $_{A}\mathfrak S$ of left $A$-module objects in $\mathfrak S$ by considering the bifunctor  $-\otimes_k-:Mod_k\times \mathfrak S  \longrightarrow \mathfrak S$. Further, we know from \cite[Proposition B2.2]{AZ} that $\mathfrak S_A$ is a Grothendieck category.
	%For a given $k$-algebra A, let $_{A}Mod$ denote the category of left $A$-modules. 
	%We recall from \cite{AZ} that there exists a bifunctor
	%\begin{equation}\label{}
	%	-\otimes_{A} -: \mathfrak S_A\otimes~ _{A}Mod\longrightarrow \mathfrak S
	%\end{equation}
	%that is right exact in both the variables. We now fix an object $\mathcal M\in\mathfrak S_A$. Then, we know from  \cite[Proposition B3.1]{AZ} that the functor $\mathcal M\otimes_{A}-:~_{A}Mod\longrightarrow \mathfrak S$ is right exact and is left adjoint to the functor $\mathfrak S(\mathcal M, -):\mathfrak S\longrightarrow ~_{A}Mod$. % For any object $\mathcal M\in \mathfrak S_A$, $\mathcal M\otimes_A A=\mathcal M$. 
	%Moreover, given a left $A$-module $V$, it is shown in \cite[Lemma B4.1]{AZ} that the tensor functor $ - \otimes_A V : \mathfrak S_A \longrightarrow \mathfrak S$ is right exact and is the left adjoint to the internal hom functor, $\underline{Hom}_{A}(V, - ) : \mathfrak S \longrightarrow \mathfrak S_A$.
	Now, for a $k$-coalgebra $C$, we recall from \cite[\S 39]{BWb} the definition of right $C$-comodule objects in $\mathfrak S$.
	\begin{defn}\label{D2.1}
		Let $(C,\Delta,\epsilon)$ be a $k$-coalgebra. A right $C$-comodule object in $\mathfrak S$ is a pair $(\mathcal M, \Delta_{\mathcal M}),$  where $\mathcal M$ is an object in $\mathfrak S$ and $\Delta_{\mathcal M} : \mathcal M  \longrightarrow \mathcal M \otimes_{k} C$ is a $k$-linear map such that the composition   $\mathcal M \xrightarrow{\Delta_{\mathcal M}} \mathcal M \otimes_{k} C \xrightarrow{\mathcal M\otimes_{k} \epsilon} \mathcal M\otimes_k k=\mathcal M$ is the identity on $\mathcal M$, and the following diagram commutes
		\[\begin{tikzcd}
			\mathcal M   \arrow{r}{\Delta_{\mathcal M}} \arrow[swap]{d}{\Delta_{\mathcal M}} & \mathcal M \otimes_{k} C \arrow{d}{\mathcal M\otimes_{k} \Delta} \\%
			\mathcal M \otimes_{k} C \arrow{r}{\Delta_{\mathcal M}\otimes_k C}& \mathcal M \otimes_{k} C \otimes_{k} C
		\end{tikzcd}
		\]
	A morphism $\phi: (\mathcal M, \Delta_{\mathcal M})\longrightarrow (\mathcal M', \Delta_{\mathcal M'})$ of right $C$-comodule objects in $\mathfrak S$ is given by a morphism $\phi:\mathcal M\longrightarrow \mathcal M'$ in $\mathfrak S$ such that $\Delta_{\mathcal M'}\circ \phi =(\phi \otimes_{k} C)\circ \Delta_{\mathcal M}$. The category of right $C$-comodule objects
	in  $\mathfrak S$ is denoted by $\mathfrak S^C$. 
	
		\end{defn}
		
		Similarly, one can define the category $^{C}\mathfrak S$ of left $C$-comodule objects in $\mathfrak S$. Further, we know from \cite[Theorem 3.3]{BBK} that $\mathfrak S^C$ is a Grothendieck category.
	Given a $k$ algebra $A$ and a $k$-coalgebra $C$, we now recall the notion of entwining structure from \cite{BM}. From now onwards, we write $\otimes:=\otimes_k$.
	\begin{defn}\label{D2.2}
		An entwining structure is a triple $(A,C,\psi)$, where $A$ is a $k$-algebra, $C$ is a $k$-coalgebra and  $\psi: C\otimes A\longrightarrow A\otimes C$ is a k-linear map such that the following equalities hold:
		\begin{align}
			\psi \circ (C\otimes \mu) &=(\mu\otimes C)\circ (A\otimes \psi)\circ (\psi\otimes A) \qquad \qquad \psi\circ (C\otimes \eta)=\eta\otimes C\label{eq2.1}\\
			(A\otimes \Delta)\circ \psi &= (\psi\otimes C)\circ (C\otimes \psi)\circ (\Delta\otimes A)\qquad \qquad(A\otimes \epsilon)\circ \psi= \epsilon\otimes A.\label{eq2.2}
		\end{align}
		A morphism between entwining structures $(A,C,\psi)$ and $(A',C',\psi')$ is a pair $(f,g)$ where $f:A\longrightarrow A'$ is an algebra map and $g:C\longrightarrow C'$ is a coalgebra map that satisfies $(f\otimes g)\circ \psi=\psi'\circ (g\otimes f)$.
	\end{defn}

	Given an entwining structure $(A,C,\psi)$, we will now define the category of entwined comodule objects in the $k$-linear Grothendieck category $\mathfrak S$. In case $\mathfrak S=Mod_k$ is the category of $k$-vector spaces, this recovers the original definition of entwined modules over $(A,C,\psi)$ given by Brzezi\'nski  (see \cite[\S 2]{Tb}).
	\begin{defn}\label{D2.3}
		Let $(A,C,\psi)$ be an entwining structure and let $\mathfrak S$ be a $k$-linear Grothendieck category. A right entwined comodule object in $\mathfrak S$ over the entwining structure $(A,C,\psi)$  is a triple $(\mathcal M,\Delta_{\mathcal M}, \mu_{\mathcal M})$ such that
		
		\smallskip
		(i) $(\mathcal M,\Delta_{\mathcal M}:\mathcal M\longrightarrow \mathcal M\otimes C)$ is a right $C$-comodule object in $\mathfrak S$
		
		\smallskip
		(ii)  $(\mathcal M,\mu_{\mathcal M}:\mathcal M\otimes A\longrightarrow \mathcal M)$ is a right $A$-module object in $\mathfrak S$
		
		\smallskip
		(iii) The following diagram commutes in $\mathfrak S$
		\begin{equation}\label{Etm}
			\begin{tikzcd}[row sep=1.8em, column sep = 3.8em]
				\mathcal M\otimes A \arrow{r}{\mu_{\mathcal  M}} \arrow{d} {(\Delta_{\mathcal M}\otimes A)}
				&  \mathcal  M \arrow{r}{\Delta_{\mathcal  M}} &(\mathcal M\otimes C)\\
				(\mathcal M\otimes C\otimes A)\arrow{rr}
				{(\mathcal M\otimes \psi)} &&(\mathcal M\otimes A\otimes C)\arrow{u}[swap]{(\mu_{\mathcal M}\otimes C)}
			\end{tikzcd}
		\end{equation}
		A morphism between right entwined comodule objects in $\mathfrak S$ is a map that preserves both the right $A$-module structure and the right $C$-comodule structure.
		We let $\mathfrak S_A^{C}(\psi)$ denote the category of right entwined comodule objects in $\mathfrak S$ over the entwining structure $(A,C,\psi)$.
	\end{defn}

	In Definition \ref{D2.3}, we notice that $\mathcal M\otimes C$ may be treated as an object of $\mathfrak S_A$ by means of the map $\mathcal M\otimes C\otimes A
	\xrightarrow{\mathcal M\otimes \psi}\mathcal M\otimes A\otimes C\xrightarrow{\mu_{\mathcal M}\otimes C}\mathcal M\otimes C$. Similarly, $\mathcal M\otimes A$
	may be treated as an object of $\mathfrak S^C$ by means of the composition $\mathcal M\otimes A\xrightarrow{\Delta_{\mathcal M}\otimes A}\mathcal M\otimes 
	C\otimes A\xrightarrow{\mathcal M\otimes \psi}\mathcal M\otimes A\otimes C$. 
	From (\ref{Etm}), it may be verified  that $\Delta_{\mathcal M}:\mathcal M\longrightarrow \mathcal M\otimes_k C$ is also a morphism in $\mathfrak S_A$ and $\mu_{\mathcal M}:\mathcal M\otimes_k A\longrightarrow \mathcal M$ is also a morphism in $\mathfrak S^C$.
	\begin{lem}\label{L2.4}
		Let $(\mathcal M,\Delta_{\mathcal M})\in \mathfrak S^C$. Then, $\mathcal M\otimes A$ is an object of $ \mathfrak S_A^{C}(\psi)$ via the structure maps $\mu_{\mathcal M\otimes A} = \mathcal M\otimes \mu$ and
		$\Delta_{\mathcal M\otimes A} = (\mathcal M\otimes \psi)\circ(\Delta_{\mathcal M}\otimes A). $
	\end{lem}
	\begin{proof} It is clear that  $\mu_{\mathcal M\otimes A} = \mathcal M\otimes \mu$  makes $\mathcal M\otimes A$ into a right $A$-module object in $\mathfrak S$ and 	
	$\Delta_{\mathcal M\otimes A} = (\mathcal M\otimes \psi)\circ(\Delta_{\mathcal M}\otimes A)$ makes $\mathcal M\otimes A$
	into a right $C$-comodule object in $\mathfrak S$.		From Definition \ref{D2.3}, it remains to show that with the structure maps $\mu_{\mathcal M\otimes A} = \mathcal M\otimes \mu$ and
		$\Delta_{\mathcal M\otimes A} = (\mathcal M\otimes \psi)\circ(\Delta_{\mathcal M}\otimes A)$, the following equality holds
		\begin{equation*}
			\Delta_{\mathcal M\otimes A}\circ \mu_{\mathcal M\otimes A}=(\mu_{\mathcal M\otimes A}\otimes C)\circ (\mathcal M\otimes A\otimes \psi)\circ (\Delta_{\mathcal M\otimes A}\otimes A)
		\end{equation*}
		We observe that $\Delta_{\mathcal M\otimes A}\circ \mu_{\mathcal M\otimes A}=(\mathcal M\otimes \psi)\circ(\Delta_M\otimes A)\circ (\mathcal M\otimes \mu).$ On the other hand, we have
		\begin{align*}
			(\mu_{\mathcal M\otimes A}\otimes C)\circ (\mathcal M\otimes A\otimes \psi)\circ (\Delta_{\mathcal M\otimes A}\otimes A)&=(\mathcal M\otimes \mu\otimes C)\circ (\mathcal M\otimes A\otimes \psi)\circ (\mathcal M\otimes \psi\otimes A)\circ (\Delta_{M}\otimes A\otimes A)\\\notag
			&=\mathcal M\otimes ((\mu\otimes C)\circ (A\otimes \psi)\circ (\psi\otimes A))\circ (\Delta_{M}\otimes A\otimes A)\\\notag
			&=\mathcal M\otimes (\psi\circ (C\otimes \mu))\circ  (\Delta_{M}\otimes A\otimes A)\tag*{\textup{by}~(\ref{eq2.1})}\\\notag
			&=(\mathcal M\otimes \psi)\circ (\mathcal M\otimes C\otimes \mu)\circ (\Delta_{M}\otimes A\otimes A)\\\notag
			&=(\mathcal M\otimes \psi)\circ(\Delta_M\otimes A)\circ (\mathcal M\otimes \mu)
		\end{align*}
		This completes the proof.
	\end{proof}
	By Lemma \ref{L2.4} we  have a functor 
	\begin{equation} \mathcal T^C: \mathfrak S^C\longrightarrow \mathfrak S_A^{C}(\psi) \qquad \mathcal M\mapsto \mathcal M\otimes A 
	\end{equation} 
	Let $\mathcal F^C:\mathfrak S_A^{C}(\psi)\longrightarrow \mathfrak S^C$ denote the forgetful functor. Then, we have the following result.
	\begin{thm}\label{P2.5}
		$(\mathcal T^C,\mathcal F^C)$ is a pair of adjoint functors, i.e., for any $(\mathcal M,\Delta_{\mathcal M})\in \mathfrak S^C$ and $(\mathcal N,\Delta_{\mathcal N},\mu_{\mathcal N}) \in \mathfrak S_A^{C}(\psi)$, we have natural isomorphisms
		\begin{equation*}
			\mathfrak S_A^{C}(\psi)(\mathcal M\otimes A,\mathcal N)\cong \mathfrak S^C(\mathcal M,\mathcal N)
		\end{equation*} 
	\end{thm}
	\begin{proof}
		Let $\zeta:\mathcal M\otimes A\longrightarrow \mathcal N$ be a morphism in $\mathfrak S_A^{C}(\psi)$. Then, $\xi:\mathcal M\otimes k = \mathcal M\xrightarrow{\mathcal M\otimes \eta} \mathcal M\otimes A\xrightarrow{\zeta}\mathcal N$ is the corresponding morphism in $\mathfrak S^C$.  
Conversely, let $\xi:\mathcal M\longrightarrow \mathcal N$ be a morphism in $\mathfrak S^C$. By Lemma \ref{L2.4},  $\xi\otimes A:\mathcal M\otimes A\longrightarrow \mathcal N\otimes A$ is a morphism in $\mathfrak S^C_A(\psi).$ We have noted before that $\mu_{\mathcal N}: \mathcal N\otimes A\longrightarrow \mathcal N$ is also a morphism in $\mathfrak S^C$ and hence in $\mathfrak S^C_A(\psi)$.   It follows that the composition $\zeta=\mu_{\mathcal N}\circ (\xi\otimes A):\mathcal M\otimes A\longrightarrow \mathcal N$ is a morphism in $\mathfrak S^C_A(\psi)$. It may be verified that these two associations are inverse to each other.
	\end{proof}
	\begin{Thm}\label{T2.6}
		Let $\mathfrak S$ be a $k$-linear Grothendieck category and let $(A,C,\psi)$ be an entwining structure. Then, the category $\mathfrak S_A^{C}(\psi)$ of entwined comodule objects in $\mathfrak S$  over $(A,C,\psi)$  is a $k$-linear Grothendieck category.
	\end{Thm}
	\begin{proof}
	Since the bifunctor $-\otimes_k-:\mathfrak S \times Mod_k\longrightarrow \mathfrak S$ preserves colimits and is exact in both variables, we note that
	$\mathfrak S_A^{C}(\psi)$ is an abelian category with filtered colimits and finite limits computed in $\mathfrak S$. 
		Since $\mathfrak S$ is a Grothendieck category, it follows that filtered colimits commute with finite limits in $\mathfrak S$. 
		
		\smallskip
		We now show  that $\mathfrak S_A^{C}(\psi)$ has a generator. Since $\mathfrak S^C$ is a Grothendieck category (see \cite[Theorem 3.3]{BBK}), we fix a generator $\mathcal G$ for $\mathfrak S^C$. Let $\mathcal M\in \mathfrak S_A^{C}(\psi)$. Then, there exists an epimorphism $\phi:\mathcal G^{(I)}\longrightarrow \mathcal M$ in 
		$\mathfrak S^C$ for some indexing set $I$. Since $-\otimes A$ is exact and preserves colimits, it follows from Lemma \ref{L2.4} that  $\tilde{\phi}=\phi\otimes A:\mathcal G^{(I)}\otimes A=(\mathcal G\otimes A)^{(I)}\longrightarrow \mathcal M\otimes A$ is also an epimorphism in $\mathfrak S_A^{C}(\psi)$. Since the composition $\mathcal M\xrightarrow{\mathcal M\otimes \eta}\mathcal M\otimes A \xrightarrow{
		\mu_{\mathcal M}}\mathcal M$ is the identity in $\mathfrak S$, we note that $\mu_{\mathcal M}:\mathcal M\otimes A\longrightarrow \mathcal M$ is an epimorphism in 
		$\mathfrak S$. Since $\mathcal M\in \mathfrak S_A^{C}(\psi)$, we have noted before that $\mu_{\mathcal M}$ is also a morphism in $\mathfrak S^C$. Clearly, $\mu_{\mathcal M}$ is also right $A$-linear. It follows that $\mu_{\mathcal M}$ is a morphism in  $\mathfrak S^C_A(\psi)$.  Since cokernels in $\mathfrak S^C_A(\psi)$ are computed in $\mathfrak S$, it follows that 
		$\mu_{\mathcal M}$ is an epimorphism in  $\mathfrak S^C_A(\psi)$.  Hence, the composition $\hat{\phi}:(\mathcal G\otimes A)^{(I)}\xrightarrow{\tilde{\phi}} \mathcal M\otimes A\xrightarrow{\mu_{\mathcal M}}\mathcal M$ is an epimorphism in $\mathfrak S^C_A(\psi)$, and  we see that $\mathcal G \otimes A$ is a generator for $\mathfrak S_A^{C}(\psi)$.  
	\end{proof}
	%%%%%%%%%%%%%%%%%%%%%%%%%%%%%%%%%%%%%%%%%%%%%%%%%%%%%%%%%%%%%%%%%%%%%%%%%%%%%%%%%%%%%%%%%%%%%%%%%%%%%%%%%%%%%%%%%%%%%%%%%%%%%%%%%%%%%%%%%%%%%%%%%%%%%%%%%%%%%%%%%%%%%%%%%%%%%%%%%%%%%%%%%%%%%%%%%%%%%%%%%%%%%%%%%%%%%%%%%%%%%%%%%%%%%%%%%%%%%%%%%%%%%%%%%%%%%%%%%%%%%%%%%%%%%%%%%%%%%%%%%%%%%
	
	\section{Entwined contramodule objects  in a Grothendieck category}
	We continue with a $k$-algebra $A$, a $k$-coalgebra $C$ and a $k$-linear Grothendieck category $\mathfrak S$ as before. The contramodules over an entwining structure  $(A,C,\psi)$ were defined by Positselski \cite[$\S$ 10]{Pos1}. In this section,  we study the base change of entwined contramodules with respect to a Grothendieck category $\mathfrak S$.  By \cite[$\S$ B4.1]{AZ}, we know that for any $V\in Mod_k$, the functor $-\otimes V: \mathfrak S\longrightarrow \mathfrak S$ has a right adjoint which we denote by $(V,-):\mathfrak S\longrightarrow \mathfrak S$. In fact, this determines a bifunctor 
	$(-,-): Mod_k^{op}\times \mathfrak S\longrightarrow \mathfrak S$. For $V$, $V'\in Mod_k$, we note that there is a canonical isomorphism $(V,(V',-))\cong (V'\otimes V,-)$ of functors. 
	\begin{defn}\label{D3.1}
		Let $(C,\Delta,\epsilon)$ be a $k$-coalgebra and $\mathfrak S$ be a $k$-linear Grothendieck category. A left $C$-contramodule object in $\mathfrak S$ is a pair $(\matholdcal M, \pi_{\matholdcal M})$, where $\matholdcal M$ is an object in $\mathfrak S$ and $\pi_{\matholdcal M}: (C,\matholdcal M)\longrightarrow \matholdcal M$ is a $k$-linear map such that the composition  $\matholdcal M \xrightarrow{(\epsilon,\mathcal M)} (C,\matholdcal M)  \xrightarrow{\pi_{\matholdcal M}} \matholdcal M$ is the identity on $\matholdcal M$, and the following diagram commutes
		\begin{equation}\label{eq3.1}
			\begin{tikzcd}
				(C,(C,\matholdcal M))\cong (C\otimes C,\matholdcal M) \arrow{r}{(\Delta,\matholdcal M)} \arrow[swap]{d}{(C,\pi_{\matholdcal M})} & (C,\matholdcal M) \arrow{d}{\pi_{\matholdcal M}} \\%
				(C,\matholdcal M)\arrow{r}{\pi_{\matholdcal M}}& \matholdcal M
			\end{tikzcd}
		\end{equation}
		A morphism $\phi:(\matholdcal M,\pi_{\matholdcal M})\longrightarrow (\matholdcal M',\pi_{\matholdcal M'})$ of left $C$-contramodules in $\mathfrak S$ consists of a morphism $\phi:\matholdcal M\longrightarrow \matholdcal M'$ in $\mathfrak S$ such $\pi_{\matholdcal M'}\circ (C,\phi)=\phi\circ \pi_{\matholdcal M}$. We denote the category of  left $C$-contramodule objects in $\mathfrak S$ by $^{[C,-]}\mathfrak S$. 
	\end{defn}
	 
	Since  $(C,-)$ is a right adjoint (and hence it preserves limits), it follows that limits exist in $^{[C,-]}\mathfrak S$ and are computed in $\mathfrak S$. If  $(C,-):\mathfrak S\longrightarrow \mathfrak S$ is exact,   then cokernels also exists in $^{[C,-]}\mathfrak S$ and are computed in $\mathfrak S$, which makes $^{[C,-]}\mathfrak S$ an abelian category. Using \cite[$\S$ B4]{AZ}, we note that this happens for instance   if $\mathfrak S$ satisfies the (AB4*) condition, i.e., products in $\mathfrak S$ are exact.   
	
	\smallskip
	From now onwards, we will assume that  $(C,-):\mathfrak S\longrightarrow \mathfrak S$ is exact, unless otherwise mentioned. For any object $\matholdcal M\in \mathfrak S$, we observe that $(C, \matholdcal M)$ is a left $C$-contramodule object in $\mathfrak S$, where the structure map $\pi_{\matholdcal M}$ is given by
	\begin{equation}\label{eq3.2}
		\pi_{\matholdcal  M}:(C, (C, \matholdcal M))\cong (C\otimes C, \matholdcal M)\xrightarrow{(\Delta,\matholdcal M)} (C, \matholdcal M)
	\end{equation} 
	This determines a functor $^{[C,-]}\mathcal T:\mathfrak S\longrightarrow~^{[C,-]}\mathfrak S$ that associates $\matholdcal M\mapsto (C, \matholdcal M)$.  Let $^{[C,-]}\mathcal F:~^{[C,-]}\mathfrak S\longrightarrow \mathfrak S$ denote the forgetful functor. We now have the following result.
	\begin{thm}\label{P3.2}
		Let $C$ be a $k$-coalgebra and $\mathfrak S$ be a $k$-linear Grothendieck category. Then, $(^{[C,-]}\mathcal T,^{[C,-]}\mathcal F)$ is a pair of adjoint functors, i.e., for any $\matholdcal M\in \mathfrak S$ and $(\matholdcal N,\pi_{\matholdcal N}) \in~^{[C,-]}\mathfrak S$, we have natural isomorphisms
		\begin{equation*}
			^{[C,-]}\mathfrak S((C, \matholdcal M),\matholdcal N)\cong \mathfrak S(\matholdcal M,\matholdcal N)
		\end{equation*} 
	\end{thm} 
	\begin{proof}
		Let  $\zeta:\matholdcal M\longrightarrow \matholdcal N$ be a morphism in $\mathfrak S$. Then, the composition $(C, \matholdcal M)\xrightarrow{(C, \zeta)} (C, \matholdcal N) \xrightarrow{\pi_{\matholdcal N}} \matholdcal N$ is the corresponding morphism in $^{[C,-]}\mathfrak S$. Conversely, let $\xi:(C,\matholdcal M)\longrightarrow \matholdcal N$ be a morphism in $^{[C,-]}\mathfrak S$. Then, the composition $\matholdcal M\cong (k,\matholdcal M)\xrightarrow{(\epsilon, \matholdcal M)} (C,\matholdcal M)\xrightarrow{\xi} \matholdcal N$ is the corresponding morphism in $\mathfrak S$. It may be easily seen that these associations are inverse to each other. 
	\end{proof}
	\begin{lem}\label{L3.3}
		Let $C$ be a $k$-coalgebra and let $\mathfrak S$ be a $k$-linear Grothendieck category . Then, we have the following:
		
		\smallskip
		(i) If $\matholdcal G$ is a generator for $\mathfrak S$, then  $(C,\matholdcal G$) is a generator for $^{[C,-]}\mathfrak S$.
		
		\smallskip
		(ii) If $\matholdcal P$ is projective in $\mathfrak S$, then $(C,\matholdcal P$) is projective in $^{[C,-]}\mathfrak S$.
		
		\smallskip
		(iii) If $\matholdcal P$ is a projective generator for $\mathfrak S$, then $(C,\matholdcal P$) is a projective generator for $^{[C,-]}\mathfrak S$.
		
	\end{lem}
	\begin{proof}
		
		(i) We consider a monomorphism $\matholdcal M'\hookrightarrow \matholdcal M$ in $^{[C,-]}\mathfrak S$ that is not an isomorphism in $^{[C,-]}\mathfrak S$. Since kernels and cokernels in $^{[C,-]}\mathfrak S$ are computed in $\mathfrak S$, we see that $\matholdcal M'\hookrightarrow \matholdcal M$ is a monomorphism in $\mathfrak S$ that is not an isomorphism.  Since $\matholdcal G$ is a generator for $\mathfrak S$, by \cite[\S 1.9]{Gro}, there exists a morphism $\phi:\matholdcal G\longrightarrow\matholdcal M$ in $\mathfrak S$ that does not factor through $\matholdcal M'$. By the adjunction in Proposition \ref{P3.2}, we obtain a  morphism $\tilde{\phi}: (C,\matholdcal G)\longrightarrow \matholdcal M$ in $^{[C,-]}\mathfrak S$ corresponding to $\phi$ that does not factor through $\matholdcal M'$. It now follows from \cite[\S 1.9]{Gro} that $(C,\matholdcal G)$ is a generator for $^{[C,-]}\mathfrak S$.
		
		\smallskip
		
		(ii) Since $\matholdcal P$ is projective in $\mathfrak S$, $^{[C,-]}\mathfrak S((C, \matholdcal P),-)\cong  \mathfrak S(\matholdcal P,-)$ is an exact functor. Therefore, $(C,\matholdcal P)$ is projective in $^{[C,-]}\mathfrak S$.
		
		\smallskip
		(iii) This is clear from (i) and (ii).
	\end{proof}

\smallskip
	  We now introduce entwined contramodule objects in $\mathfrak S$ over an entwining structure $(A,C,\psi)$. In the context of entwined contramodules, we will typically use 
$\mu_{\matholdcal M}:\matholdcal M\longrightarrow (A,\matholdcal M)$  to denote the structure map of an object $(\matholdcal M,\mu_{\matholdcal M})\in {_A}\mathfrak S$. 	   

	\begin{defn}\label{D3.8}
		Let $(A,C,\psi)$ be an entwining structure and $\mathfrak S$ be a $k$-linear Grothendieck category. A left entwined contramodule object in $\mathfrak S$ over the entwining structure $(A,C,\psi)$ is a triple $(\matholdcal M,\pi_{\matholdcal M}, \mu_{\matholdcal M})$ such that 
		
		\smallskip
		(i) $(\matholdcal M,\pi_{\matholdcal M}:(C,\matholdcal M)\longrightarrow \matholdcal M)$ is a left $C$-contramodule object in $\mathfrak S$ 
		
		\smallskip
		(ii) $(\matholdcal M,\mu_{\matholdcal M}:\matholdcal M\longrightarrow (A,\matholdcal M))$ is a left $A$-module object in $\mathfrak S$
		
		\smallskip
		(iii) The following diagram commutes in $\mathfrak S$
		\begin{equation}\label{eq3.9}
			\begin{tikzcd}[row sep=1.8em, column sep = 3.8em]
				(C,\matholdcal M) \arrow{r}{\pi_{\matholdcal M}} \arrow{d} {(C,\mu_{\matholdcal M})}
				&  \matholdcal M \arrow{r}{\mu_{\matholdcal M}} &(A,\matholdcal M)\\
				(C,(A,\matholdcal M))\cong (A\otimes C,\matholdcal M) \arrow{rr}
				{(\psi,\matholdcal M)} &&(C\otimes A,\matholdcal M)\cong (A,(C,\matholdcal M))\arrow{u}[swap]{(A,\pi_{\matholdcal M})}
			\end{tikzcd}
		\end{equation}
		A morphism between left entwined contramodule objects in $\mathfrak S$ is a map that preserves both the left $A$-module structure and the left $C$-contramodule structure. 
		We denote by $_{\hspace{1em}A}^{[C,-]}\mathfrak S(\psi)$ 
		 the category of left entwined contramodule objects in $\mathfrak S$ over the entwining structure $(A,C,\psi)$.
	\end{defn}

	\smallskip
	From now onwards, for any sequence  $V_1,V_2,\ldots, V_n$   of $k$-vector spaces   and any object $\matholdcal M\in \mathfrak S$, we will write 	
	\begin{equation}\label{eq3.10}
		(V_1,V_2,\ldots, V_n,\matholdcal M):=(V_n\otimes V_{n-1}\otimes ...\otimes V_1,\matholdcal M)=(V_1,(V_2,(.\ldots(V_n,\matholdcal M))))\in \mathfrak S
	\end{equation}
	We will also use this to denote all the equivalent forms in which the right hand side of \eqref{eq3.10} can be written using the hom-tensor adjunction. Then for an entwining structure $(A,C,\psi)$, we observe that   applying the functor $(-,\matholdcal M)$ to (\ref{eq2.1}) gives the following commutative diagrams
	\begin{equation}\label{eq3.11}
		\begin{array}{lll}
			\begin{tikzcd}[row sep=1.8em, column sep = 3.8em]
				(C, A, \matholdcal M)  \arrow{r}{(\psi,\matholdcal M)} \arrow{d}{(C,\mu, \matholdcal M)}& (A, C, \matholdcal M) \arrow{r} {(\mu, C, \matholdcal M)} &(A, A, C, \matholdcal M) \\%
				(C, A, A,\matholdcal M)   \arrow{rr}{(\psi,A,\matholdcal M)}&  &(A, C, A, \matholdcal M)\arrow{u}[swap] {(A,\psi,\matholdcal M)}
			\end{tikzcd}
			\qquad\qquad
			\begin{tikzcd}
				(C,A,\matholdcal M)\arrow{r}{(\psi,\matholdcal M)}\arrow{dr}[swap]{(C,\eta,\matholdcal M)}&(A,C,\matholdcal M)\arrow{d}{(\eta,C,\matholdcal M)}\\
				&(C,\matholdcal M)
			\end{tikzcd}
		\end{array}
	\end{equation}
	Similarly, applying the functor $(-,\matholdcal M)$ to (\ref{eq2.2}) gives the following commutative diagrams 
	\begin{equation}\label{eq3.12}
		\begin{array}{lll}
			\begin{tikzcd}[row sep=1.8em, column sep = 3.8em]
				(C, C, A, \matholdcal M)  \arrow{r}{(\Delta, A, \matholdcal M)} \arrow{d}{(C,\psi, \matholdcal M)}& (C, A,\matholdcal M) \arrow{r} {(\psi,\matholdcal M)} &(A, C, \matholdcal M) \\%
				(C, A, C, \matholdcal M)   \arrow{rr}{(\psi,C, \matholdcal M)}&  &(A, C, C, \matholdcal M)\arrow{u}[swap] {(A,\Delta,\matholdcal M)}
			\end{tikzcd}
			\qquad\qquad
			\begin{tikzcd}
				(A,\matholdcal M)\arrow{r}{(\epsilon,A,\matholdcal M)}\arrow{dr}[swap]{(A,\epsilon,\matholdcal M)}&(C,A,\matholdcal M)\arrow{d}{(\psi,\matholdcal M)}\\
				&(A,C,\matholdcal M)
			\end{tikzcd}
		\end{array}
	\end{equation}
	\begin{lem}\label{L3.9}  
		Let $(\matholdcal M,\pi_{\matholdcal M})$ be an object in $~^{[C,-]}\mathfrak S$. Then, $(A,\matholdcal M)\in ~_{\hspace{1em}A}^{[C,-]}\mathfrak S(\psi)$. This   determines a functor $~^{[C,-]}\mathcal T:~^{[C,-]}\mathfrak S\longrightarrow ~_{\hspace{1em}A}^{[C,-]}\mathfrak S(\psi)$.
	\end{lem}	
	\begin{proof}
		Clearly, $(A,\matholdcal M)$ is a left $A$-module object in $\mathfrak S$ with  structure map $\mu_{(A,\matholdcal M)}:(A,\matholdcal M)\xrightarrow{(\mu, \matholdcal M)}(A,A,\matholdcal M)$. The composition $\pi_{(A,\matholdcal M)}:(C, A,\matholdcal M)\xrightarrow{(\psi, \matholdcal M)} (A,C, \matholdcal M)\xrightarrow{(A,\pi_{\matholdcal M})} (A,\matholdcal M)$ makes $(A,\matholdcal M)$ into a   left $C$-contramodule object in $\mathfrak S$. 
	To show that $(A,\matholdcal M)$ satisfies the commutative diagram in (\ref{eq3.9}), we verify that
		\begin{align*}
			(A,\pi_{(A,\matholdcal M)})\circ (\psi, A,\matholdcal M)\circ (C,\mu_{(A,\matholdcal M)}) 
			&=(A,A,\pi_{\matholdcal M})\circ (A,\psi,\matholdcal M)\circ (\psi, A,\matholdcal M)\circ (C,\mu,\matholdcal M)\\\notag
			&=(A,A,\pi_{\matholdcal M})\circ (\mu,C,\matholdcal M)\circ (\psi,\matholdcal M)\tag*{(\textup{by}~(\ref{eq3.11}))}\\\notag
			&=(\mu,\matholdcal M)\circ (A,\pi_{\matholdcal M})\circ (\psi,\matholdcal M)=\mu_{(A,\matholdcal M)}\circ \pi_{(A,\matholdcal M)} 
		\end{align*}
		This proves the result. 
	\end{proof}	
	\begin{thm}\label{L3.10} The functor $~^{[C,-]}\mathcal T:~^{[C,-]}\mathfrak S\longrightarrow ~_{\hspace{1em}A}^{[C,-]}\mathfrak S(\psi)$ is right adjoint to the forgetful 
	functor 
 $^{[C,-]}\mathcal F:~_{\hspace{1em}A}^{[C,-]}\mathfrak S(\psi)\longrightarrow~^{[C,-]}\mathfrak S$, i.e., for all $(\matholdcal M,\pi_{\matholdcal M},\mu_{\matholdcal M})\in~_{\hspace{1em}A}^{[C,-]}\mathfrak S(\psi)$ and $(\matholdcal N,\pi_{\matholdcal N})\in~^{[C,-]}\mathfrak S$, we have natural isomorphisms
		\begin{equation*}
			^{[C,-]}\mathfrak S(\matholdcal M,\matholdcal N)\cong~ ^{[C,-]}_{\hspace{1em}A}\mathfrak S(\psi)(\matholdcal M,(A,\matholdcal N))
		\end{equation*}
	\end{thm}
	\begin{proof}
		Let $\xi:\matholdcal M\longrightarrow (A,\matholdcal N)$ be a morphism in $^{[C,-]}_{\hspace{1em}A}\mathfrak S(\psi)$. We consider the composition $\zeta:\matholdcal M\xrightarrow{\xi}(A,\matholdcal N)\xrightarrow{(\eta,\matholdcal N)}\matholdcal N$. We now show that $\zeta$ is a morphism in $^{[C,-]}\mathfrak S$, i.e., we will show that $\pi_{\matholdcal N}\circ (C,\zeta)=\zeta\circ \pi_{\matholdcal M}.$ We have 
		\begin{align*}
			\zeta\circ \pi_{\matholdcal M}= (\eta,\matholdcal N)\circ \xi\circ \pi_{\matholdcal M}=(\eta,\matholdcal N)\circ \pi_{(A,\matholdcal N)}\circ (C,\xi)
			&=(\eta,\matholdcal N)\circ (A, \pi_{\matholdcal N})\circ(\psi, \matholdcal N)\circ (C,\xi)\\\notag
			&=\pi_{\matholdcal N}\circ (\eta, C,\matholdcal N)\circ (\psi, \matholdcal N)\circ (C,\xi)\\\notag
			&=\pi_{\matholdcal N}\circ (C,\eta,\matholdcal N)\circ (C,\xi)\tag*{(\textup{by}~(\ref{eq3.11}))}\\\notag
			&=\pi_{\matholdcal N}\circ (C,((\eta,\matholdcal N)\circ\xi) )= \pi_{\matholdcal N}\circ (C,\zeta)
		\end{align*}
		For the converse, we consider a morphism $\zeta:\matholdcal M\longrightarrow \matholdcal N$ in $^{[C,-]}\mathfrak S$. Then, we take the composition $\xi: \matholdcal M\xrightarrow{\mu_{\matholdcal M}} (A,\matholdcal M)\xrightarrow{(A,\zeta)} (A,\matholdcal N)$. By the functor  $~^{[C,-]}\mathcal T:~^{[C,-]}\mathfrak S\longrightarrow ~_{\hspace{1em}A}^{[C,-]}\mathfrak S(\psi)$, we see that $(A,\zeta)$ must be a morphism in $^{[C,-]}_{\hspace{1em}A}\mathfrak S(\psi)$. We know that $\mu_{\matholdcal M}$ is left 
		$A$-linear. From the diagram in (\ref{eq3.9}), we can verify that $\mu_{\matholdcal M}$ is also a morphism in $^{[C,-]}\mathfrak S$. Hence, the composition $\xi=(A,\zeta)\circ \mu_{\matholdcal M}$ is a morphism in $^{[C,-]}_{\hspace{1em}A}\mathfrak S(\psi)\big(\matholdcal M,(A,\matholdcal N)\big)$. It may be verified that  these two associations are inverse to each other.
		%				We now show that $\xi$ is also a morphism in $^{[C,-]}\mathfrak S$. For this we shall show that $\pi_{(A,\matholdcal N)}\circ (C,\xi)=\xi\circ \pi_{\matholdcal M}$.
		%				We see that 
		%				\begin{align*}
			%					\pi_{(A,\matholdcal N)}\circ (C,\xi)&=(A,\pi_{\matholdcal N})\circ (\psi,\matholdcal N)\circ (C,A,\zeta)\circ (C, \mu_{\matholdcal M})\\\notag
			%					&=(A,\pi_{\matholdcal N})\circ (A, C, \zeta)\circ (\psi,\matholdcal M)\circ (C, \mu_{\matholdcal M})\\\notag
			%					&=(A,\zeta)\circ (A,\pi_{\matholdcal M})\circ (\psi,\matholdcal M)\circ (C, \mu_{\matholdcal M})\\\notag
			%					&=(A,\zeta)\circ  \mu_{\matholdcal M}\circ \pi_{\matholdcal M}=\xi\circ \pi_{\matholdcal M}
			%				\end{align*}
	\end{proof}
	\begin{lem}\label{L3.11}
		Let $(\matholdcal M,\mu_{\matholdcal M})\in~_A\mathfrak S$. Then, $(C,\matholdcal M)\in ~_{\hspace{1em}A}^{[C,-]}\mathfrak S(\psi)$ with the structure maps $\pi_{(C, \matholdcal M)}:(C,C,\matholdcal M)\xrightarrow{(\Delta,\matholdcal M)}(C,\matholdcal M)$ and $	\mu_{(C,\matholdcal M)}:(C,\matholdcal M)\xrightarrow{(C,\mu_{\matholdcal M})} (C, A, \matholdcal M)\xrightarrow{(\psi, \matholdcal M)} (A,C,\matholdcal M)$. 
	\end{lem}
	\begin{proof}
		We can verify directly that  $\pi_{(C, \matholdcal M)}$ makes $(C,\matholdcal M)$ an object of $^{[C,-]}\mathfrak S$   and that $	\mu_{(C,\matholdcal M)}$ makes 
		$(C,\matholdcal M)$ an object of $_A\mathfrak S$. To show that $(C,\matholdcal M)$ satisfies the commutative diagram in (\ref{eq3.9}), we note that 
		\begin{align*}
			(A,\pi_{(C,\matholdcal M)})\circ (\psi, (C,\matholdcal M))\circ ( C,\mu_{(C,\matholdcal M)})
			&=(A, \Delta,\matholdcal M)\circ (\psi, C,\matholdcal M)\circ (C, \psi,\matholdcal M)\circ (C, C, \mu_{\matholdcal M})\\\notag
			&=(\psi, \matholdcal M)\circ (\Delta, A, \matholdcal M)\circ (C,C,\mu_{\matholdcal M}) \qquad {(\textup{by}~(\ref{eq3.12}))}\\\notag
			&=(\psi, \matholdcal M)\circ (\Delta, A, \matholdcal M)\circ (C\otimes C,\mu_{\matholdcal M}) \\ \notag &=(\psi, \matholdcal M)\circ (C, \mu_{\matholdcal M})\circ (\Delta,\matholdcal M) =\mu_{(C,\matholdcal M)} \circ \pi_{(C,\matholdcal M)} 
		\end{align*}
	\end{proof}	
	By Lemma \ref{L3.11}, we have a functor $_A\mathcal T:~_A\mathfrak S\longrightarrow ~_{\hspace{1em}A}^{[C,-]}\mathfrak S(\psi)$ that takes $\matholdcal M\mapsto (C,\matholdcal M)$ for each $(\matholdcal M,\mu_{\matholdcal M})\in  {_A}\mathfrak S$. Let $_A\mathcal F:~_{\hspace{1em}A}^{[C,-]}\mathfrak S(\psi)\longrightarrow~_A\mathfrak S$ denote the forgetful functor.
	\begin{thm}\label{L3.12}
		Let $\mathfrak S$ be a $k$-linear Grothendieck category. Then, the functors $(_A\mathcal T,_A\mathcal F)$ form an adjoint pair, i.e., for all $(\matholdcal M,\mu_{\matholdcal M})\in~_A\mathfrak S$ and $(\matholdcal N,\pi_{\matholdcal N},\mu_{\matholdcal N})\in~_{\hspace{1em}A}^{[C,-]}\mathfrak S(\psi)$, we have natural isomorphisms
		\begin{equation*}
			^{[C,-]}_{\hspace{1em}A}\mathfrak S(\psi)((C,\matholdcal M),\matholdcal N)\cong~_A\mathfrak S(\matholdcal M,\matholdcal N)
		\end{equation*}
	\end{thm}
	\begin{proof}
		Let $\zeta:(C,\matholdcal M)\longrightarrow \matholdcal N$ be a morphism in $^{[C,-]}_{\hspace{1em}A}\mathfrak S(\psi)$. We consider the composition $\xi:\matholdcal M\cong (k,\matholdcal M)\xrightarrow{(\epsilon,\matholdcal M)}  (C,\matholdcal M)\xrightarrow{\zeta} \matholdcal N$.
		To verify that $\xi$ is a morphism in $_A\mathfrak S$, we note that
		\begin{align*}
			\mu_{\matholdcal N}\circ \xi = \mu_{\matholdcal N}\circ \zeta\circ (\epsilon,\matholdcal M)
			=(A,\zeta)\circ \mu_{(C,\matholdcal M)}\circ (\epsilon,\matholdcal M)
			&=(A,\zeta)\circ (\psi,\matholdcal M)\circ (C,\mu_{\matholdcal M})\circ (\epsilon,\matholdcal M)\\\notag
			&=(A,\zeta)\circ (\psi,\matholdcal M)\circ (\epsilon,A,\matholdcal M)\circ \mu_{\matholdcal M}\\\notag
			&=(A,\zeta)\circ (A, \epsilon,\matholdcal M)\circ \mu_{\matholdcal M}\tag*{(\textup{by}~(\ref{eq3.12}))}\\\notag
			&=(A, \zeta\circ ( \epsilon,\matholdcal M))\circ \mu_{\matholdcal M}=(A,\xi)\circ \mu_{\matholdcal M}
		\end{align*}
		Conversely, let $\xi:\matholdcal M\longrightarrow \matholdcal N$ be a morphism in $_A\mathfrak S$. Then, we consider the following composition
		\begin{equation*}
			\zeta:(C,\matholdcal M)\xrightarrow{(C,\xi)}(C,\matholdcal N)\xrightarrow{\pi_{\matholdcal N}} \matholdcal N
		\end{equation*}
		Then $(C,\xi)={_A}\mathcal T(\xi)$ is a morphism in 	$^{[C,-]}_{\hspace{1em}A}\mathfrak S(\psi)$. Since  $(\matholdcal N,\pi_{\matholdcal N},\mu_{\matholdcal N})\in~_{\hspace{1em}A}^{[C,-]}\mathfrak S(\psi)$, we know that $\pi_{\matholdcal N}$ is a morphism in both $_A\mathfrak S$ and $^{[C,-]}\mathfrak S$. Hence,  $\zeta=\pi_{\matholdcal N}\circ (C,\xi)$ lies in $^{[C,-]}_{\hspace{1em}A}\mathfrak S(\psi)$. It may be verified  that these two associations are inverse to each other.
		%					It is clear that $\zeta$ is a morphism in $^{[C,-]}\mathfrak S$. We now verify that $\zeta$ is also a morphism in $_A\mathfrak S$ i.e., we will show that $\mu_{\matholdcal N}\circ \zeta=(A,\zeta)\circ \mu_{(C,\matholdcal M)}$, where $\mu_{(C,\matholdcal M)}:(C,\matholdcal M)\xrightarrow{(C,\mu_{\matholdcal M})} (C,A,\matholdcal M)\xrightarrow{(\psi,\matholdcal M)}(A,C,\matholdcal M)$ is the left $A$-module map on $(C,\matholdcal M)$. Consider
		%					\begin{align*}
			%						(A,\zeta)\circ \mu_{(C,\matholdcal M)}&=(A,\zeta)\circ (\psi,\matholdcal M)\circ (C,\mu_{\matholdcal M})\\\notag
			%						&=(A,\pi_{\matholdcal N})\circ (A,C,\xi)\circ (\psi, \matholdcal M)\circ (C,\mu_{\matholdcal M})\\\notag
			%						&=(A,\pi_{\matholdcal N})\circ (\psi,\matholdcal N)\circ (C,A,\xi)\circ (C,\mu_{\matholdcal M})\\\notag
			%							&=(A,\pi_{\matholdcal N})\circ (\psi,\matholdcal N)\circ (C,\mu_{\matholdcal N})\circ (C,\xi)\\\notag
			%							&=\mu_{\matholdcal N}\circ \pi_{\matholdcal N}\circ (C,\xi)=\mu_{\matholdcal N}\circ \zeta~.
			%						\end{align*}
		
	\end{proof}
	\begin{Thm}\label{T3.13}
		Let $\mathfrak S$ be a $k$-linear Grothendieck category  and let $(A,C,\psi)$ be an entwining structure. Suppose that the functor $(C,-):\mathfrak S
		\longrightarrow \mathfrak S$ is exact. Then $_{\hspace{1em}A}^{[C,-]}\mathfrak S(\psi)$ is an abelian category. 
	\end{Thm}
	\begin{proof}
		We know that any morphism $\phi:\matholdcal M\longrightarrow \matholdcal N$ in $_{\hspace{1em}A}^{[C,-]}\mathfrak S(\psi)$ is also a morphism in $_A\mathfrak S$. Both $ker(\phi)$ and $coker(\phi)$ exist in $_A\mathfrak S$ and are computed in $\mathfrak S$. Since $(C,-)$ is  exact, there are induced morphisms $\pi_{ker(\phi)}:(C,ker(\phi))\longrightarrow ker(\phi)$ and $\pi_{coker(\phi)}:(C,coker(\phi))\longrightarrow coker(\phi)$ making $ker(\phi)$ and $coker(\phi)$ respectively into   $C$-contramodule objects in $\mathfrak S$. Since the right adjoint $(A,-):\mathfrak S\longrightarrow \mathfrak S$ preserves monomorphisms, it follows that $ker(\phi)$ satisfies the condition in \eqref{eq3.9}, i.e., $ker(\phi)\in  {^{[C,-]}}{_A}\mathfrak S(\psi)$. By assumption,  $(C,-):\mathfrak S\longrightarrow \mathfrak S$ preserves epimorphisms, and it follows that $coker(\phi)$ satisfies the condition in \eqref{eq3.9}, i.e., $coker(\phi)\in  {^{[C,-]}}{_A}\mathfrak S(\psi)$. We now see that $_{\hspace{1em}A}^{[C,-]}\mathfrak S(\psi)$ is an abelian category.
	\end{proof}
	\begin{Thm}\label{T3.14}
		Let $\mathfrak S$ be a $k$-linear Grothendieck category and let $\mathcal G$ be a generator for $_A\mathfrak S$. Then, $(C,\mathcal G)$ is a generator for $_{\hspace{1em}A}^{[C,-]}\mathfrak S(\psi)$.
	\end{Thm}
	\begin{proof}
		Let $\matholdcal M'\hookrightarrow \matholdcal M$ be a monomorphism in $_{\hspace{1em}A}^{[C,-]}\mathfrak S(\psi)$ that is not an isomorphism in $_{\hspace{1em}A}^{[C,-]}\mathfrak S(\psi)$. Since kernels in $_{\hspace{1em}A}^{[C,-]}\mathfrak S(\psi)$ and $_A\mathfrak S$ are computed in $\mathfrak S$, it follows that $\matholdcal M'\hookrightarrow \matholdcal M$ is also a monomorphism in $_A\mathfrak S$ that is not an isomorphism. Then since $\mathcal G$ is a generator for $_A\mathfrak S$, we can choose a morphism $\phi:\mathcal G\longrightarrow \matholdcal M$ in $_A\mathfrak S$ which does not factor through $\matholdcal M'$. By the adjunction in Proposition \ref{L3.12}, we obtain a morphism $\tilde{\phi}:(C, \mathcal G)\longrightarrow \matholdcal M$ in $_{\hspace{1em}A}^{[C,-]}\mathfrak S(\psi)$ that does not factor through $\matholdcal M'$. It now follows from \cite[\S 1.9]{Gro} that $(C,\mathcal G)$ is a generator for $_{\hspace{1em}A}^{[C,-]}\mathfrak S(\psi)$.
	\end{proof}
	%%%%%%%%%%%%%%%%%%%%%%%%%%%%%%%%%%%%%%%%%%%%%%%%%%%%%%%%%%%%%%%%%%%%%%%%%%%%%%%%%%%%%%%%%%%%%%%%%%%%%%%%%%%%%%%%%%%%%%%%%%%%%%%%%%%%%%%%%%%%%%%%%%%%%%%%%%%%%%%%%%%%%%%%%%%%%%%%%%%%%%%%%%%%%%%%%%%%%%%%%%%%%%%%%%%%%%%%%%%%%%%%%%%%%%%%%%%%%%%%
	\section{$\mathfrak S$-contra-Galois measurings, $\mathfrak S$-co-Galois measurings and adjoint functors}
	We take $k$-algebras $(A,\mu,\eta)$ and $(A',\mu',\eta')$ and $k$-coalgebras $(C,\Delta,\epsilon)$ and $(C',\Delta',\epsilon')$. Let $(A,C,\psi)$ and $(A',C',\psi')$ be   entwining structures. We first recall the notion of a measuring between entwining structures due to  Brzezi\'{n}ski   \cite{Tb}.

	\begin{defn}(see \cite[Definition 3.2]{Tb})
		Let $(A,C,\psi)$ and $(A',C',\psi')$ be two entwining structures and let $\alpha: C'\otimes A'\longrightarrow A$ and $\gamma:C'\longrightarrow A\otimes C$ be two $k$-linear maps such that
		\begin{align}
			&\alpha\circ (C'\otimes \mu')=\mu\otimes (\alpha\otimes \alpha)\circ (C'\otimes \psi'\otimes A') \circ (\Delta'\otimes A'\otimes A')\qquad \alpha\circ (C'\otimes \eta')=\eta\circ \epsilon'\label{eq4.1}\\
			&(A\otimes \Delta)\circ \gamma=(\mu\otimes C\otimes C)\circ (A\otimes \psi \otimes C)\circ (\gamma \otimes \gamma)\circ (\Delta')\qquad \qquad(A\otimes \epsilon)\circ \gamma=\eta\circ \epsilon'\label{eq4.2}\\
			&(\mu\otimes C)\circ (\alpha\otimes\gamma)\circ(C'\otimes\psi')\circ (\Delta'\otimes A')=(\mu\otimes C)\circ(A\otimes \psi)\circ (\gamma\otimes\alpha)\circ (\Delta'\otimes A')\label{eq4.3}
		\end{align}
		Then, the pair $(\alpha, \gamma)$ is said to be a measuring from $(A',C',\psi')$ to $(A,C,\psi)$ and is denoted by $(A',C',\psi')~|\frac{\alpha}{\gamma}~(A,C,\psi)$.
	\end{defn} We assume throughout that the functors $(C,-)$, $(C',-):\mathfrak S\longrightarrow \mathfrak S$ are exact. 
	\subsection{$\mathfrak S$-contra Galois measurings  and entwined contramodule objects}
	%Given $k$-algebras $A$ and $A'$ and $k$-coalgebras $C$ and $C'$, let $(A,C,\psi)$ and $(A',C',\psi')$ be two entwining structures.
	Let $\matholdcal M\in \mathfrak S$. Using the notation in \eqref{eq3.10} and   applying the functor $(-,\matholdcal M)$ to the expressions in (\ref{eq4.1}) and  (\ref{eq4.2}), we have the following commutative diagrams.
	\begin{equation}\label{eq4.4}\begin{tikzcd}[row sep=1.8em, column sep = 3.8em]
			(A, \matholdcal M)  \arrow{r}{(\alpha,\matholdcal M)} \arrow{d}{(\mu, \matholdcal M)}& (A', C', \matholdcal M) \arrow{r} {(\mu', C', \matholdcal M)} &(A', A', C', \matholdcal M) \\%
			(A, A,\matholdcal M) \arrow{r}{(\alpha,\alpha,\matholdcal M)}&  (A', C', A', C', \matholdcal M)\arrow{r}{(A', \psi', C',\matholdcal M)}& (A', A', C', C', \matholdcal M)\arrow{u}[swap] {(A',A', \Delta',\matholdcal M)}
		\end{tikzcd}
	\end{equation}
	and
	\begin{equation}\label{eq4.5}\begin{tikzcd}[row sep=1.8em, column sep = 3.8em]
			(C,C,A, \matholdcal M)  \arrow{r}{(\Delta,A,\matholdcal M)} \arrow{d}{(C,C,\mu, \matholdcal M)}& (C,A, \matholdcal M) \arrow{r} {(\gamma, \matholdcal M)} &(C',\matholdcal M) \\%
			(C,C,A, A,\matholdcal M) \arrow{r}{(C,\psi,A,\matholdcal M)}&  (C, A,C, A, \matholdcal M)\arrow{r}{(\gamma,\gamma \matholdcal M)}& (C', C',\matholdcal M)\arrow{u} {(\Delta',\matholdcal M)}
		\end{tikzcd}
	\end{equation}
	Similarly,  applying the functor $(-,\matholdcal M)$ to (\ref{eq4.3}) gives the commutative diagram
	\begin{equation}\label{eq4.6}\small \begin{tikzcd}[row sep=1.8em, column sep = 3.8em]
			(C,A, \matholdcal M)  \arrow{r}{(C,\mu,\matholdcal M)} \arrow{d}{(C,\mu, \matholdcal M)}& (C,A,A,\matholdcal M) \arrow{r} {(\gamma,\alpha, \matholdcal M)} &(C', A', C', \matholdcal M)\arrow{r}{(\psi',C',\matholdcal M)}&(A',C',C',\matholdcal M)\arrow{r}{(A',\Delta',\matholdcal M)}&(A',C',\matholdcal M) \\%
			(C,A, A,\matholdcal M) \arrow{rr}{(\psi,A,\matholdcal M)}&&  (A, C, A, \matholdcal M)\arrow{rr}{(\alpha,\gamma,\matholdcal M)}&& (A', C', C', \matholdcal M)\arrow{u}[swap] {(A', \Delta',\matholdcal M)}
		\end{tikzcd}
	\end{equation}
	In \cite[$\S$ 3]{Tb},  Brzezi\'{n}ski  has used measurings of entwining structures to obtain adjoint functors, similar to restriction and extension of scalars,  between categories of entwined modules. We are inspired by the methods of \cite{Tb} 
	to use measurings of entwining structures to construct adjoint functors between categories  $_{\hspace{1em}A}^{[C,-]}\mathfrak S(\psi)$ and $_{\hspace{1em}A'}^{[C',-]}\mathfrak S(\psi')$ of entwined contradmodule objects in $\mathfrak S$. In Section 4.2, we will obtain adjoint functors between categories $\mathfrak S_A^C(\psi)$ and $\mathfrak S_{A'}^{C'}(\psi')$
	of entwined comodule objects in $\mathfrak S$ induced by measurings.  
	\begin{thm}\label{P4.2}
		Let $\mathfrak S$ be a $k$-linear Grothendieck category. Let $(\alpha, \gamma)$ be a measuring from $(A',C',\psi')$ to $(A,C,\psi)$. Then, the following statements hold.
		
		\smallskip
		(i) Let $(\matholdcal M,\mu_{ \matholdcal M})\in~_A\mathfrak S$. Then, $(C', \matholdcal M)\in~ _{\hspace{1em}A'}^{[C',-]}\mathfrak S(\psi')$ with structure maps
		\begin{equation}\label{vidd1}
		\begin{array}{c}
		\pi_{ (C',\matholdcal M)}: (C',C',\matholdcal M)\xrightarrow{(\Delta',\matholdcal M)} (C',\matholdcal M)
		\\
		\mu_{(C',\matholdcal M)}: (C',\matholdcal M)\xrightarrow{(C',\mu_{\matholdcal M})} (C',A,\matholdcal M)\xrightarrow{ (C',\alpha, \matholdcal M)} (C',A',C',\matholdcal M)\xrightarrow{(\psi',C',\matholdcal M)} (A',C',C',\matholdcal M)\xrightarrow{(A',\Delta', \matholdcal M)}(A',C',\matholdcal M)\\
		\end{array}
		\end{equation}  
		
		\smallskip
		(ii) Let $(\matholdcal M',\pi_{ \matholdcal M'})\in~^{[C',-]}\mathfrak S$. Then, $(A,\matholdcal M')\in ~_{\hspace{1em}A}^{[C,-]}\mathfrak S(\psi)$ with  structure maps
		\begin{equation}
		\begin{array}{c}
		\mu_{(A,\matholdcal M')}: (A,\matholdcal M') \xrightarrow{(\mu,\matholdcal M')}(A,A,\matholdcal M') \\
		\pi_{(A,\matholdcal M')}:(C,A,\matholdcal M')\xrightarrow{(C,\mu,\matholdcal M')}(C,A,A,\matholdcal M')\xrightarrow{(\psi,A,\matholdcal M')}(A,C,A\matholdcal M')\xrightarrow{(A,\gamma,\matholdcal M')}(A,C',\matholdcal M)\xrightarrow{(A,\pi_{\matholdcal M'})}(A,\matholdcal M')\\
		\end{array}
		\end{equation} 
		  
	\end{thm}
	\begin{proof}
		(i) It is clear that the morphism $\pi_{ (C',\matholdcal M)}$ in \eqref{vidd1} makes $(C',\matholdcal M)$ an object of $^{[C',-]}\mathfrak S$. To show that $\mu_{(C',\matholdcal M)}$ in \eqref{vidd1} makes $(C',\matholdcal M)$ an object of $_{A'}\mathfrak S$,  we need to check that the following diagram commutes.
		\begin{small}
			\begin{equation}\label{vidd2}
				\begin{tikzcd}[row sep=2.5em, column sep = 4.8em]
					(C',\matholdcal M) \arrow{r}{(C',\mu_{\matholdcal M})} \arrow{dddd}{\mu_{(C',\matholdcal M)}}
					& (C',A,\matholdcal M)\arrow{r}{(C',\alpha,\matholdcal M)}\arrow{d}{(C',A,\mu_{\matholdcal M})} &(C',A',C',\matholdcal M)\arrow{d}{(C',A',C',\mu_{\matholdcal M})}\arrow{r}{\pi_{(A',C',\matholdcal M)}}&(A',C',\matholdcal M)\arrow{d}{(A',C', \mu_{\matholdcal M})}\\
					&(C',A,A,\matholdcal M)\arrow{r}{(C',\alpha,A,\matholdcal M)}\arrow{rd}{(C',\alpha,\alpha,\matholdcal M)}&(C',A',C',A,\matholdcal M)\arrow{d}{(C',A',C',\alpha,\matholdcal M)}\arrow{r}{\pi_{(A',C',A,\matholdcal M)}}&(A',C',A,\matholdcal M)\arrow{d}{(A',C',\alpha,\matholdcal M)}\\
					&&(C',A',C',A',C',\matholdcal M) \arrow{r}
					{\pi_{(A',C',A',C',\matholdcal M)}}\arrow{d}{(C',A',C',\psi',\matholdcal M)} &(A',C',A',C',\matholdcal M)\arrow{d}{(A',\psi',C',\matholdcal M)}\\
					&&(C',A',A',C',C',\matholdcal M)\arrow{r}{\pi_{(A',A',C',C',\matholdcal M)}}&(A',A',C',C',\matholdcal M)\arrow{d}{(A',A',\Delta',\matholdcal M)}\\
					(A',C',\matholdcal M)\arrow{rrr}{(\mu',C',\matholdcal M)}&&&(A',A',C',\matholdcal M)
				\end{tikzcd}
			\end{equation}
		\end{small}
Since $(C',\matholdcal M)\in {^{[C',-]}\mathfrak S}$, it follows by Lemma \ref{L3.9} that  $^{[C',-]}\mathcal T(C',\matholdcal M)=(A',C',\matholdcal M)\in~_{\hspace{1em}A'}^{[C',-]}\mathfrak S(\psi')$. The condition in \eqref{eq3.9} now gives
		\begin{align*}
			&(\mu', C',\matholdcal M)\circ \mu_{(C',\matholdcal M)}\\\notag&= (\mu', C',\matholdcal M)\circ(A',\Delta',\matholdcal M)\circ (\psi', C',\matholdcal M)\circ (C',\alpha, \matholdcal M)\circ (C',\mu_{\matholdcal M})\\\notag
			&= \mu_{(A',C',\matholdcal M)}\circ \pi_{(A',C',\matholdcal M)}\circ (C',\alpha, \matholdcal M)\circ (C',\mu_{\matholdcal M})\\\notag
			&=(A',A',\Delta',\matholdcal M)\circ (A',\psi',C',\matholdcal M)\circ (\psi',A',C',\matholdcal M)\circ (C',\mu', C',\matholdcal M)\circ (C',\alpha, \matholdcal M)\circ (C',\mu_{\matholdcal M})\tag*{(\textup{by}~ (\ref{eq3.9}))}\\\notag
			&= (A',A',\Delta',\matholdcal M)\circ (A',\psi',C',\matholdcal M)\circ (\psi',A',C',\matholdcal M)\circ (C',A',A',\Delta',\matholdcal M)\circ (C', A',\psi',C',\matholdcal M)\circ\\\notag&\qquad (C',\alpha,\alpha,\matholdcal M)\circ (C',\mu,\matholdcal M)\circ (C',\mu_{\matholdcal M})\tag*{(\textup{by}~ (\ref{eq4.4}))}\\\notag
			&=(A',A',\Delta',\matholdcal M)\circ \pi_{(A',A',C',C',\matholdcal M)}\circ(C',A',\psi',C',\matholdcal M)\circ (C',\alpha,\alpha,\matholdcal M)\circ (C',\mu,\matholdcal M)\circ (C',\mu_{\matholdcal M})\notag
		\end{align*}
		It now follows that 
		\begin{align*}
			&(\mu', C',\matholdcal M)\circ \mu_{(C',\matholdcal M)}\\\notag&=
			(A',A',\Delta',\matholdcal M)\circ \pi_{(A',C',A',C',\matholdcal M)}\circ(C',A',\psi',C',\matholdcal M)\circ (C',A',C',\alpha,\matholdcal M)\circ(C',\alpha,A,\matholdcal M)\circ (C',\mu,\matholdcal M)\circ (C',\mu_{\matholdcal M})\\\notag
			&=(A',A',\Delta',\matholdcal M)\circ \pi_{(A',C',A',C',\matholdcal M)}\circ(C',A',\psi',C',\matholdcal M)\circ (C',A',C',\alpha,\matholdcal M)\circ(C',\alpha,A,\matholdcal M)\circ (C',A,\mu_{\matholdcal M})\circ (C',\mu_{\matholdcal M})\\\notag
			&=(A',A',\Delta',\matholdcal M)\circ(A',\psi',C',\matholdcal M)\circ (A',C',\alpha,\matholdcal M)\circ (A',C',\mu_{\matholdcal M})\circ \pi_{(A',C',\matholdcal M)}\circ (C',\alpha,\matholdcal M)\circ (c',\mu_{\matholdcal M})\\\notag
			&=(A',A',\Delta',\matholdcal M)\circ(A',\psi',C',\matholdcal M)\circ (A',C',\alpha,\matholdcal M)\circ (A',C',\mu_{\matholdcal M})\circ (A',\Delta',\matholdcal M)\circ (\psi',C',\matholdcal M)\circ (C',\alpha,\matholdcal M)\circ (c',\mu_{\matholdcal M})\\\notag
			&=(A',\mu_{(C',\matholdcal M)})\circ \mu_{(C',\matholdcal M)}
		\end{align*}
	It remains to show that $(C',\matholdcal M)$ satisfies the condition in  \eqref{eq3.9}. For  this, we observe that
		\begin{align*}
			\mu_{(C',\matholdcal M)}\circ \pi_{(C',\matholdcal M)}&=(A',\Delta',\matholdcal M)\circ(\psi',C',\matholdcal M)\circ (C',\alpha,\matholdcal M)\circ (C',\mu_{\matholdcal M})\circ (\Delta',\matholdcal M)\\\notag
			&= (A',\Delta',\matholdcal M)\circ(\psi',C',\matholdcal M)\circ (C',\alpha,\matholdcal M)\circ(\Delta',A,\matholdcal M)\circ (C',C',\mu_{\matholdcal M})\\\notag
			&=(A',\Delta',\matholdcal M)\circ(\psi',C',\matholdcal M)\circ(\Delta',A',C',\matholdcal M)\circ (C',C',\alpha,\matholdcal M)\circ (C',C',\mu_{\matholdcal M})\\\notag
			&=(A',\Delta',\matholdcal M)\circ(A',\Delta',C',\matholdcal M)\circ (\psi',C',C',\matholdcal M)\circ(C',\psi',C',\matholdcal M) \circ (C',C',\alpha,\matholdcal M)\circ (C',C',\mu_{\matholdcal M})&&(\textup{by}~ (\ref{eq3.12}))\\\notag
			&=(A',\Delta',\matholdcal M)\circ(\psi',C',\matholdcal M)\circ (C',A',\Delta',\matholdcal M)\circ(C',\psi',C',\matholdcal M) \circ (C',C',\alpha,\matholdcal M)\circ (C',C',\mu_{\matholdcal M})\\\notag
			&=(A',\pi_{(C',\matholdcal M)})\circ (\psi', (C',\matholdcal M))\circ (C',\mu_{(C',\matholdcal M)})
		\end{align*}
		This proves the result.
		
		\smallskip
		(ii) The proof of this is dual to that of (i).
	\end{proof}
	\begin{lem}\label{L4.3}
		Let $(\alpha, \gamma)$ be a measuring from $(A',C',\psi')$ to $(A,C,\psi)$. Then, for any object $(\matholdcal M,\pi_{ \matholdcal M},\mu_{ \matholdcal M})\in ~ _{\hspace{1em}A}^{[C,-]}\mathfrak S(\psi)$ and  any object $(\matholdcal M',\pi_{ \matholdcal M'},\mu_{ \matholdcal M'})\in~_{\hspace{1em}A'}^{[C',-]}\mathfrak S(\psi')$, we have:
		
		\smallskip
		(i) Let $s^{\matholdcal M}:(C',C,\matholdcal M)\longrightarrow (C',\matholdcal M)$ be defined as follows: 
		\begin{equation} \label{dv1}s^{\matholdcal M}:=(C',\pi_{\matholdcal M})-(\Delta', \matholdcal M)\circ (C',\gamma,\matholdcal M)\circ (C',C,\mu_{\matholdcal M})
		\end{equation} Then, $s^{\matholdcal M}$ is a morphism in $_{\hspace{1em}A'}^{[C',-]}\mathfrak S(\psi')$, where $(C',\matholdcal M)$ and $(C',C,\matholdcal M)$ are treated as objects in $_{\hspace{1em}A'}^{[C',-]}\mathfrak S(\psi')$.

		\smallskip
		(ii) Let $s_{\matholdcal M'}:(A,\matholdcal M')\longrightarrow (A,A',\matholdcal M')$ be defined as follows: 
		\begin{equation}\label{dv2}  s_{\matholdcal M'}:=(A,\mu_{\matholdcal M'})-(A,A',\pi_{\matholdcal M'})\circ (A,\alpha,\matholdcal M')\circ(\mu,\matholdcal M')
		\end{equation} Then, $s_{\matholdcal M'}$ is a morphism in $_{\hspace{1em}A}^{[C,-]}\mathfrak S(\psi)$, where $(A,\matholdcal M')$ and $(A,A',\matholdcal M')$ are treated as the objects in $_{\hspace{1em}A}^{[C,-]}\mathfrak S(\psi)$.
		
	\end{lem}
	\begin{proof}
		(i) It is clear $s^{\matholdcal M}$ is a morphism in $^{[C',-]}\mathfrak S.$ We now show that $s^{\matholdcal M}$ is a morphism in $_{A'}\mathfrak S$, i.e., $s^{\matholdcal M}$ satisfies $\mu_{(C',\matholdcal M)}\circ s^{\matholdcal M}=(A',s^{\matholdcal M})\circ \mu_{(C',C,\matholdcal M)}$. From \eqref{vidd1} and \eqref{dv1}, we have
		\begin{align}\label{eq4.7}
			\mu_{(C',\matholdcal M)}\circ s^{\matholdcal M}&=\big((A',\Delta',\matholdcal M)\circ (\psi',C',\matholdcal M)\circ (C',\alpha,\matholdcal M)\circ (C',\mu_{\matholdcal M})\circ (C',\pi_{\matholdcal M})\big)\\ \notag
			&\qquad -\big((A',\Delta',\matholdcal M)\circ (\psi',C',\matholdcal M)\circ (C',\alpha,\matholdcal M)\circ  (C',\mu_{\matholdcal M})\circ(\Delta',\matholdcal M)\circ (C',\gamma,\matholdcal M)\circ (C',C,\mu_{\matholdcal M})\big)\notag
		\end{align}
		On the one hand, we note that 
		\begin{equation}\label{eq4.8}
		\begin{array}{l}
			(A',\Delta',\matholdcal M)\circ (\psi',C',\matholdcal M)\circ (C',\alpha,\matholdcal M)\circ (C',\mu_{\matholdcal M})\circ (C',\pi_{\matholdcal M})\\=(A',\Delta',\matholdcal M)\circ (\psi',C',\matholdcal M)\circ (C',\alpha,\matholdcal M)\circ (C',A,\pi_{\matholdcal M})\circ (C',\psi, \matholdcal M)\circ (C',C,\mu_{\matholdcal M}) \qquad \mbox{(as $ \matholdcal M\in~_{\hspace{1em}A}^{[C,-]}\mathfrak S(\psi)$)} \\
			=(A',\Delta',\matholdcal M)\circ (\psi',C',\matholdcal M)\circ (C',A',C',\pi_\matholdcal M)\circ (C',\alpha,C,\matholdcal M)\circ (C',\psi, \matholdcal M)\circ (C',C,\mu_{\matholdcal M})\\
			=(A',\Delta',\matholdcal M)\circ (A',C',C',\pi_\matholdcal M)\circ (\psi',C',C, \matholdcal M)\circ (C',\alpha,C,\matholdcal M)\circ (C',\psi, \matholdcal M)\circ (C',C,\mu_{\matholdcal M})\\
			=(A',C',\pi_\matholdcal M)\circ (A',\Delta',C,\matholdcal M)\circ (\psi',C',C, \matholdcal M)\circ (C',\alpha,C,\matholdcal M)\circ (C',\psi, \matholdcal M)\circ (C',C,\mu_{\matholdcal M})\\
			\end{array}
		\end{equation}
		Since $(C',\matholdcal M)\in~_{\hspace{1em}A'}^{[C',-]}\mathfrak S(\psi')$, we have on the other hand
		\begin{equation}\label{eq4.9}
		\begin{array}{l}
			(A',\Delta',\matholdcal M)\circ (\psi',C',\matholdcal M)\circ (C',\alpha,\matholdcal M)\circ (C',\mu_{\matholdcal M})\circ(\Delta',\matholdcal M)\circ (C',\gamma,\matholdcal M)\circ (C',C,\mu_{\matholdcal M})\\
			=(A',\Delta',\matholdcal M)\circ (\psi',C',\matholdcal M)\circ (C',A',\Delta',\matholdcal M)\circ (C',\psi',C',\matholdcal M)\circ(C',C',\alpha,\matholdcal M)\circ (C',C',\mu_\matholdcal M)\circ (C',\gamma,\matholdcal M)\circ(C',C,\mu_{\matholdcal M})\\
			=(A',\Delta',\matholdcal M)\circ (\psi',C',\matholdcal M)\circ (C',A',\Delta',\matholdcal M)\circ(C',\psi',C',\matholdcal M)\circ (C',\gamma,\alpha,\matholdcal M)\circ (C',C,\mu, \matholdcal M)\circ(C',C,\mu_{\matholdcal M})\\
			=(A',\Delta',\matholdcal M)\circ (\psi',C',\matholdcal M)\circ (C',A',\Delta',\matholdcal M)\circ(C',\alpha,\gamma,\matholdcal M)\circ (C',\psi,A,\matholdcal M)\circ (C',C,\mu, \matholdcal M)\circ(C',C,\mu_{\matholdcal M}) \qquad \mbox{(by \eqref{eq4.6})}\\
			=\pi_{(A',C',\matholdcal M)}\circ (C',A',\Delta',\matholdcal M)\circ(C',\alpha,\gamma,\matholdcal M)\circ (C',\psi,A,\matholdcal M)\circ (C',C,\mu, \matholdcal M)\circ(C',C,\mu_{\matholdcal M})\\
		\end{array}
		\end{equation}
		where the third equality in \eqref{eq4.9} follows from the following commutative diagram
		\begin{equation}
			\begin{tikzcd}[row sep=1.8em, column sep = 3.8em]
				(C',C,A, \matholdcal M)  \arrow{r}{(C',C,\mu,,\matholdcal M)} \arrow{d}{(C',\gamma,\matholdcal M)}& (C',C,A,A, \matholdcal M) \arrow{r} {(C',C,A,\alpha,\matholdcal M)}\arrow{d}{(C',\gamma,A,\matholdcal M)}\arrow{rd}{(C',\gamma,\alpha,\matholdcal M)} &(C',C,A,A',C',\matholdcal M) \arrow{d}{(C',\gamma,A',C',\matholdcal M)}\\%
				(C',C',\matholdcal M) \arrow{r}{(C',C'\mu_{\matholdcal M})}&  (C',C', A, \matholdcal M)\arrow{r}{(C',C',\alpha, \matholdcal M)}& (C', C',A',C',\matholdcal M)
			\end{tikzcd}
		\end{equation}
		We now consider the following commutative diagram
		\begin{small}
			\begin{equation}\label{416tf}
				\begin{tikzcd}[row sep=2.8em, column sep = 4.8em]
					(C',C,A,\matholdcal M) \arrow{r}{(C',\psi,\matholdcal M)} \arrow{d}{(C',C,\mu,\matholdcal M)}
					& (C',A,C,\matholdcal M)\arrow{r}{(C',\alpha,C,\matholdcal M)}\arrow{d}{(C',A,C,\mu_{\matholdcal M})} &(C',A',C',C,\matholdcal M)\arrow{d}{(C',A',C',C,\mu_{\matholdcal M})}\arrow{r}{\pi_{(A',C',C,\matholdcal M)}}&(A',C',C,\matholdcal M)\arrow{d}{(A',C',C, \mu_{\matholdcal M})}\\
					(C',C,A,A,\matholdcal M)\arrow{r}{(C',\psi,A,\matholdcal M)}&(C',A',C,A,\matholdcal M)\arrow{rd}{(C',\alpha,\gamma,\matholdcal M)}\arrow{r}{(C',\alpha,C,A,\matholdcal M)}&(C',A',C',C,A,\matholdcal M)\arrow{d}{(C',A',C',\gamma,\matholdcal M)}\arrow{r}{(\pi_{(A',C',C,A,\matholdcal M)})}&(A',C',C,A,\matholdcal M)\arrow{d}{(A',C',\gamma,\matholdcal M)}\\
					&&(C',A',C',C',\matholdcal M) \arrow{r}
					{\pi_{(A',C',C',\matholdcal M)}}\arrow{d}{(C',A',\Delta',\matholdcal M)} &(A',C',C',\matholdcal M)\arrow{d}{(A',\Delta',\matholdcal M)}\\
					&&(C',A',C',\matholdcal M)\arrow{r}{\pi_{(A',C',\matholdcal M)}}&(A',C',\matholdcal M)
				\end{tikzcd}
			\end{equation}
		\end{small}
		Using \eqref{416tf}, we see that  \eqref{eq4.9} reduces to 
		\begin{align}\label{eq4.10}
			&(A',\Delta',\matholdcal M)\circ (\psi',C',\matholdcal M)\circ (C',\alpha,\matholdcal M)\circ (C',\mu_{\matholdcal M})\circ(\Delta',\matholdcal M)\circ (C',\gamma,\matholdcal M)\circ (C',C,\mu_{\matholdcal M})\\\notag
			&=(A',\Delta',\matholdcal M)\circ (A',C',\gamma,\matholdcal M)\circ (A',C',C,\mu_\matholdcal M)\circ\pi_{(A',C',C,\matholdcal M)}\circ (C',\alpha,C,\matholdcal M)\circ (C',\psi, \matholdcal M)\circ(C',C,\mu_{\matholdcal M})\\\notag
			&=(A',\Delta',\matholdcal M)\circ (A',C',\gamma,\matholdcal M)\circ (A',C',C,\mu_\matholdcal M)\circ(A',\Delta',C,\matholdcal M)\circ(\psi',C',C,\matholdcal M)\circ (C',\alpha,C,\matholdcal M)\circ (C',\psi, \matholdcal M)\circ(C',C,\mu_{\matholdcal M})
		\end{align}
		From \eqref{eq4.8},  \eqref{eq4.10} and the expression in \eqref{eq4.7},  we get
		\begin{align*}
			\mu_{(C',\matholdcal M)}\circ s^{\matholdcal M}&=\big((A',C',\pi_\matholdcal M)\circ (A',\Delta',C,\matholdcal M)\circ (\psi',C',C, \matholdcal M)\circ (C',\alpha,C,\matholdcal M)\circ (C',\psi, \matholdcal M)\circ (C',C,\mu_{\matholdcal M})\big)-\big((A',\Delta',\matholdcal M)\circ\\\notag&\qquad (A',C',\gamma,\matholdcal M)\circ (A',C',C,\mu_\matholdcal M)\circ(A',\Delta',C,\matholdcal M)\circ(\psi',C',C,\matholdcal M)\circ (C',\alpha,C,\matholdcal M)\circ (C',\psi, \matholdcal M)\circ(C',C,\mu_{\matholdcal M})\big)\\\notag
			&=(A',s^{\matholdcal M})\circ \mu_{(C',C,\matholdcal M)}
		\end{align*}
		This proves the result.
		
		\smallskip
		(ii) The proof of this is dual to that of (i).
	\end{proof}
	We are now ready to define adjoint functors between the categories of entwined contramodule objects in $\mathfrak S$ over $(A,C,\psi)$ and $(A',C',\psi')$. Given a measuring $(\alpha,\gamma)$ of entwining structures from $(A',C',\psi')$ to $(A,C,\psi)$, it follows from  Proposition \ref{P4.2} that for any $\matholdcal M\in~_{\hspace{1em}A}^{[C,-]}\mathfrak S(\psi)$, both $(C',\matholdcal M)$ and $(C',C,\matholdcal M)$ are objects in $_{\hspace{1em}A'}^{[C',-]}\mathfrak S(\psi')$. By Lemma \ref{L4.3}, we know that $s^{\matholdcal M}:(C',C,\matholdcal M)\longrightarrow (C',\matholdcal M)$ is a morphism in $_{\hspace{1em}A'}^{[C',-]}\mathfrak S(\psi')$.   We define $\widetilde{Cohom_C}(C',\matholdcal M)\in~_{\hspace{1em}A'}^{[C',-]}\mathfrak S(\psi')$ to be the cokernel
	\begin{equation}\label{418tf}
		(C',C,\matholdcal M)\xrightarrow{s^{\matholdcal M}}(C',\matholdcal M)\xrightarrow{p^{\matholdcal M}}\widetilde{Cohom_C}(C',\matholdcal M)\longrightarrow 0
	\end{equation}
	As such, we have an induced functor 
	\begin{align*}
		\widetilde{Cohom_C}(C',-):~_{\hspace{1em}A}^{[C,-]}\mathfrak S(\psi)\longrightarrow ~_{\hspace{1em}A'}^{[C',-]}\mathfrak S(\psi'),\qquad \matholdcal M\mapsto \widetilde{Cohom_C}(C',\matholdcal M) 
	\end{align*}
	On the other hand, for $\matholdcal M'\in~_{\hspace{1em}A'}^{[C',-]}\mathfrak S(\psi')$, we know from Proposition \ref{P4.2} and Lemma \ref{L4.3} that both $(A,\matholdcal M')$ and $(A,A',\matholdcal M)$ are objects in $_{\hspace{1em}A}^{[C,-]}\mathfrak S(\psi)$ and $s_{\matholdcal M'}$ is a morphism in $_{\hspace{1em}A}^{[C,-]}\mathfrak S(\psi)$. Let $\widetilde{Hom_{A'}}(A,\matholdcal M')$ denote the kernel of $s_{\matholdcal M'}$. We now set $\widetilde{Hom_{A'}}(A,\matholdcal M') \in~_{\hspace{1em}A'}^{[C',-]}\mathfrak S(\psi')$ to be the kernel 
	\begin{equation}\label{419uy}
		0\longrightarrow \widetilde{Hom_{A'}}(A,\matholdcal M')\xrightarrow{\iota_{\matholdcal M'}} (A,\matholdcal M')\xrightarrow{s_{\matholdcal M'}}(A,A',\matholdcal M')
	\end{equation}
	This determines a functor 
	\begin{align*}
		\widetilde{Hom_{A'}}(A,-):~_{\hspace{1em}A'}^{[C',-]}\mathfrak S(\psi')\longrightarrow _{\hspace{1em}A}^{[C,-]}\mathfrak S(\psi),\qquad \matholdcal M'\mapsto \widetilde{Hom_{A'}}(A,\matholdcal M')
	\end{align*}
	\begin{thm}\label{P4.4}
		Let $(\alpha, \gamma)$ be a measuring from $(A',C',\psi')$ to $(A,C,\psi)$ and let $\mathfrak S$ be a $k$-linear Grothendieck category. Then, $\widetilde{Cohom_C}(C',-):~_{\hspace{1em}A}^{[C,-]}\mathfrak S(\psi)\longrightarrow ~_{\hspace{1em}A'}^{[C',-]}\mathfrak S(\psi')$ is  left adjoint to the functor $\widetilde{Hom_{A'}}(A,-):~_{\hspace{1em}A'}^{[C',-]}\mathfrak S(\psi')\longrightarrow _{\hspace{1em}A}^{[C,-]}\mathfrak S(\psi)$. In other words, we have natural isomorphisms
		\begin{equation}
			_{\hspace{1em}A'}^{[C',-]}\mathfrak S(\psi')(\widetilde{Cohom_C}(C',\matholdcal M),\matholdcal M')\cong  {_{\hspace{1em}A}^{[C,-]}\mathfrak S(\psi)}(\matholdcal M,\widetilde{Hom_{A'}}(A,\matholdcal M'))
		\end{equation}
		for any $(\matholdcal M,\pi_{ \matholdcal M},\mu_{ \matholdcal M})\in~_{\hspace{1em}A}^{[C,-]}\mathfrak S(\psi)$ and $(\matholdcal M',\pi_{ \matholdcal M'},\mu_{ \matholdcal M'})\in~ _{\hspace{1em}A'}^{[C',-]}\mathfrak S(\psi')$.
	\end{thm}		
	\begin{proof}
		Let $\zeta:\widetilde{Cohom_C}(C',\matholdcal M) \longrightarrow\matholdcal M'$ be a morphism in $_{\hspace{1em}A'}^{[C',-]}\mathfrak S(\psi').$ We define $\xi:\matholdcal M\longrightarrow (A,\matholdcal M')$ to be
		the following composition
		\begin{equation}\label{421qe}
			\xi:\matholdcal M\xrightarrow{\mu_{\matholdcal M}}(A,\matholdcal M)\xrightarrow{(A,\epsilon',\matholdcal M)}(A,C',\matholdcal M)\xrightarrow{(A,p^{\matholdcal M})}(A,\widetilde{Cohom_C}(C',\matholdcal M))\xrightarrow{(A,\zeta)}(A,\matholdcal M')
		\end{equation}
		From \eqref{421qe}, it is evident that $\xi$ is a morphism in $_A\mathfrak S$. To prove that $\xi$ is also a morphism in $^{[C,-]}\mathfrak S$, we must show that $\xi\circ \pi_{\matholdcal M}=\pi_{(A,\matholdcal M')}\circ (C,\xi)$. Since $\pi_{\matholdcal M'}$ and $p^{\matholdcal M}$ are morphisms in $_{\hspace{1em}A'}^{[C',-]}\mathfrak S(\psi')$, we consider the following diagram 
		\begin{small}
			\begin{equation*}
				\begin{tikzcd}[row sep=2.8em, column sep = 4.8em]
					(C,\matholdcal M) \arrow{r}{(C,\mu_{\matholdcal M})} \arrow{d}{\mu_{(C,\matholdcal M)}}
					& (C,A,\matholdcal M)\arrow{r}{(C,A,\epsilon',\matholdcal M)}\arrow{d}{\mu_{(C,A,\matholdcal M)}} &(C,A,C',\matholdcal M)\arrow{d}{\mu_{(C,A,C',\matholdcal M)}}\arrow{r}{(C,A,p^{\matholdcal M})}&(C,A,\widetilde{Cohom_C}(C',\matholdcal M))\arrow{d}{\mu_{(C,A,\widetilde{Cohom_C}(C',\matholdcal M))}}\arrow{r}{(C,A,\zeta)}&(C,A,\matholdcal M')\arrow{d}{\mu_{(C,A,\matholdcal M')}}\\
					(A,C,\matholdcal M)\arrow{r}{(A,C,\mu_{\matholdcal M})}\arrow{dd}{(A,\pi_{ \matholdcal M})}&(A,C,A,\matholdcal M)\arrow{d}{(A,\gamma,\matholdcal M)}\arrow{r}{(A,C,A,\epsilon',\matholdcal M)}&(A,C,A,C',\matholdcal M)\arrow{d}{(A,\gamma,C',\matholdcal M)}\arrow{r}{(A,C,A,p^{\matholdcal M})}&(A,C,A,\widetilde{Cohom_C}(C',\matholdcal M))\arrow{d}{(A,\gamma,\widetilde{Cohom_C}(C',\matholdcal M))}\arrow{r}{(A,C,A,\zeta)}&(A,C,A,\matholdcal M')\arrow{d}{(A,\gamma,\matholdcal M')}\\
					&(A,C',\matholdcal M) \arrow{r}
					{(A,\epsilon',C',\matholdcal M)}\arrow{rd}{id} &(A,C',C',\matholdcal M)\arrow{d}{(A',\Delta',\matholdcal M)}\arrow{r}{(A,C',p^{\matholdcal M})}&(A,C',\widetilde{Cohom_C}(C',\matholdcal M))\arrow{r}{(A,C',\zeta)}\arrow{d}{(A,\pi_{\widetilde{Cohom_C}(C',\matholdcal M)})}&(A,C',\matholdcal M')\arrow{d}{(A,\pi_{ \matholdcal M'})}\\
					(A,\matholdcal M)\arrow{rr}{(A,\epsilon',\matholdcal M)}&&(A,C',\matholdcal M)\arrow{r}{(A,p^{\matholdcal M})}&(A,\widetilde{Cohom_C}(C',\matholdcal M))\arrow{r}{(A,\zeta)}&(A,\matholdcal M')
				\end{tikzcd}
			\end{equation*}
		\end{small}
		Here we note that
		\begin{equation*}
		\begin{array}{ll}
			(A,\gamma,\matholdcal M)\circ (A,C,\mu_{ \matholdcal M})&=(A,\gamma,\matholdcal M)\circ (A,\pi_{(C,A,\matholdcal M)})\circ (A,\epsilon,C,A,\matholdcal M)\circ(A,C,\mu_{ \matholdcal M})\\ 
			&=(A,\gamma,\matholdcal M)\circ (A,\epsilon,A,\matholdcal M)\circ(A,\pi_{(A,\matholdcal M)})\circ(A,C,\mu_{ \matholdcal M})\\ 
			&=(A,\epsilon',\matholdcal M)\circ (A,\eta,\matholdcal M)\circ (A,\pi_{(A,\matholdcal M)})\circ (A,C,\mu_{\matholdcal M})\qquad \qquad \mbox{(by \eqref{eq4.2})}\\ 
			&=(A,\epsilon',\matholdcal M)\circ(A,\pi_{\matholdcal M})\circ (A,C,\eta,\matholdcal M)\circ (A,C,\mu_{\matholdcal M})\\ 
			&=(A,\epsilon',\matholdcal M)\circ(A,\pi_{\matholdcal M})\\
			\end{array}
		\end{equation*}
		It therefore follows from the diagram above that
		\begin{equation*}
		\begin{array}{ll}
			\xi\circ \pi_{\matholdcal M}&=(A,\zeta)\circ(A,p^{\matholdcal M})\circ(A,\epsilon',\matholdcal M)\circ\mu_{\matholdcal M}\circ\pi_{\matholdcal M}\\ 
			&=(A,\zeta)\circ(A,p^{\matholdcal M})\circ(A,\epsilon',\matholdcal M)\circ(A,\pi_{\matholdcal M})\circ(\psi,\matholdcal M)\circ(C,\mu_{\matholdcal M})\qquad
			\qquad \mbox{(by \eqref{eq3.9})}\\ 
			&=(A,\zeta)\circ(A,p^{\matholdcal M})\circ(A,\epsilon',\matholdcal M)\circ(A,\pi_{\matholdcal M})\circ\mu_{(C,\matholdcal M)}\\ 
			&=(A,\zeta)\circ(A,p^{\matholdcal M})\circ(A,\gamma,\matholdcal M)\circ(A,C,\mu_{\matholdcal M})\circ \mu_{(C,\matholdcal M)}\\ 
			&=(A,\pi_{\matholdcal M'})\circ(A,\gamma,\matholdcal M')\circ\mu_{(C,A,\matholdcal M')}\circ(C,A,\zeta)\circ(C,A,p^{\matholdcal M})\circ(C,A,\epsilon',\matholdcal M)\circ(C,\mu_{\matholdcal M})\\ 
			&=(A,\pi_{\matholdcal M'})\circ(A,\gamma,\matholdcal M')\circ(\psi,A,\matholdcal M')\circ(C,\mu,\matholdcal M')\circ(C,A,\zeta)\circ(C,A,p^{\matholdcal M})\circ(C,A,\epsilon',\matholdcal M)\circ(C,\mu_{\matholdcal M})\\ 
			&=\pi_{(A,\matholdcal M')}\circ (C,\xi)\\
			\end{array}
		\end{equation*}
		Hence, $\xi:\matholdcal M\longrightarrow (A,\matholdcal M')$ is a morphism in $_{\hspace{1em}A}^{[C,-]}\mathfrak S(\psi)$.
		To show that $Im(\xi)$ lies inside the kernel $\widetilde{Hom_{A'}}(A,\matholdcal M')$ in \eqref{419uy}, we have to show that $s_{\matholdcal M'}\circ \xi=0$. We first note that
		\begin{align}\label{eq4.11}
			s_{\matholdcal M'}\circ \xi&=\big((A,\mu_{\matholdcal M'})\circ(A,\zeta)\circ(A,p^{\matholdcal M})\circ(A,\epsilon',\matholdcal M)\circ\mu_{\matholdcal M}\big)\\ \notag
			&\quad -\big((A,A',\pi_{\matholdcal M'})\circ(A,\alpha,\matholdcal M')\circ(\mu,\matholdcal M')\circ(A,\zeta)\circ(A,p^{\matholdcal M})\circ (A,\epsilon',\matholdcal M)\circ \mu_{\matholdcal M}\big) \notag
		\end{align}
		Since $p^{\matholdcal M}$ and $\zeta$ are morphisms in $_{\hspace{1em}A'}^{[C',-]}\mathfrak S(\psi')$, we
		consider the following commutative diagram
		\begin{small}
			\begin{equation}\label{eq4.12}
				\begin{tikzcd}[row sep=2.8em, column sep = 4.8em]
					\matholdcal M\arrow{r}{\mu_{\matholdcal M}}
					& (A,\matholdcal M)\arrow{r}{(A,\epsilon',\matholdcal M)} &(A,C',\matholdcal M)\arrow{d}{(\mu_,C'\matholdcal M)}\arrow{r}{(A,p^{\matholdcal M})}&(A,\widetilde{Cohom_C}(C',\matholdcal M))\arrow{d}{(\mu_,\widetilde{Cohom_C}(C',\matholdcal M))}\arrow{r}{(A,\zeta)}&(A,\matholdcal M')\arrow{d}{(\mu,\matholdcal M')}\\
					&&(A,A,C',\matholdcal M)\arrow{r}{(A,A,p^{\matholdcal M})}\arrow{d}{(A,\alpha,C',\matholdcal M)}&(A,A,\widetilde{Cohom_C}(C',\matholdcal M))\arrow{d}{(A,\alpha,\widetilde{Cohom_C}(C',\matholdcal M)}\arrow{r}{(A,A,\zeta)}&(A,A,\matholdcal M')\arrow{d}{(A,\alpha,\matholdcal M')}\\
					&&(A,A',C',C',\matholdcal M) \arrow{r}
					{(A,A',C',p^{\matholdcal M})}\arrow{d}{(A,A',\Delta',\matholdcal M)} &(A,A',C',\widetilde{Cohom_C}(C',\matholdcal M))\arrow{d}{(A,A',\pi_{\widetilde{Cohom_C}(c',\matholdcal M)})}\arrow{r}{(A,A',C',\zeta)}&(A,A',C',\matholdcal M')\arrow{d}{(A,A',\pi_{\matholdcal M'})}\\
					&&(A,A',C',\matholdcal M)\arrow{r}{(A,A',p^{\matholdcal M})}&(A,A',\widetilde{Cohom_C}(C',\matholdcal M))\arrow{r}{(A,A',\zeta)}&(A,A',\matholdcal M')
				\end{tikzcd}
			\end{equation}
		\end{small}
		Further, we note that 
		\begin{equation}\label{eq4.13}
		\begin{array}{lr}
			(A,A',\Delta',\matholdcal M)\circ (A,\alpha,C',\matholdcal M)\circ (\mu,C',\matholdcal M)&\\ =	(A,A',\Delta',\matholdcal M)\circ (A,A',\eta',C',C',\matholdcal M)\circ(A,\mu',C',C',\matholdcal M)\circ (A,\alpha,C',\matholdcal M)\circ (\mu,C',\matholdcal M)&\\
			=(A,A',\Delta',\matholdcal M)\circ (A,\mu_{(C',C',\matholdcal M)})\circ (A,\eta',C',C',\matholdcal M)\circ (A,\alpha,C',\matholdcal M)\circ (\mu,C',\matholdcal M)&\\
			=(A,A',\Delta',\matholdcal M)\circ (A,\mu_{(C',C',\matholdcal M)})\circ(A,\epsilon',C',\matholdcal M)\circ (A,\eta,C',\matholdcal M)\circ (\mu,C',\matholdcal M) & \mbox{( by \eqref{eq4.1})}\\
			=(A,A',\Delta',\matholdcal M)\circ (A,\mu_{(C',C',\matholdcal M)})\circ(A,\epsilon',C',\matholdcal M)& \\
			=(A,A',\Delta',\matholdcal M)\circ(A,A',\epsilon',C',\matholdcal M)\circ (A,\mu_{(C',\matholdcal M)})&\\
			=(A,A',\Delta',\matholdcal M)\circ(A,A',C',\epsilon',\matholdcal M)\circ (A,\mu_{(C',\matholdcal M)})&\\
			=(A,A',\Delta',\matholdcal M)\circ(A,\mu_{(C',C',\matholdcal M)})\circ (A,C',\epsilon',\matholdcal M)&\\
			=(A,A',\Delta',\matholdcal M)\circ(A,\mu_{(C',C',\matholdcal M)})\circ (A,C',\epsilon',\matholdcal M)\circ (A,C',\eta,\matholdcal M)\circ (A,C',\mu_{\matholdcal M})&\\
			=(A,A',\Delta',\matholdcal M)\circ(A,\mu_{(C',C',\matholdcal M)})\circ(A,C',C',\eta',\matholdcal M)\circ (A,C',\alpha,\matholdcal M)\circ (A,C',\mu_{\matholdcal M})&  \mbox{( by \eqref{eq4.1})}\\
			=(A,A',\Delta',\matholdcal M)\circ(A,\psi',C',\matholdcal M)\circ (A,C',\mu_{(C',\matholdcal M)})\circ (A,C',C',\eta',\matholdcal M)\circ (A,C',\alpha,\matholdcal M)\circ (A,C',\mu_{\matholdcal M})&\\
			=(A,A',\Delta',\matholdcal M)\circ(A,\psi',C',\matholdcal M)\circ (A,C',\eta',C',\matholdcal  M)\circ (A,C',\mu',C',\matholdcal M)\circ (A,C',\alpha,\matholdcal M)\circ (A,C',\mu_{\matholdcal M})&\\ 
			=(A,A',\Delta',\matholdcal M)\circ(A,\psi',C',\matholdcal M)\circ (A,C',\alpha,\matholdcal M)\circ (A,C',\mu_{\matholdcal M})&\\
			\end{array}
		\end{equation}
		Consequently, the second term in \eqref{eq4.11} can be expressed as
		\begin{equation*}
		\begin{array}{ll}
			(A,A',\pi_{\matholdcal M'})\circ(A,\alpha,\matholdcal M')\circ(\mu,\matholdcal M')\circ(A,\zeta)\circ(A,p^{\matholdcal M})\circ(A,\epsilon',\matholdcal M)\circ\mu_{\matholdcal M}&\\ 
			=(A,A',\zeta)\circ (A,A',p^{\matholdcal M})\circ (A,A',\Delta',\matholdcal M)\circ (A,\alpha,C',\matholdcal M)\circ (\mu,C',\matholdcal M)\circ (A,\epsilon',\matholdcal M)\circ\mu_{\matholdcal M}& \mbox{(by \eqref{eq4.12})}\\ 
			=(A,A',\zeta)\circ (A,A',p^{\matholdcal M})\circ (A,A',\Delta',\matholdcal M)\circ (A,\psi',C',\matholdcal M)\circ (A,C',\alpha,\matholdcal M)\circ (A,C',\mu_{\matholdcal M})\circ (A,\epsilon',\matholdcal M)\circ\mu_{\matholdcal M}& \mbox{(by \eqref{eq4.13})}\\ 
			=(A,A',\zeta)\circ (A,A',p^{\matholdcal M})\circ(A,\mu_{(C',\matholdcal M)})\circ (A,\epsilon',\matholdcal M)\circ\mu_{\matholdcal M}&\\
			=(A,A',\zeta)\circ(A,\mu_{\widetilde{Cohom_C}(C',\matholdcal M)})\circ (A,p^{\matholdcal M})\circ (A,\epsilon',\matholdcal M)\circ\mu_{\matholdcal M}&\\
			=(A,\mu_{\matholdcal M'})\circ(A,\zeta)\circ (A,p^{\matholdcal M})\circ (A,\epsilon',\matholdcal M)\circ\mu_{\matholdcal M}& \\
			\end{array}
		\end{equation*}
		It is now clear from \eqref{eq4.11} that $s_{\matholdcal M'}\circ \xi=0$. Hence, $Im(\xi)\subseteq \widetilde{Hom_{A'}}(A,\matholdcal M')$.
		
		\smallskip
		Conversely, we consider a morphism $\xi:\matholdcal M\longrightarrow \widetilde{Hom_{A'}}(A,\matholdcal M')$ in $_{\hspace{1em}A}^{[C,-]}\mathfrak S(\psi)$. Then, we define   $\zeta:(C',\matholdcal M)\longrightarrow \matholdcal M'$ in $_{\hspace{1em}A'}^{[C',-]}\mathfrak S(\psi')$ to be the following composition
		\begin{equation}\label{424da}
			(C',\matholdcal M)\xrightarrow{(C',\xi)}(C',\widetilde{Hom_{A'}}(A,\matholdcal M'))\xrightarrow{(C',\iota_{\matholdcal M'})}(C',A,\matholdcal M')\xrightarrow{(C',\eta,\matholdcal M')}(C',\matholdcal M')\xrightarrow{\pi_{\matholdcal M'}}\matholdcal M'
		\end{equation}
		We can now check that $\zeta\circ s^{\matholdcal M}=0$, whence it follows from \eqref{418tf} that   $\zeta$ factors through $ \widetilde{Cohom_C}(C', \matholdcal M)$. We can also verify that these two associations are inverse to each other. This proves the result.
	\end{proof}

	We will now give conditions for the  pair $\big(\widetilde{Cohom_C}(C',-),\widetilde{Hom_{A'}}(A,-)\big)$ in Proposition \ref{P4.4} to be an adjoint equivalence of categories. For this, we begin by describing the unit and counit maps of the adjoint pair $\big(\widetilde{Cohom_C}(C',-),\widetilde{Hom_{A'}}(A,-)\big).$ 
	For any $(\matholdcal M,\pi_{ \matholdcal M},\mu_{ \matholdcal M}) \in~_{\hspace{1em}A}^{[C,-]}\mathfrak S(\psi)$, let $\Psi_{\matholdcal M}: \matholdcal M \longrightarrow \widetilde{Hom_{A'}}(A, \widetilde{Cohom_C}(C', \matholdcal M))$ denote the unit map of adjunction given by
	\begin{equation}\label{421w}
		\matholdcal M\xrightarrow{\mu_{\matholdcal M}}(A,\matholdcal M)\xrightarrow{(A,\epsilon',\matholdcal M)}(A,C',\matholdcal M)\xrightarrow{(A,p^{\matholdcal M})}\widetilde{Hom_{A'}}(A, \widetilde{Cohom_C}(C', \matholdcal M))\subseteq (A, \widetilde{Cohom_C}(C', \matholdcal M))
	\end{equation} By abuse of notation, we denote the last map in \eqref{421w} by $(A,p^{\matholdcal M})$, since we know from the definition in \eqref{421qe} and the proof of 
	Proposition \ref{P4.4} that the composition in \eqref{421w} takes values in the kernel $\widetilde{Hom_{A'}}(A, \widetilde{Cohom_C}(C', \matholdcal M))$. 
	
	\smallskip
	 For  $(\matholdcal M',\pi_{ \matholdcal M'},\mu_{ \matholdcal M'}) \in~_{\hspace{1em}A'}^{[C',-]}\mathfrak S(\psi')$, let  $\Phi_{\matholdcal M'}: \widetilde{Cohom_C}(C', \widetilde{Hom_{A'}}(A, \matholdcal M')) \longrightarrow \matholdcal M'$ denote the counit of the adjunction obtained by 
	 \begin{equation}\label{426da}
			\widetilde{Cohom_C}(C',\widetilde{Hom_{A'}}(A,\matholdcal M'))\xrightarrow{(C',\iota_{\matholdcal M'})}(C',A,\matholdcal M')\xrightarrow{(C',\eta,\matholdcal M')}(C',\matholdcal M')\xrightarrow{\pi_{\matholdcal M'}}\matholdcal M'
		\end{equation} Again by abuse of notation, we denote the first map in \eqref{426da} by $(C',\iota_{\matholdcal M'})$, since we know from the definition in \eqref{424da} and the proof of 
	Proposition \ref{P4.4} that the composition in \eqref{426da} descends to the cokernel $\widetilde{Cohom_C}(C',\widetilde{Hom_{A'}}(A,\matholdcal M'))$.
	 
	\begin{defn}\label{D4.5}
		Let $(\alpha, \gamma)$ be a measuring from $(A',C',\psi')$ to $(A,C,\psi)$ and let   $\mathfrak S$ be a $k$-linear Grothendieck category.	 Then, we say that $(\alpha,\gamma)$ is an $\mathfrak S$-contra-Galois measuring if for every $\matholdcal M\in\mathfrak S$ , the unit $\Psi_{(C,A,\matholdcal M)}$ and the counit $\Phi_{(A',C',\matholdcal M)}$ are  isomorphisms.
	\end{defn}
	\begin{thm}\label{P4.6}
		Let $\mathfrak S$ be a $k$-linear Grothendieck category and let $(\alpha, \gamma)$ be an $\mathfrak S$-contra-Galois measuring from $(A',C',\psi')$ to $(A,C,\psi)$. Suppose that
		
		\smallskip
		(i) The functors $\widetilde{Cohom_C}(C',-):_{\hspace{1em}A}^{[C,-]}\mathfrak S(\psi)\longrightarrow _{\hspace{1em}A'}^{[C',-]}\mathfrak S(\psi')$ and $\widetilde{Hom_{A'}}(A,-):_{\hspace{1em}A'}^{[C',-]}\mathfrak S(\psi')\longrightarrow _{\hspace{1em}A}^{[C,-]}\mathfrak S(\psi)$ are exact. 
		
		\smallskip
		(ii) The functor $(A',-):\mathfrak S\longrightarrow \mathfrak S$ is  exact. 
		
		\smallskip Then, the functors $\widetilde{Cohom_C}(C',-)$ and $\widetilde{Hom_{A'}}(A,-)$ form an adjoint equivalence of categories.
	\end{thm}
	\begin{proof}
		We will show that the unit and the counit maps  for the adjoint pair $(\widetilde{Cohom_C}(C',-),\widetilde{Hom_{A'}}(A,-))$ are isomorphisms. For any $(\matholdcal{N'},\pi_{\matholdcal N'})\in~^{[C',-]}\mathfrak S$, we first show that $\Phi_{(A', \matholdcal{N'})}$ is an isomorphism in $_{\hspace{1em}A'}^{[C',-]}\mathfrak S(\psi')$. 
For this, we consider the following commutative diagram in $_{\hspace{1em}A'}^{[C',-]}\mathfrak S(\psi')$
		\begin{equation}\label{428d}
		\small
			\begin{tikzcd}[row sep=1.8em, column sep = 3em]
				\widetilde{Cohom_C}(C',\widetilde{Hom_{A'}}(A,(A',C',C',\matholdcal N'))) \arrow{r}{} \arrow{d}{ \Phi_{(A',C',C',\matholdcal N')} }
				& \widetilde{Cohom_C}(C',\widetilde{Hom_{A'}}(A,(A',C',\matholdcal N'))) \arrow{r}{}\arrow{d}{\Phi_{(A',C',\matholdcal N')}}&\widetilde{Cohom_C}(C',\widetilde{Hom_{A'}}(A,(A',\matholdcal N')))\arrow{d}{\Phi_{(A',\matholdcal N')}}\arrow{r}{}&0\\ 
				(A',C',C',\matholdcal N')\arrow{r}
				{(A',\tilde{\tau}_{\matholdcal N'})} &(A',C',\matholdcal N')\arrow{r}{(A',\pi_{\matholdcal N'})}&(A',\matholdcal N')\arrow{r}{}&0
			\end{tikzcd}
		\end{equation}
where $\tilde{\tau}_{\matholdcal N'}:(C',C',\matholdcal N')\xrightarrow{(\Delta', \matholdcal N')-(C', \pi_{\matholdcal N'})}(C', \matholdcal N')$ and the top row is obtained by applying
$\widetilde{Cohom_C}(C',\widetilde{Hom_{A'}}(A,-))$ to the bottom row. We first claim that both rows in the diagram (\ref{428d}) are exact. Since $(C',-)$ is exact,  $^{[C',-]}\mathfrak S$ is an abelian category with cokernels computed in $\mathfrak S$. Hence, $\matholdcal N'\in~^{[C',-]}\mathfrak S$ fits into the following exact sequence  in $^{[C',-]}\mathfrak S$. 
		\begin{equation}\label{429e}
			(C',C',\matholdcal N')\xrightarrow
			{\tilde{\tau}_{\matholdcal N'}} (C',\matholdcal N')\xrightarrow{\pi_{\matholdcal N'}}\matholdcal N'\longrightarrow 0
		\end{equation}
		By Theorem \ref{T3.13}, the category  $_{\hspace{1em}A'}^{[C',-]}\mathfrak S(\psi')$ is also abelian, and cokernels in $_{\hspace{1em}A'}^{[C',-]}\mathfrak S(\psi')$  are computed in $\mathfrak S$. Applying the exact functor $(A', -)$ to the sequence in (\ref{429e}) gives the following exact sequence  in $_{\hspace{1em}A'}^{[C',-]}\mathfrak S(\psi')$.
		 \begin{equation}\label{430e}
		 		(A',C',C',\matholdcal N')\xrightarrow
		 	{(A',\tilde{\tau}_{\matholdcal N'})} (A',C',\matholdcal N')\xrightarrow{(A',\pi_{\matholdcal N'})}(A',\matholdcal N')\longrightarrow 0
		 \end{equation} 
		 Hence, the bottom row in \eqref{428d} is exact. By assumption,  $\widetilde{Cohom_C}(C',-)$ and $\widetilde{Hom_{A'}}(A,-)$ are exact and it follows that the top row in 
		 \eqref{428d} is also exact.  
		 
		 \smallskip
		Since $(\alpha, \gamma)$ is an $\mathfrak S$-contra-Galois measuring, both $\Phi_{(A', C', \matholdcal N')}$ and $\Phi_{(A', C', C', \matholdcal N')} $ are  isomorphisms in $_{\hspace{1em}A'}^{[C',-]}\mathfrak S(\psi')$. From the diagram (\ref{428d}), it now follows that $\Phi_{(A',\matholdcal N')}$ is also an isomorphism in $_{\hspace{1em}A'}^{[C',-]}\mathfrak S(\psi')$. 
		
		\smallskip
		We now take $(\matholdcal M',\pi_{ \matholdcal M'},\mu_{ \matholdcal M'})\in~_{\hspace{1em}A'}^{[C',-]}\mathfrak S(\psi')$. To show that $\Phi_{\matholdcal M'}$ is a natural isomorphism, we consider the following commutative diagram in $_{\hspace{1em}A'}^{[C',-]}\mathfrak S(\psi')$
		\begin{equation}\label{432d}
		\small
			\begin{tikzcd}[row sep=1.8em, column sep = 3.8em]
				0\arrow{r}{}&\widetilde{Cohom_C}(C',\widetilde{Hom_{A'}}(A,\matholdcal M'))\arrow{r}{} \arrow{d}{ \Phi_{\matholdcal M'}}
				&\widetilde{Cohom_C}(C',\widetilde{Hom_{A'}}(A,(A',\matholdcal M'))) \arrow{r}{}\arrow{d}{\Phi_{(A',\matholdcal M')}}& \widetilde{Cohom_C}(C',\widetilde{Hom_{A'}}(A,(A',A',\matholdcal M'))) \arrow{d}{\Phi_{(A',A',\matholdcal M')}}\\ 
				0\arrow{r}{}&\matholdcal M' \arrow{r}
				{\mu_{\matholdcal M'}}&(A',\matholdcal M')\arrow{r}{(A',\mu_{\matholdcal M'})-(\mu',\matholdcal M')}&(A',A',\matholdcal M')
			\end{tikzcd}
		\end{equation}
where  the top row is obtained by applying
$\widetilde{Cohom_C}(C',\widetilde{Hom_{A'}}(A,-))$ to the bottom row.
		 Since $\matholdcal M'\in {_{A'}}\mathfrak S$, it can be described as the kernel of the morphism $(A',\mu_{\matholdcal M'})-(\mu',\matholdcal M')$ in $\mathfrak S$. Since  kernels in $_{\hspace{1em}A'}^{[C',-]}\mathfrak S(\psi')$ are computed in $\mathfrak S$, the bottom row in \eqref{432d} is exact.   By assumption,  $\widetilde{Cohom_C}(C',-)$ and $\widetilde{Hom_{A'}}(A,-)$ are both exact and it follows that the top row in 
		 \eqref{432d} is also exact.   
		 
		 \smallskip
		Since $\matholdcal M'$, $(A',\matholdcal M')\in  {^{[C',-]}\mathfrak S}$, it follows from the above that both $\Phi_{(A',\matholdcal M')}$ and $\Phi_{(A',A',\matholdcal M')}$ are  isomorphisms in $_{\hspace{1em}A'}^{[C',-]}\mathfrak S(\psi')$. From the diagram (\ref{432d}), it is now clear that  $\Phi_{\matholdcal M'}:\widetilde{Cohom_C}(C',\widetilde{Hom_{A'}}(A,\matholdcal M'))\longrightarrow \matholdcal M'$ is an  isomorphism in $_{\hspace{1em}A'}^{[C',-]}\mathfrak S(\psi')$. 
		  
		  \smallskip
		  %Likewise, given the functors $\widetilde{Hom_{A'}}(A,-)$, $\widetilde{Cohom_C}(C',-)$ and $(C,-)$ are exact, it can be shown that 
		  We now sketch the proof of the fact that for any $(\matholdcal M,\pi_{ \matholdcal M},\mu_{ \matholdcal M})\in~_{\hspace{1em}A}^{[C,-]}\mathfrak S(\psi)$, $\Psi_{\matholdcal M}$ is an  isomorphism in $_{\hspace{1em}A}^{[C,-]}\mathfrak S(\psi)$. Let $(\matholdcal P,\mu_{\matholdcal P})$ be any object in $_A\mathfrak S$. Since the functors 
		  $(C,-)$,  $\widetilde{Cohom_C}(C',-)$  and $\widetilde{Hom_{A'}}(A,-)$ are exact, we have the following commutative diagram with exact rows in $_{\hspace{1em}A}^{[C,-]}\mathfrak S(\psi)$
		  \begin{equation}\label{435d}
		  \small
		  		\begin{tikzcd}[row sep=1.8em, column sep = 2em]
		  		0\arrow{r}{}&(C,\matholdcal P)\arrow{r}
		  		{(C,\mu_{\matholdcal P})}\arrow{d}{ \Psi_{(C,\matholdcal P)} } &(C,A,\matholdcal P)\arrow{r}{(C,(A,\mu_{\matholdcal P})-(\mu,\matholdcal P))}\arrow{d}{\Psi_{(C,A,\matholdcal P)}}&(C,A,A,\matholdcal P)\arrow{d}{\Psi_{(C,A,A,\matholdcal P)}}\\
		  		0\arrow{r}{}&\widetilde{Hom_{A'}}(A,\widetilde{Cohom_C}(C',(C,\matholdcal P))) \arrow{r}{} 
		  		& \widetilde{Hom_{A'}}(A,\widetilde{Cohom_C}(C',(C,A,\matholdcal P))) \arrow{r}{}&\widetilde{Hom_{A'}}(A,\widetilde{Cohom_C}(C',(C,A,A,\matholdcal P)))
		  	\end{tikzcd}
		  \end{equation}
		  Since $(\alpha,\gamma)$ is an $\mathfrak S$-contra-Galois measuring, we see that $\Psi_{(C,A,\matholdcal P)}$ and $\Psi_{(C,A,A,\matholdcal P)}$  are both 
		  isomorphisms. Hence, it follows from \eqref{435d} that   $\Psi_{(C,\matholdcal P)}:(C,\matholdcal P)\longrightarrow \widetilde{Hom_{A'}}(A, \widetilde{Cohom_C}(C', (C,\matholdcal P)))$ is an  isomorphism in $_{\hspace{1em}A}^{[C,-]}\mathfrak S(\psi)$. In general, for any $\matholdcal M\in ~_{\hspace{1em}A}^{[C,-]}\mathfrak S(\psi)$,    we can  consider the following commutative diagram with exact rows
		  \begin{equation}\label{436d}
		  \small
		  		\begin{tikzcd}[row sep=1.8em, column sep = 3em]
		  		(C,C,\matholdcal M)\arrow{r}
		  		{(C,\pi_{\matholdcal M})-(\Delta,\matholdcal M)}\arrow{d}{ \Psi_{(C,C,\matholdcal M)} } &(C,\matholdcal M)\arrow{r}{\pi_{\matholdcal M}}\arrow{d}{\Psi_{(C,\matholdcal M)}}&\matholdcal M\arrow{r}{}\arrow{d}{\Psi_{\matholdcal M}}\arrow{r}{}&0\\
		  		\widetilde{Hom_{A'}}(A,\widetilde{Cohom_C}(C',(C,C,\matholdcal M))) \arrow{r}{} 
		  		& \widetilde{Hom_{A'}}(A,\widetilde{Cohom_C}(C',(C,\matholdcal M))) \arrow{r}{}&\widetilde{Hom_{A'}}(A,\widetilde{Cohom_C}(C',\matholdcal M))\arrow{r}{}&0
		  	\end{tikzcd}
		  \end{equation}
		  Since $\matholdcal M$, $(C,\matholdcal M)\in ~_A\mathfrak S$, it follows from the above that $\Psi_{(C,\matholdcal M)}$ and $\Psi_{(C,C,\matholdcal M)}$ are both isomorphisms in $_{\hspace{1em}A}^{[C,-]}\mathfrak S(\psi)$.  From \eqref{436d}, it is now clear that $\Psi_{\matholdcal M}$ is an isomorphism in $_{\hspace{1em}A}^{[C,-]}\mathfrak S(\psi)$. 
	\end{proof}
	\begin{Thm}\label{T4.7}
		Let $\mathfrak S$ be a $k$-linear Grothendieck category.	Let $(\alpha, \gamma)$ be a measuring from $(A',C',\psi')$ to $(A,C,\psi)$ and suppose that  the functor $(A',-):\mathfrak S\longrightarrow \mathfrak S$ is exact. Then, the following are equivalent:
		
		\smallskip
		(i) The functors $\widetilde{Cohom_C}(C',-)$ and $\widetilde{Hom_{A'}}(A,-)$ form an adjoint equivalence between the categories  $_{\hspace{1em}A}^{[C,-]}\mathfrak S(\psi)$ and  $_{\hspace{1em}A'}^{[C',-]}\mathfrak S(\psi')$.
		
		\smallskip
		(ii) $(\alpha, \gamma)$ is an $\mathfrak S$-contra-Galois measuring and the functors $\widetilde{Cohom_C}(C',-)$ and $\widetilde{Hom_{A'}}(A,-)$ are both exact.
		
		\smallskip
		(iii)  $(\alpha, \gamma)$ is an $\mathfrak S$-contra-Galois measuring, $\widetilde{Hom_{A'}}(A,-)$ is exact and $\widetilde{Cohom_C}(C',-)$ is faithfully exact.
		
		\smallskip
		(iv) $(\alpha, \gamma)$ is an $\mathfrak S$-contra-Galois measuring, $\widetilde{Cohom_C}(C',-)$ is exact and $\widetilde{Hom_{A'}}(A,-)$ is faithfully exact.
	\end{Thm}
	\begin{proof}
		$(i) \Rightarrow (ii), (iii)~\textup{and}~(iv)$. We recall that a functor is said to be faithfully exact if it preserves and reflects exact sequences. Since the functors  $\widetilde{Cohom_C}(C',-)$ and $\widetilde{Hom_{A'}}(A,-)$ form an adjoint equivalence, it follows from \cite[p.~93]{Mac} that for each $(\matholdcal M,\pi_{ \matholdcal M},\mu_{ \matholdcal M}) \in~_{\hspace{1em}A}^{[C,-]}\mathfrak S(\psi)$, and $(\matholdcal M',\pi_{ \matholdcal M'},\mu_{ \matholdcal M'}) \in~_{\hspace{1em}A'}^{[C',-]}\mathfrak S(\psi')$, the unit $\Psi_{\matholdcal M}: \matholdcal M \longrightarrow \widetilde{Hom_{A'}}(A, \widetilde{Cohom_C}(C', \matholdcal M))$ and the counit $\Phi_{\matholdcal M'}: \widetilde{Cohom_C}(C', \widetilde{Hom_{A'}}(A, \matholdcal M')) \longrightarrow \matholdcal M'$  maps are natural isomorphisms. From Definition \ref{D4.5}, it is now clear that $(\alpha, \gamma)$ is an $\mathfrak S$-contra-Galois measuring.  
		
		\smallskip Furthermore,  since $(\widetilde{Cohom_C}(C',-),\widetilde{Hom_{A'}}(A,-))$ is an adjoint equivalence, we know that $(\widetilde{Hom_{A'}}(A,-), \widetilde{Cohom_C}(C',-))$ is also a pair of adjoint functors (see \cite[p93]{Mac}). Hence, both the functors $\widetilde{Cohom_C}(C',-)$ and $\widetilde{Hom_{A'}}(A,-)$ are exact. We now check that the functors $\widetilde{Cohom_C}(C',-)$ and $\widetilde{Hom_{A'}}(A,-)$ also reflect exact sequences. 
		For this, we first consider the sequence 
		\begin{equation}\label{435e}
			0\longrightarrow \matholdcal M_1\longrightarrow \matholdcal M_2\longrightarrow \matholdcal M_3\longrightarrow 0
		\end{equation}
		in $_{\hspace{1em}A}^{[C,-]}\mathfrak S(\psi)$ such that the following sequence 
		\begin{equation}\label{436e}
			0\longrightarrow\widetilde{Cohom_C}(C',\matholdcal M_1)\longrightarrow \widetilde{Cohom_C}(C',\matholdcal M_2)\longrightarrow \widetilde{Cohom_C}(C',\matholdcal M_3)\longrightarrow 0
		\end{equation}
		is exact in $_{\hspace{1em}A'}^{[C',-]}\mathfrak S(\psi')$. Since the functor $\widetilde{Hom_{A'}}(A,-):_{\hspace{1em}A'}^{[C',-]}\mathfrak S(\psi')\longrightarrow ~_{\hspace{1em}A}^{[C,-]}\mathfrak S(\psi)$ is exact, the following sequence
		\begin{align}
				0\longrightarrow \widetilde{Hom_{A'}}(A,\widetilde{Cohom_C}(C',\matholdcal M_1))\longrightarrow  \widetilde{Hom_{A'}}(A,\widetilde{Cohom_C}(C',\matholdcal M_2))\longrightarrow  \widetilde{Hom_{A'}}(A,\widetilde{Cohom_C}(C',\matholdcal M_3))\longrightarrow 0	
		\end{align}
is exact in $_{\hspace{1em}A}^{[C,-]}\mathfrak S(\psi)$. As $\Psi_{\matholdcal M}: \matholdcal M \longrightarrow \widetilde{Hom_{A'}}(A, \widetilde{Cohom_C}(C', \matholdcal M))$ is an isomorphism for each $(\matholdcal M,\pi_{ \matholdcal M},\mu_{ \matholdcal M}) \in~_{\hspace{1em}A}^{[C,-]}\mathfrak S(\psi)$, the sequence in (\ref{435e}) must be exact. It follows that the functor $\widetilde{Cohom_C}(C',-):~_{\hspace{1em}A}^{[C,-]}\mathfrak S(\psi)\longrightarrow _{\hspace{1em}A'}^{[C',-]}\mathfrak S(\psi')$ is faithfully exact. Similarly, it can be verified that the functor  $\widetilde{Hom_{A'}}(A,-):~_{\hspace{1em}A'}^{[C',-]}\mathfrak S(\psi')\longrightarrow ~_{\hspace{1em}A}^{[C,-]}\mathfrak S(\psi)$  is faithfully exact.
	
		\smallskip
		$(ii) \Rightarrow (i)$. This follows from Proposition \ref{P4.6}.
		
		\smallskip
		$(iii)\Rightarrow (ii)$, $(iv)\Rightarrow (ii)$. This is clear, as every faithfully exact functor is exact.

	\end{proof}
	\begin{eg}\label{Example4.8}
\emph{Let $C$ be a $k$-coalgebra and let $A$ be a $k$-algebra. Suppose that $A$  is also a right $C$-comodule with structure map $\Delta_A:A\longrightarrow A\otimes C$. Let $B\overset{\iota_{B}}{\hookrightarrow} A$ be the subalgebra of $A$  given by $B:=\{b\in A~|~\forall a\in A, \Delta_A(ba)=b\Delta_A(a)\}.$ We recall from \cite[Definition 2.1]{Tb} that $A$ is said to be a $C$-Galois extension if and only if the canonical left $A$-module and right $C$-comodule map }
\begin{equation} can: A\otimes_B A\xrightarrow{A\otimes_B \Delta_A} A\otimes_B A\otimes C\xrightarrow{\mu\otimes C} A\otimes C
\end{equation} \emph{is bijective. Let $(A,C,\psi)$ be the entwining structure   canonically associated to the $C$-Galois extension (see \cite[Theorem 2.7]{Tb99} for details).   From \cite[\S 3]{Tb}, we know that $(\alpha,\gamma)=(\iota_B,\Delta_A\circ \eta)$ is a measuring from $(B,k,id_B)$ to $(A,C,\psi)$. We will show that if $A$ is a $C$-Galois extension, then $(\alpha,\gamma)$ is an $\mathfrak S$-contra-Galois measuring.  From 
\cite[Example 3.9]{Tb}, we  know that }
\begin{align}  A_0&:=Ker\left(A\xrightarrow{\quad\Delta_A-(\mu \otimes C)(A\otimes \Delta_A)(A\otimes \eta)\quad}A\otimes C\right)=B  \label{4.8dtu}\\
 (A\otimes C)_0&:=Ker\left(A\otimes C\xrightarrow{\quad(A\otimes \Delta)-(((\mu\otimes C)(A\otimes \psi))\otimes C)(A\otimes C\otimes \Delta_A)(A\otimes C\otimes \eta)\quad}A\otimes C\otimes C\right)=A  \label{4.8ud} 
\end{align}
\emph{It may be verified directly that the functor $\widetilde{Hom}_B(A,-):{_B\mathfrak S}\longrightarrow _{\hspace{1em}A}^{[C,-]}\mathfrak S(\psi)$ has a left adjoint given by}
\begin{equation}
\matholdcal N\mapsto \matholdcal N^0:=Coker\left( (C,\matholdcal N)\xrightarrow{\qquad\pi_{\matholdcal N}-(\eta,\matholdcal N)(\Delta_A,\matholdcal N)(C,\mu_{\matholdcal N})\qquad}\matholdcal N\right)
\end{equation}
\emph{for each $(\matholdcal N,\pi_{\matholdcal N},\mu_{\matholdcal N})\in~_{\hspace{1em}A}^{[C,-]}\mathfrak S(\psi)$. By the uniqueness of adjoints, it now follows from 
Proposition \ref{P4.4} that $\widetilde{Cohom_C}(k,\matholdcal N)=\matholdcal N^0$. We now take $\matholdcal M\in \mathfrak S$. Then, $(B,\matholdcal M)\in {_B}\mathfrak S$ and it may be verified that $\widetilde{Hom}_B(A,(B,\matholdcal M))\cong (A,\matholdcal M)$. Accordingly, we see that}
\begin{equation}\label{4.41ce}
\begin{array}{ll}
\widetilde{Cohom_C}(k,\widetilde{Hom}_B(A,(B,\matholdcal M)))&\cong (A,\matholdcal M)^0\\
&=Coker\left((A\otimes C,{\matholdcal M})= (C,A,\matholdcal M)\underset{=(\Delta_A,{\matholdcal M})-(A\otimes \eta,{\matholdcal M})(A\otimes \Delta_A,{\matholdcal M})(\mu\otimes C,{\matholdcal M})}{\xrightarrow{\qquad(\Delta_A,{\matholdcal M})-(\eta,A,\matholdcal M)(\Delta_A,A,\matholdcal M)(C,\mu,{\matholdcal M})\qquad}}(A,\matholdcal M)\right)\\
&=\left(Ker\left(A\xrightarrow{\quad\Delta_A-(\mu \otimes C)(A\otimes \Delta_A)(A\otimes \eta)\quad}A\otimes C\right),\matholdcal M\right)=(A_0,\matholdcal M)\cong (B,\matholdcal M)\\
\end{array} 
\end{equation} \emph{It follows from \eqref{4.41ce} that the counit map $\Phi_{(B,\matholdcal M)}$ induced by the adjunction is an isomorphism. Similarly, for $\matholdcal M\in \mathfrak S$, we know that $(C,A,
\matholdcal M)\in ~_{\hspace{1em}A}^{[C,-]}\mathfrak S(\psi)$. Accordingly, we have}
\begin{equation}\label{442dh}
\begin{array}{l}
\widetilde{Cohom_C}(k,(C,A,
\matholdcal M))=(C,A,
\matholdcal M)^0\\
=Coker\left( (C,C,A,
\matholdcal M)\xrightarrow{\qquad(\Delta,A,
\matholdcal M)-(\eta,C,A,
\matholdcal M)(\Delta_A,C,A,
\matholdcal M)(C,\psi,A,
\matholdcal M)(C,C,\mu,\matholdcal M)\qquad}(C,A,
\matholdcal M)\right)\\
=Coker\left( (A\otimes C\otimes C,
\matholdcal M)\xrightarrow{\qquad(A\otimes \Delta,
\matholdcal M)-(A\otimes C\otimes \eta,
\matholdcal M)(A\otimes C\otimes \Delta_A,
\matholdcal M)(A\otimes \psi\otimes C, 
\matholdcal M)(\mu\otimes C\otimes C,\matholdcal M)\qquad}(A\otimes C,
\matholdcal M)\right)\\
\cong \left(Ker\left(A\otimes C\xrightarrow{\quad(A\otimes \Delta)-(((\mu\otimes C)(A\otimes \psi))\otimes C)(A\otimes C\otimes \Delta_A)(A\otimes C\otimes \eta)\quad}A\otimes C\otimes C\right),\matholdcal M\right)\cong ((A\otimes C)_0,\matholdcal M)\cong (A,\matholdcal M)\\
\end{array}
\end{equation} \emph{Frpm \eqref{442dh}, it follows that}
\begin{equation}\label{443xk}
\widetilde{Hom}_B(A,\widetilde{Cohom_C}(k,(C,A,
\matholdcal M)))\cong \widetilde{Hom}_B(A,(A,\matholdcal M))\cong (A\otimes_BA,\matholdcal M)\underset{\cong}{\xrightarrow{(can^{-1},\matholdcal M)}}(A\otimes C,\matholdcal M)\cong (C,A,
\matholdcal M)
\end{equation}  \emph{It follows from \eqref{443xk} that the unit map $\Psi_{(C,A,
\matholdcal M)}$ induced by the adjunction is an isomorphism.}
\end{eg}  
	\subsection{$\mathfrak S$-Galois measurings  and entwined comodule objects}
	We continue with $k$-algebras $A$, $A'$ and $k$-coalgebras $C$, $C'$. Let $(A,C,\psi)$ and $(A',C',\psi')$ be   entwining structures.  
	\begin{thm}\label{P4.8}
		Let $(\alpha, \gamma)$ be a measuring from $(A',C',\psi')$ to $(A,C,\psi)$. Then, we have the following statements.
		
		\smallskip
		(i) Let $(\mathcal M,\mu_{\mathcal M})\in \mathfrak S_A$. Then, $\mathcal M\otimes C'\in \mathfrak S_{A'}^{C'}(\psi')$ with structure maps
		\begin{equation}\label{441lk}
		\begin{array}{c} 
		\Delta_{\mathcal M\otimes C'}:\mathcal M\otimes C'\xrightarrow{\mathcal M\otimes \Delta'}\mathcal M\otimes C'\otimes C'\\
		\mu_{\mathcal M\otimes C'}: \mathcal M\otimes C'\otimes A'\xrightarrow{\mathcal M \otimes \Delta'\otimes A'}\mathcal M \otimes C'\otimes C'\otimes A'\xrightarrow{\mathcal M\otimes C'\otimes \psi'}\mathcal M\otimes C'\otimes A'\otimes C'\xrightarrow{\mathcal M\otimes \alpha \otimes C'}\mathcal M\otimes A\otimes C'\xrightarrow{\mu_{\mathcal M}\otimes C'}\mathcal M\otimes C'\\
		\end{array}
		\end{equation} 
		
		\smallskip
		(ii) Let $(\mathcal M',\Delta_{\mathcal M'})\in \mathfrak S^{C'}$. Then, $\mathcal M'\otimes A\in \mathfrak S_A^C(\psi)$ with structure maps
		\begin{equation}\label{442lk}
		\begin{array}{c}
		\mu_{\mathcal M'\otimes A}:\mathcal M'\otimes A\otimes A\xrightarrow{\mathcal M'\otimes \mu}\mathcal M'\otimes A\\
		\Delta_{\mathcal M'\otimes A}:\mathcal M'\otimes A\xrightarrow{\Delta_{\mathcal M'}\otimes A} \mathcal M'\otimes C'\otimes A\xrightarrow{\mathcal M'\otimes\gamma \otimes A}\mathcal M'\otimes A\otimes C \otimes A\xrightarrow{\mathcal M'\otimes A \otimes \psi} \mathcal M'\otimes A\otimes A\otimes C\xrightarrow{\mathcal M'\otimes \mu\otimes C}\mathcal M'\otimes A\otimes C
		\end{array}
		\end{equation}

	\end{thm}
	\begin{proof}
	The proof of this is similar to that of \cite[Proposition 3.3]{Tb}. 
	\end{proof}

	\begin{lem}\label{L4.9}
		Let $(\alpha, \gamma)$ be a measuring from $(A',C',\psi')$ to $(A,C,\psi)$. Let   $(\mathcal M,\Delta_{\mathcal M},\mu_{\mathcal M})\in \mathfrak S_A^C(\psi)$ and $(\mathcal M',\Delta_{\mathcal M'},\mu_{\mathcal M'})\in\mathfrak S_{A'}^{C'}(\psi')$. Then,  we have the following statements.
		
		\smallskip
		(i)  Let $t^{\mathcal M}:\mathcal M\otimes C'\longrightarrow \mathcal M\otimes C\otimes C'$ be the morphism given by setting 
		\begin{equation}\label{4411x} t^{\mathcal M}:=(\Delta_{\mathcal M}\otimes C')-(\mu_{\mathcal M}\otimes C\otimes C')\circ (\mathcal M\otimes\gamma\otimes C')\circ(\mathcal M\otimes \Delta')
		\end{equation} Then, $t^{\mathcal M}$ is a morphism in $\mathfrak S_{A'}^{C'}(\psi')$, where $\mathcal M\otimes C'$ and $\mathcal M\otimes C\otimes C'$ are treated as the objects in $\mathfrak S_{A'}^{C'}(\psi')$.

		\smallskip
		(ii) Let $t_{\mathcal M'}:\mathcal M'\otimes A'\otimes A\longrightarrow \mathcal M'\otimes A$ be a morphism given by setting 
		\begin{equation}\label{4412x} t_{\mathcal M'}:=(\mu_{\mathcal M'}\otimes A) -(\mathcal M'\otimes \mu)\circ (\mathcal M'\otimes\alpha\otimes A)\circ(\Delta_{\mathcal M'}\otimes A'\otimes A)
		\end{equation} Then, $t_{\mathcal M'}$ is a morphism in $\mathfrak S_A^C(\psi)$, where $\mathcal M'\otimes A$ and $\mathcal M'\otimes A'\otimes A$ are viewed as the objects in $\mathfrak S_A^C(\psi)$.

	\end{lem}
	\begin{proof}
	The proof of this is similar to that of \cite[Proposition 3.5]{Tb}. 
	\end{proof}

	Let $(\alpha, \gamma)$ be a measuring from $(A',C',\psi')$ to $(A,C,\psi)$. We now define a pair of adjoint functors between the categories of entwined comodule objects in $\mathfrak S$ over $(A,C,\psi)$ and $(A',C',\psi')$. For any $\mathcal M\in \mathfrak S_A^C(\psi)$, it follows by Proposition \ref{P4.8} and Lemma \ref{L4.9} that  $\mathcal M\otimes C'$ and $\mathcal M\otimes C\otimes C'$ are objects in $\mathfrak S_{A'}^{C'}(\psi')$ and that $t^{\mathcal M}$ is a morphism in $\mathfrak S_{A'}^{C'}(\psi')$. Then, we set $\mathcal M \hat{\square}_C C'\subseteq \mathcal M\otimes C'$ in $\mathfrak S_{A'}^{C'}(\psi')$ to be the kernel
	\begin{equation}\label{442x}
		0\longrightarrow \mathcal M \hat{\square}_C C'\xrightarrow {\iota^{\mathcal M}} \mathcal M\otimes C'\xrightarrow{t^{\mathcal M}} \mathcal M\otimes C\otimes C'
	\end{equation}
	which induces a functor \begin{equation}\label{4425x} -\hat{\square}_C C':\mathfrak S_A^C(\psi)\longrightarrow \mathfrak S_{A'}^{C'}(\psi'),\qquad \mathcal M\mapsto \mathcal M \hat{\square}_C C'
	\end{equation} On the other hand, for any $\mathcal M'\in \mathfrak S_{A'}^{C'}(\psi')$, it follows  from Proposition \ref{P4.8} and Lemma \ref{L4.9} that $\mathcal M'\otimes A$ and $\mathcal M'\otimes A'\otimes A$ are objects in $\mathfrak S_A^C(\psi)$ and $t_{\mathcal M'}$ is a morphism in $\mathfrak S_A^C(\psi)$. Then, we set $\mathcal M'\hat{\otimes}_{A'} A\in \mathfrak S_A^C(\psi)$ to be the cokernel
	\begin{equation}\label{443x}
		\mathcal M'\otimes  A'\otimes A\xrightarrow{t_{\mathcal M'}} \mathcal M'\otimes A\xrightarrow{p_{\mathcal M'}} \mathcal M'\hat{\otimes}_{A'} A\longrightarrow 0
	\end{equation}
	This induces a functor 
	\begin{equation}\label{444x}
		-\hat{\otimes}_{A'}A:\mathfrak S_{A'}^{C'}(\psi')\longrightarrow\mathfrak S_A^C(\psi),\qquad \mathcal M'\mapsto \mathcal M'\hat{\otimes}_{A'} A
	\end{equation}
	\begin{thm}\label{P4.10}
		Let $(\alpha, \gamma)$ be a measuring from $(A',C',\psi')$ to $(A,C,\psi)$ and let $\mathfrak S$ be a $k$-linear Grothendieck category. Then, the functor $-\hat{\otimes}_{A'}A:\mathfrak S_{A'}^{C'}(\psi')\longrightarrow\mathfrak S_A^C(\psi)$ is   left adjoint to $-\hat{\square}_C C':\mathfrak S_A^C(\psi)\longrightarrow \mathfrak S_{A'}^{C'}(\psi')$. In other words,  for any $(\mathcal M,\Delta_{\mathcal M},\mu_{\mathcal M})\in\mathfrak S_A^C(\psi)$ and $(\mathcal M',\Delta_{\mathcal M'},\mu_{\mathcal M'})\in \mathfrak S_{A'}^{C'}(\psi')$ we have natural isomorphisms
		\begin{equation}\label{445fc}
			\mathfrak S_A^C(\mathcal M'\hat{\otimes}_{A'}A, \mathcal M)\cong \mathfrak S_{A'}^{C'}(\psi')(\mathcal M', \mathcal M\hat{\square}_CC')
		\end{equation}
		
	\end{thm}
	\begin{proof} We outline the proof of this result, which is similar to that of \cite[Proposition 3.6]{Tb}. Let $\zeta:\mathcal M'\hat{\otimes}_{A'}A\longrightarrow\mathcal M$ be a morphism in $\mathfrak S_A^C(\psi)$. We define a morphism  $\xi:\mathcal M'\longrightarrow \mathcal M\otimes C'$ by
		\begin{equation}\label{445yx}
			\xi:\mathcal M'\xrightarrow{\Delta_{\mathcal M'}} \mathcal M'\otimes C'\xrightarrow{\mathcal M'\otimes\eta\otimes C'}\mathcal M'\otimes A\otimes C'\xrightarrow{p_{\mathcal M'}\otimes C'}\mathcal M'\hat{\otimes}_{A'}A\otimes C'\xrightarrow{\zeta\otimes C'}\mathcal M\otimes C'
		\end{equation} Then, it may be shown that  $\xi$ is a morphism in $\mathfrak S_{A'}^{C'}(\psi')$ and that 
		$t^{\mathcal M}\circ \xi=0$, i.e., $Im(\xi)\subseteq Ker(t^{\mathcal M})= \mathcal M\hat{\square}_C C'$ in $\mathfrak S_{A'}^{C'}(\psi')$. 
		
		\smallskip Conversely, let $\xi:\mathcal M'\longrightarrow \mathcal M\hat{\square}_CC'$ be a morphism in $\mathfrak S_{A'}^{C'}(\psi')$. Then, we can define a morphism $\zeta :\mathcal M'\otimes A\longrightarrow \mathcal M$ by
		\begin{equation}\label{446yx}
			\zeta:\mathcal M'\otimes A\xrightarrow{\xi\otimes A}\mathcal M\hat{\square}_CC'\otimes A\xrightarrow{\iota^{\mathcal M}\otimes A}\mathcal M\otimes C'\otimes A\xrightarrow{\mathcal M\otimes\epsilon'\otimes A}\mathcal M\otimes A\xrightarrow{\mu_{\mathcal M}}\mathcal M
		\end{equation}
		Then, we can show that $\zeta\circ t_{\mathcal M'}=0$, whence it follows that $\zeta$ factors through $\mathcal M'\hat{\otimes}_{A'}A=Coker(t_{\mathcal M'})$. Further, it may be verified that these two associations are inverse to each other.  This proves the result. 
	\end{proof}

	We will now introduce $\mathfrak S$-co-Galois measurings, which we will use to study conditions for the pair $(-\hat{\otimes}_{A'}A,-\hat{\square}_C C')$ 
	of functors in 
	Proposition \ref{P4.10}  to be an adjoint equivalence of categories. First, we note that for any  $(\mathcal M',\Delta_{\mathcal M'},\mu_{\mathcal M'})\in\mathfrak S_{A'}^{C'}(\psi')$, 
	the unit $\Omega_{\mathcal M'}:\mathcal M'\longrightarrow (\mathcal M'\hat{\otimes}_{A'}A)\hat{\square}_{C}C'$  of the adjunction in \eqref{445fc} is given by 	
	\begin{equation}\label{450h}
		\Omega_{\mathcal M'}:\mathcal M'\xrightarrow{\Delta_{\mathcal M'}}\mathcal M'\otimes C'\xrightarrow{\mathcal M'\otimes \eta\otimes C'}\mathcal M'\otimes A\otimes C'\xrightarrow{p_{\mathcal M'}\otimes C'}(\mathcal M'\hat{\otimes}_{A'}A)\hat{\square}_{C}C',
	\end{equation}  By abuse of notation, we have written the last map  in \eqref{450h} as $p_{\mathcal M'}\otimes C'$, since we know from the definition in \eqref{445yx} that the composition 
	in \eqref{450h} factors through the kernel $ (\mathcal M'\hat{\otimes}_{A'}A)\hat{\square}_{C}C'$.  Similarly, for any  $(\mathcal M,\Delta_{\mathcal M},\mu_{\mathcal M})\in\mathfrak S_A^C(\psi)$, the counit $\Upsilon_{\mathcal M}:(\mathcal M\hat{\square}_{C}C')\hat{\otimes}_{A'}A\longrightarrow \mathcal M$  of the adjunction in \eqref{445fc} is given by 
	\begin{equation}\label{451h}
			\Upsilon_{\mathcal M}:(\mathcal M\hat{\square}_CC')\hat\otimes_{A'} A\xrightarrow{\iota^{\mathcal M}\otimes A}\mathcal M\otimes C'\otimes A\xrightarrow{\mathcal M\otimes\epsilon'\otimes A}\mathcal M\otimes A\xrightarrow{\mu_{\mathcal M}}\mathcal M
		\end{equation} Again by abuse of notation, we have written the first map in \eqref{451h} as $\iota^{\mathcal M}\otimes A$, since we know from the definition 
		in \eqref{446yx} that the composition in \eqref{451h} factors through the cokernel $(\mathcal M\hat{\square}_{C}C')\hat{\otimes}_{A'}A$.

	\begin{defn}\label{D4.11}
		Let $\mathfrak S$ be a $k$-linear Grothendieck category and let $(\alpha, \gamma)$ be a measuring from $(A',C',\psi')$ to $(A,C,\psi)$. Then, we say that $(\alpha, \gamma)$ is an $\mathfrak S$-co-Galois measuring if for every $\mathcal M\in\mathfrak S$, the unit $\Omega_{\mathcal M\otimes C'\otimes A'}$ and counit $\Upsilon_{\mathcal M\otimes A\otimes C}$ are isomorphisms.
	\end{defn}
	\begin{thm}\label{P4.12}
		Let $\mathfrak S$ be a $k$-linear Grothendieck category. Let $(\alpha, \gamma)$ be an $\mathfrak S$-co-Galois measuring from $(A',C',\psi')$ to $(A,C,\psi)$ and suppose that
		the functors $-\hat{\otimes}_{A'}A:\mathfrak S_{A'}^{C'}(\psi')\longrightarrow\mathfrak S_A^C(\psi)$ and $-\hat{\square}_C C':\mathfrak S_A^C(\psi)\longrightarrow \mathfrak S_{A'}^{C'}(\psi')$ are exact. Then, the pair $(-\hat{\otimes}_{A'}A,-\hat{\square}_C C')$  of functors  forms an adjoint equivalence of categories.
	\end{thm}
	\begin{proof}
		 Let $(\mathcal N,\mu_{\mathcal N})$ be an object in $\mathfrak S_A$. We first claim that $\Upsilon_{\mathcal N\otimes C}$ is a natural isomorphism in $\mathfrak S_A^C(\psi)$. Since $\mathfrak S_A$ is an abelian category with cokernels computed in $\mathfrak S$, we see that  the following sequence is exact in $\mathfrak S_A$. 
		\begin{equation}\label{453e1}
\mathcal N\otimes A\otimes A\xrightarrow{(\mu_{\mathcal  N}\otimes A)-(\mathcal N\otimes \mu)}\mathcal N\otimes A\xrightarrow{\mu_{\mathcal N}}\mathcal N\longrightarrow 0			
		\end{equation}
Further, we know  that $\mathfrak S_A^C(\psi)$ is an abelian category with kernels and cokernels computed in $\mathfrak S$. Applying the exact functor $-\otimes C$ to the sequence in \eqref{453e1} gives the following exact sequence 
		\begin{equation}\label{453e2}
			\mathcal N\otimes A\otimes A\otimes C\xrightarrow{((\mu_{\mathcal  N}\otimes A) -(\mathcal N\otimes \mu))\otimes C}\mathcal N\otimes A\otimes C\xrightarrow{(\mu_{\mathcal N}\otimes C)}\mathcal N\otimes C\longrightarrow 0
		\end{equation}
		in $\mathfrak S_A^C(\psi)$. By assumption, the functors $-\hat{\otimes}_{A'}A$ and  $-\hat{\square}_C C'$ are exact. Applying these to  the sequence in \eqref{453e2}, we obtain the following  commutative diagram with exact rows in $\mathfrak S_A^C(\psi)$
		\begin{equation}\label{eq4.35}
			\begin{tikzcd}[row sep=1.8em, column sep = 3.8em]
				((\mathcal N\otimes A \otimes A\otimes  C)\hat{\square}_{C}C')\hat{\otimes}_{A'} A \arrow{r}{} \arrow{d}{ \Upsilon_{\mathcal N\otimes A\otimes A\otimes C} }
				& ((\mathcal N\otimes A\otimes C)\hat{\square}_{C}C')\hat{\otimes}_{A'} A \arrow{r}{}\arrow{d}{\Upsilon_{\mathcal N\otimes A\otimes C}}&((\mathcal N\otimes C)\hat{\square}_{C}C')\hat{\otimes}_{A'}A\arrow{d}{\Upsilon_{\mathcal N\otimes C}}\arrow{r}{}&0\\ 
				\mathcal N\otimes A\otimes A\otimes C\arrow{r}
				{(\mu_{\mathcal  N}\otimes A -\mathcal N\otimes \mu)\otimes C} &\mathcal N\otimes  A\otimes C\arrow{r}{\mu_{\mathcal N}\otimes C}&\mathcal N\otimes C\arrow{r}{}&0
			\end{tikzcd}
		\end{equation}
	 Given that $(\alpha, \gamma)$ is an $\mathfrak S$-co-Galois measuring, both $\Upsilon_{\mathcal N\otimes A\otimes C}$ and $\Upsilon_{\mathcal N\otimes A\otimes A\otimes C}$ are  isomorphisms in $\mathfrak S_A^C(\psi)$. Therefore, it follows from the diagram in \eqref{eq4.35} that $\Upsilon_{\mathcal N\otimes C}$ is an isomorphism in $\mathfrak S_A^C(\psi)$.
		
		\smallskip
		We now take any object $(\mathcal M,\Delta_{\mathcal M},\mu_{\mathcal M})\in \mathfrak S_A^C(\psi)$.  Since $\mathfrak S_A^C(\psi)$ is an abelian category with kernels computed in $\mathfrak S$, the following sequence 
		\begin{equation}\label{457e1}
			0\longrightarrow \mathcal M\xrightarrow{\Delta_{\mathcal M}}\mathcal M\otimes C\xrightarrow{(\Delta_{\mathcal M}\otimes C)-(\mathcal M\otimes \Delta)}\mathcal M\otimes C\otimes C
		\end{equation}
		is exact in $\mathfrak S_A^C(\psi)$. Applying the exact functors  $-\hat{\square}_C C'$ and $-\hat{\otimes}_{A'}A$ to the sequence in \eqref{457e1} gives the following  commutative diagram with exact rows in $\mathfrak S_A^C(\psi)$
		\begin{equation}\label{eq4.36}
			\begin{tikzcd}[row sep=1.8em, column sep = 3.8em]
				0\arrow{r}{}&(\mathcal M\hat{\square}_{C}C')\hat{\otimes}_{A'}A \arrow{r}{} \arrow{d}{ \Upsilon_{\mathcal M}}
				&((\mathcal M\otimes C)\hat{\square}_{C}C')\hat{\otimes}_{A'}A \arrow{r}{}\arrow{d}{\Upsilon_{\mathcal M\otimes C}}& ((\mathcal M\otimes C\otimes  C)\hat{\square}_{C}C')\hat{\otimes}_{A'}A \arrow{d}{\Upsilon_{\mathcal M\otimes C\otimes C}}\\ 
				0\arrow{r}{}&\mathcal M \arrow{r}
				{\Delta_{\mathcal M}} &\mathcal M\otimes C\arrow{r}{(\Delta_{\mathcal M}\otimes C)-(\mathcal M\otimes \Delta)}&\mathcal M\otimes C\otimes C
			\end{tikzcd}
		\end{equation}
	 Since $\mathcal M$, $\mathcal M\otimes C\in \mathfrak S_A$,  it follows from the above that both $\Upsilon_{\mathcal M\otimes C}$ and $\Upsilon_{\mathcal M\otimes C\otimes C}$ are isomorphisms in $\mathfrak S_A^C(\psi)$. It follows from \eqref{eq4.36} that $\Upsilon_{\mathcal M}$ is an  isomorphism in $\mathfrak S_A^C(\psi)$.

	 \smallskip
	 We now outline the proof of the fact  that if the functors $-\hat{\square}_CC'$ and $-\hat{\otimes}_{A'}A$ are exact, then for any $\mathcal M'\in\mathfrak S_{A'}^{C'}(\psi')$, the unit $\Omega_{\mathcal M'}$ is an isomorphism in $\mathfrak S_{A'}^{C'}(\psi')$. For this, we first consider any object $(\mathcal P',\Delta_{\mathcal P'})\in \mathfrak S^{C'}$. As the functors $-\hat{\otimes}_{A'}A$ and $-\hat{\square}_CC'$ are exact, we have the following commutative diagram with exact rows in $\mathfrak S_{A'}^{C'}(\psi')$
	 \begin{equation}\label{460d1}
	 	\begin{tikzcd}
	 		0\arrow{r}{}&\mathcal P'\otimes A'\arrow{r}{\Delta_{\mathcal P'}\otimes A'}\arrow{d}{\Omega_{\mathcal P'\otimes A'}}&\mathcal P'\otimes C'\otimes  A'\arrow{r}{((\Delta_{\mathcal P'}\otimes C')-(\mathcal P'\otimes \Delta'))\otimes A'}\arrow{d}{\Omega_{\mathcal P'\otimes C'\otimes A'}}&\mathcal P'\otimes C'\otimes C'\otimes A'\arrow{d}{\Omega_{\mathcal P'\otimes C'\otimes C'\otimes A'}}\\
	 		0\arrow{r}{}&((\mathcal P'\otimes A')\hat{\otimes}_{A'}A)\hat{\square}_{C}C'\arrow{r}{}&((\mathcal P'\otimes C'\otimes  A')\hat{\otimes}_{A'}A)\hat{\square}_{C}C'\arrow{r}{}&((\mathcal P'\otimes C'\otimes C'\otimes  A')\hat{\otimes}_{A'}A)\hat{\square}_{C}C'
	 	\end{tikzcd}
	 \end{equation}
	 Since $(\alpha,\gamma)$ is an $\mathfrak S$-co-Galois measuring, both $\Omega_{\mathcal P'\otimes C'\otimes A'}$ and $\Omega_{\mathcal P'\otimes C'\otimes C'\otimes A'}$ are isomorphisms in $\mathfrak S_{A'}^{C'}(\psi')$. It follows from the diagram in \eqref{460d1} that $\Omega_{\mathcal P'\otimes A'}$ is an  isomorphism in $\mathfrak S_{A'}^{C'}(\psi')$. We now consider any object $(\mathcal M',\Delta_{\mathcal M'},\mu_{\mathcal M'})\in \mathfrak S_{A'}^{C'}(\psi')$. Since the functors $-\hat{\otimes}_{A'}A$ and $-\hat{\square}_CC'$ are exact, we obtain the following commutative diagram with exact rows in $\mathfrak S_{A'}^{C'}(\psi')$
	 \begin{equation}\label{460d2}
	 	\begin{tikzcd}
	 		\mathcal M'\otimes A'\otimes A'\arrow{r}{(\mu_{\mathcal M'}\otimes A'-\mathcal M'\otimes \mu')}\arrow{d}{\Omega_{\mathcal M'\otimes A'\otimes A'}}&\mathcal M'\otimes A'\arrow{r}{\mu_{\mathcal M'}}\arrow{d}{\Omega_{\mathcal M'\otimes A'}}&\mathcal M'\arrow{d}{\Omega_{\mathcal M'}}\arrow{r}{}&0\\
	 	((\mathcal M'\otimes A'\otimes A')\hat{\otimes}_{A'}A)\hat{\square}_{C}C'\arrow{r}{}&((\mathcal M'\otimes A')\hat{\otimes}_{A'}A)\hat{\square}_{C}C'\arrow{r}{}&(\mathcal M'\hat{\otimes}_{A'}A)\hat{\square}_{C}C'\arrow{r}{}&0
	 	\end{tikzcd}
	 \end{equation}
	 Since $\mathcal M'$, $\mathcal M'\otimes A'\in \mathfrak S^{C'}$, it follows from the above that $\Omega_{\mathcal M'\otimes A'}$ and $\Omega_{\mathcal M'\otimes A'\otimes A'}$ are isomorphisms. It now follows from the diagram \eqref{460d2} that $\Omega_{\mathcal M'}:\mathcal M'\longrightarrow (\mathcal M'\hat{\otimes}_{A'}A)\hat{\square}_{C}C'$ is an isomorphism in $\mathfrak S_{A'}^{C'}(\psi')$.
	\end{proof}
	\begin{Thm}\label{T4.13}
	Let $\mathfrak S$ be a $k$-linear Grothendieck category.	Let $(\alpha, \gamma)$ be a measuring from $(A',C',\psi')$ to $(A,C,\psi)$. Then, the following are equivalent:
		
		\smallskip
		(i) The functors $-\hat{\otimes}_{A'}A$ and $-\hat{\square}_C C'$ form an adjoint equivalence of categories between $\mathfrak S_{A}^{C}(\psi)$ and $\mathfrak S_{A'}^{C'}(\psi')$.
		
		\smallskip
		(ii) $(\alpha, \gamma)$ is an $\mathfrak S$-co-Galois measuring and the functors $-\hat{\otimes}_{A'}A$ and $-\hat{\square}_C C'$ are exact.
		
		\smallskip
		(iii)  $(\alpha, \gamma)$ is an $\mathfrak S$-co-Galois measuring, the functor $-\hat{\square}_C C'$ is exact and the functor $-\hat{\otimes}_{A'}A$ is faithfully exact.
		
		\smallskip
		(iv) $(\alpha, \gamma)$ is an $\mathfrak S$-co-Galois measuring, the functor  $-\hat{\otimes}_{A'}A$ is exact and the functor $-\hat{\square}_C C'$ is faithfully exact.
	\end{Thm}
	\begin{proof}
		%We will prove this in a manner similar to \cite[Theorem 3.10]{Tb}. 
		$(i) \Rightarrow (ii), (iii)~\textup{and}~(iv)$. Since the functors $-\hat{\otimes}_{A'}A$ and $-\hat{\square}_C C'$ form an adjoint equivalence, the corresponding unit and counit maps are always isomorphisms. By Definition \ref{D4.11}, it is now clear that $(\alpha, \gamma)$ is an $\mathfrak S$-co-Galois measuring. Again since $(-\hat{\otimes}_{A'}A,-\hat{\square}_C C')$ is an adjoint equivalence, we know that $(-\hat{\square}_C C',-\hat{\otimes}_{A'}A)$ is also an adjoint pair. Hence, the functors $-\hat{\otimes}_{A'}A$ and $-\hat{\square}_C C'$ are both exact. Just as in the proof of Theorem 
\ref{T4.7}, it now follows that   the functors $-\hat{\otimes}_{A'}A$ and $-\hat{\square}_C C'$ reflect exact sequences, i.e., they are faithfully exact.

		\smallskip
		$(ii) \Rightarrow (i).$ This follows from Proposition \ref{P4.12}.
		
		\smallskip
		$(iii) \Rightarrow (ii)$, $(iv) \Rightarrow (ii)$. Since every faithfully exact functor is exact, this is clear. 

	\end{proof}
	\begin{eg} 

\emph{Let $A$ be a $C$-Galois extension as in Example \ref{Example4.8}. Then,  $(A,\Delta_A:A\longrightarrow A\otimes C)$  is a right $C$-comodule  and the  canonical map}
\begin{equation} can: A\otimes_B A\xrightarrow{A\otimes_B \Delta_A} A\otimes_B A\otimes C\xrightarrow{\mu\otimes C} A\otimes C
\end{equation} \emph{is bijective, where $B\overset{\iota_{B}}{\hookrightarrow} A$ is the subalgebra of $A$  given by $B:=\{b\in A~|~\forall a\in A, \Delta_A(ba)=b\Delta_A(a)\}.$  Let $(A,C,\psi)$ be the entwining structure   canonically associated to the $C$-Galois extension  and let $(\alpha,\gamma)=(\iota_B,\Delta_A\circ \eta)$ be the measuring from $(B,k,id_B)$ to $(A,C,\psi)$ as in Example \ref{Example4.8}. Then, it may be verified directly 
that the functor $-\hat{\otimes}_BA=-\otimes_BA:\mathfrak S_B\longrightarrow \mathfrak S^C_A(\psi)$ has a right adjoint that takes} 
\begin{equation}
\mathcal N\mapsto \mathcal N_0:=Ker\left(\mathcal N\xrightarrow{\quad\Delta_{\mathcal N}-(\mu_{\mathcal N}\otimes C)(\mathcal N\otimes \Delta_A)(\mathcal N\otimes \eta)\quad}\mathcal N\otimes C\right)
\end{equation}\emph{for each $(\mathcal N,\Delta_{\mathcal N},\mu_{\mathcal N})\in \mathfrak S^C_A(\psi)$. Using the uniqueness of adjoints, it now follows from Proposition \ref{P4.10} that $\mathcal N\hat{\square}_Ck=\mathcal N_0$ for each  $(\mathcal N,\Delta_{\mathcal N},\mu_{\mathcal N})\in \mathfrak S^C_A(\psi)$.   We now consider an object $\mathcal M\in\mathfrak S$. Then, $\mathcal M\otimes B\otimes k\in \mathfrak S_B.$ From 
\cite[Example 3.9]{Tb}, we  know that }
\begin{equation}\label{4.73m} A_0:=Ker\left(A\xrightarrow{\quad\Delta_A-(\mu \otimes C)(A\otimes \Delta_A)(A\otimes \eta)\quad}A\otimes C\right)=B
\end{equation} \emph{It now follows that }
		\begin{equation}
			 ((\mathcal M\otimes B)\hat{\otimes}_B A)\hat{\square}_C k 
= ((\mathcal M\otimes B){\otimes}_B A)\hat{\square}_C k = (\mathcal M\otimes  A)_0=\mathcal M\otimes A_0\cong \mathcal M\otimes B
		\end{equation}\emph{and hence the map $\Omega_{\mathcal M\otimes B}$ induced by the unit of the adjunction
	is an isomorphism. Further, if $\mathcal M\in \mathfrak S$, we have $\mathcal M\otimes A\otimes C\in \mathfrak S^C_A(\psi)$.  From 
\cite[Example 3.9]{Tb}, we also know that } 
\begin{equation}\label{4.75m}
 (A\otimes C)_0:=Ker\left(A\otimes C\xrightarrow{\quad(A\otimes \Delta)-(((\mu\otimes C)(A\otimes \psi))\otimes C)(A\otimes C\otimes \Delta_A)(A\otimes C\otimes \eta)\quad}A\otimes C\otimes C\right)=A
\end{equation}\emph{It now follows that}
		\begin{equation}
		((\mathcal M\otimes A\otimes C)\hat{\square}_C k)\hat{\otimes}_B A =(\mathcal M\otimes A\otimes C)_0\otimes_BA =
\mathcal M\otimes (A\otimes C)_0\otimes_BA\cong \mathcal M\otimes A\otimes_BA\underset{\cong}{\xrightarrow{\mathcal M\otimes~can}} \mathcal M\otimes A\otimes C
		\end{equation}\emph{and hence the  map $\Upsilon_{\mathcal M\otimes A\otimes C}$ induced by the counit of the adjunction is an isomorphism. It now follows from Definition \ref{D4.11} that $(\alpha,\gamma)$ is an $\mathfrak S$-co-Galois measuring.}
		
	\end{eg}
	
		\section{Separability properties }
Given an entwining structure $(A,C,\psi)$, we know from Proposition \ref{P2.5} that there exists a pair of adjoint functors $$\mathcal T^C:\mathfrak S^C\longrightarrow \mathfrak S_A^C(\psi)~\qquad \qquad~\mathcal F^C:\mathfrak S_A^C(\psi)\longrightarrow \mathfrak S^C$$ such that $\mathcal T^C$ is left adjoint to the functor $\mathcal F^C$. Further in the case of entwined contramodule objects in $\mathfrak S$ over $(A,C,\psi)$, we have shown in Proposition \ref{L3.10} that there is a pair of adjoint functors $$^{[C,-]}\mathcal F:~_{\hspace{1em}A}^{[C,-]}\mathfrak S(\psi)\longrightarrow~ ^{[C,-]}\mathfrak S~\qquad\qquad~^{[C,-]}\mathcal T:~^{[C,-]}\mathfrak S\longrightarrow~ _{\hspace{1em}A}^{[C,-]}\mathfrak S(\psi)$$ where $^{[C,-]}\mathcal F$ is left adjoint to the functor $^{[C,-]}\mathcal T$. In this section we will study separability   conditions for the functors $\mathcal T^C$ and $\mathcal F^C$,  as well as for the functors $^{[C,-]}\mathcal F$ and $^{[C,-]}\mathcal T$.
\subsection{Separability of the functors $^{[C,-]}\mathcal T$ and $\mathcal T^C$}
We first consider the functor $^{[C,-]}\mathcal T:~^{[C,-]}\mathfrak S\longrightarrow~ _{\hspace{1em}A}^{[C,-]}\mathfrak S(\psi)$. We set  $V:=Nat(id_{^{[C,-]}\mathfrak S}, ^{[C,-]}\mathcal F~^{[C,-]}\mathcal T)$. We now set $V_1$ to be the $k$-vector space of natural transformations $\sigma:id_{\mathfrak S}\longrightarrow (A,C,-)$ satisfying the following condition
\begin{equation}\label{eq5.1}
	(A,\Delta,\matholdcal M)\circ (\psi,C,\matholdcal M)\circ (C,\sigma_{\matholdcal M})= (A,\Delta,\matholdcal M)\circ \sigma_{(C,\matholdcal M)}:(C,\matholdcal M)
	\longrightarrow (A,C,\matholdcal M)
\end{equation}
for any $\matholdcal M\in\mathfrak S$.
\begin{thm}\label{P5.1} Let $(A,C,\psi)$ be an entwining structure  and $\mathfrak S$ be a $k$-linear Grothendieck category. Then,
	
	\smallskip
	(i) There exists a map $\chi:V_1\longrightarrow V$ that associates $\sigma\in V_1$ to  a natural transformation $\tau\in V,$ defined as follows: for any $(\matholdcal M,\pi_{\matholdcal M})\in~^{[C,-]}\mathfrak S$, the map $	\tau_{\matholdcal M}:\matholdcal M\longrightarrow (A,\matholdcal M)$ is given by the composition
	\begin{equation}\label{eq5.2}
		\matholdcal M\xrightarrow{\sigma_{\matholdcal M}}(A,C,\matholdcal M)\xrightarrow{(A,\pi_{\matholdcal M})}(A,\matholdcal M)
	\end{equation}
	
	\smallskip
	(ii) There exists a map $\vartheta:V\longrightarrow V_1$ that assigns to $\tau\in V$ the natural transformation $\sigma\in V_1,$ defined as follows: for any $\matholdcal M\in\mathfrak S$, the map $\sigma_{\matholdcal M}:\matholdcal M\longrightarrow (A,C,\matholdcal M)$ is given by the composition
	\begin{equation}\label{eq5.3}
		\matholdcal M\xrightarrow{(\epsilon,\matholdcal M)}(C,\matholdcal M)\xrightarrow{\tau_{(C,\matholdcal M)}}(A,C,\matholdcal M)
	\end{equation}
	
	\smallskip
	(iii) $\chi:V_1\longrightarrow V$ and $\vartheta:V\longrightarrow V_1$ are mutually inverse isomorphisms.
\end{thm}
\begin{proof}
	(i) For $(\matholdcal M,\pi_{\matholdcal M})\in~^{[C,-]}\mathfrak S$, we first show that $\tau_{\matholdcal M}$ is a morphism in $^{[C,-]}\mathfrak S$. To do this, we must check that
	\begin{equation*}
		\begin{array}{lll}
			\pi_{(A,\matholdcal M)}\circ(C,\tau_{\matholdcal M})&=\pi_{(A,\matholdcal M)}\circ (C,A,\pi_{\matholdcal M})\circ(C,\sigma_{\matholdcal M})\\ &=(A,\pi_{\matholdcal M})\circ (\psi,\matholdcal M)\circ (C,A,\pi_{\matholdcal M})\circ(C,\sigma_{\matholdcal M})\\ 
			&=(A,\pi_{\matholdcal M})\circ(A,C,\pi_{\matholdcal M})\circ(\psi,C,\matholdcal M)\circ(C,\sigma_{\matholdcal M})\\ 
			&=(A,\pi_{\matholdcal M})\circ(A,\Delta, \matholdcal M)\circ(\psi,C,\matholdcal M)\circ(C,\sigma_{\matholdcal M})
			& \mbox{(as $\matholdcal M\in~^{[C,-]}\mathfrak S$ )} \\
			&=(A,\pi_{\matholdcal M})\circ(A,\Delta, \matholdcal M)\circ\sigma_{(C,\matholdcal M)}& \mbox{(by \eqref{eq5.1})} \\ 
			&=(A,\pi_{\matholdcal M})\circ(A,C,\pi_{ \matholdcal M})\circ\sigma_{(C,\matholdcal M)}\\ 
			&=(A,\pi_{\matholdcal M})\circ \sigma_{\matholdcal M}\circ \pi_{\matholdcal M}\\ 
			&=\tau_{\matholdcal M}\circ \pi_{\matholdcal M}\\
		\end{array}
	\end{equation*}  
	Further, if $\phi:\matholdcal M\longrightarrow \matholdcal N$is a morphism  in $^{[C,-]}\mathfrak S$,  we see that
	\begin{equation*}
		\begin{array}{ll}
			(A,\phi)\circ \tau_{\matholdcal M}=(A,\phi)\circ (A,\pi_{\matholdcal M})\circ \sigma_{\matholdcal M}&=(A,\pi_{\matholdcal  N})\circ (A,C,\phi)\circ\sigma_{\matholdcal M}\\ 
			&=(A,\pi_{\matholdcal  N})\circ\sigma_{\matholdcal N}\circ \phi\\\notag&=\tau_{\matholdcal N}\circ \phi
		\end{array}
	\end{equation*}
	This shows that $\tau$ is a natural transformation in $V$.
	
	\smallskip
	(ii) For each $\matholdcal M\in\mathfrak S$, we need to show that $\sigma$ satisfies the condition in \eqref{eq5.1}. Since $\tau_{(C,\matholdcal M)}:(C,\matholdcal M)\longrightarrow (A,C,\matholdcal M)$ is a morphism in $^{[C,-]}\mathfrak S$, we have the following commutative diagram
	\begin{equation*}\begin{tikzcd}[row sep=1.8em, column sep = 3.8em]
			(C,\matholdcal M)  \arrow{r}{(C,\epsilon,\matholdcal M)} \arrow{dr}[swap]{id}& (C,C,\matholdcal M) \arrow{r} {(C,\tau_{(C,\matholdcal M)})} \arrow{d}{(\Delta,\matholdcal M)}&(C,A,C,\matholdcal M)\arrow{d}{\pi_{(A,C,\matholdcal M)}} \\%
			&(C,\matholdcal M) \arrow{r}{\tau_{(C,\matholdcal M)}}& (A, C, \matholdcal M)
		\end{tikzcd}
	\end{equation*}
	where we know from \eqref{eq3.2} and Lemma \ref{L3.9} that $\pi_{(A,C,\matholdcal M)}:(C,A,C,\matholdcal M)\xrightarrow{(\psi, C,\matholdcal M)}(A,C,C,\matholdcal M)\xrightarrow{(A,\Delta,\matholdcal M)}(A,C,\matholdcal M)$.
	Then, we have
	\begin{equation}\label{eq5.4}
		\begin{array}{lll}
			\tau_{(C,\matholdcal M)}&=\pi_{(A,C,\matholdcal M)}\circ (C,\tau_{(C,\matholdcal M)})\circ (C,\epsilon,\matholdcal M)
			\\&=(A,\Delta,\matholdcal M)\circ (\psi,C,\matholdcal M)\circ (C,\tau_{(C,\matholdcal M)})\circ (C,\epsilon,\matholdcal M)\\&=(A,\Delta,\matholdcal M)\circ (\psi,C,\matholdcal M)\circ(C,\sigma_{\matholdcal M})& \mbox{(by (\ref{eq5.3}))}
		\end{array}
	\end{equation}
	Further, as $(C\otimes \epsilon)\circ \Delta=id=(\epsilon\otimes C)\circ \Delta$ and $\tau$ is a natural transformation, we obtain the following commutative diagram
	\begin{equation*}\begin{tikzcd}[row sep=1.8em, column sep = 3.8em]
			(C,\matholdcal M)  \arrow{r}{(\epsilon,C,\matholdcal M)} \arrow{dr}[swap]{id}& (C,C,\matholdcal M) \arrow{r} {\tau_{(C,C,\matholdcal M)}} \arrow{d}{(\Delta,\matholdcal M)}&(A,C,C,\matholdcal M)\arrow{d}{(A,\Delta,\matholdcal M)} \\%
			&(C,\matholdcal M) \arrow{r}{\tau_{(C,\matholdcal M)}}& (A, C, \matholdcal M)
		\end{tikzcd}
	\end{equation*}
	Therefore,
	\begin{equation}\label{eq5.5}
		\begin{array}{lll}
			\tau_{(C,\matholdcal M)}&=(A,\Delta,\matholdcal M)\circ \tau_{(C,C,\matholdcal M)}\circ (\epsilon,C,\matholdcal M)\\&=(A,\Delta,\matholdcal M)\circ \sigma_{(C,\matholdcal M)}&\mbox{(\textup{by} (\ref{eq5.3}))}
		\end{array}
	\end{equation}
	The result now follows from (\ref{eq5.4}) and (\ref{eq5.5}).
	
	\smallskip
	(iii) We consider a natural transformation $\sigma\in V_1$. Set $\chi(\sigma):=\tau$ and $\vartheta(\tau):=\sigma'$. Then, for each $\matholdcal M \in \mathfrak S$, we see that
	\begin{equation*}
		\begin{array}{lll}
			\sigma'_{\matholdcal M}&=\tau_{(C,\matholdcal M)}\circ (\epsilon,\matholdcal M)&\mbox{(\textup{by} (\ref{eq5.3}))}\\\notag&=(A,\pi_{(C,\matholdcal M)})\circ \sigma_{(C,\matholdcal M)}\circ (\epsilon,\matholdcal M)&\mbox{(\textup{by} (\ref{eq5.2}))}\\\notag&=(A,\pi_{(C,\matholdcal M)})\circ(A,C,\epsilon,\matholdcal M)\circ \sigma_{\matholdcal M}\\\notag
			&=(A,\Delta,\matholdcal M)\circ(A,C,\epsilon,\matholdcal M)\circ \sigma_{\matholdcal M}&\mbox{(\textup{as}~$\pi_{(C,\matholdcal M)}=(\Delta,\matholdcal M)$)}\\\notag
			&=\sigma_{\matholdcal M}&\mbox{(\textup{by} (\ref{eq3.1}))}
		\end{array}
	\end{equation*}
	Hence, we have $\vartheta \circ \chi = id$.  To prove that $\chi \circ \vartheta = id$, it suffices to show that $\vartheta$ is a monomorphism. We suppose that there exists a natural transformation $0\neq \tau\in V$ such that $\vartheta(\tau):=\sigma=0$. Let $(\matholdcal N,\pi_{\matholdcal N})\in~^{[C,-]}\mathfrak S$ be an object such that $0\neq\tau_{\matholdcal N}:\matholdcal N\longrightarrow (A,\matholdcal N)$. Then, we see that
	\begin{equation*}
		\begin{array}{lll}
			\tau_{\matholdcal N}&=\tau_{\matholdcal N}\circ \pi_{\matholdcal N}\circ (\epsilon,\matholdcal N)&\mbox{(\textup{by} (\ref{eq3.1}))}\\\notag&=(A,\pi_{\matholdcal N})\circ\tau_{(C,\matholdcal N)}\circ (\epsilon,\matholdcal N)\\\notag&=(A,\pi_{\matholdcal N})\circ\sigma_{\matholdcal N}&\mbox{(\textup{by}~(\ref{eq5.3}))}\\\notag&=0
		\end{array}
	\end{equation*}
	This leads to a contradiction, and the result follows.
\end{proof}
\begin{Thm}\label{T5.2}
	Let $(A,C,\psi)$ be an entwining structure  and $\mathfrak S$ be a $k$-linear Grothendieck category. Then, the following are equivalent.
	
	\smallskip
	(i) The functor $^{[C,-]}\mathcal T:~^{[C,-]}\mathfrak S \longrightarrow~ _{\hspace{1em}A}^{[C,-]}\mathfrak S(\psi)$ is separable. 
	
	\smallskip
	(ii) There exists $\sigma\in V_1$ such that
	\begin{equation}\label{eq5.6}
		(\eta,C,\matholdcal M)\circ \sigma_{\matholdcal M}=(\epsilon,\matholdcal M)
	\end{equation}
	for each $\matholdcal M\in\mathfrak S$.
\end{Thm}
\begin{proof}
	$(i) \Rightarrow (ii).$	Let $\omega:~^{[C,-]}\mathcal F^{[C,-]}\mathcal T\longrightarrow id_{^{[C,-]}\mathfrak S}$ denote the counit of adjunction $(^{[C,-]}\mathcal F,^{[C,-]}\mathcal T)$, which is  given by $\omega_{\matholdcal N}:(A,\matholdcal N)\xrightarrow{(\eta,\matholdcal N)}\matholdcal N$ for any $(\matholdcal N,\pi_{\matholdcal N})\in~^{[C,-]}\mathfrak S$. Since $^{[C,-]}\mathcal T$ is separable, we know from \cite[Proposition 1.1]{BCMZ} that there exists $\tau\in V=Nat(id_{^{[C,-]}\mathfrak S}, ^{[C,-]}\mathcal F~^{[C,-]}\mathcal T)$ such that $\omega\circ \tau=id_{^{[C,-]}\mathfrak S}$. We set $\sigma:=\vartheta(\tau)$. Then, it is clear that $\sigma\in V_1$ and for any $\matholdcal M\in\mathfrak S$, we have
	\begin{equation*}
		\begin{array}{lll}
			(\eta,C,\matholdcal M)\circ \sigma_{\matholdcal M}&=	(\eta,C,\matholdcal M)\circ\tau_{(C,\matholdcal M)}\circ (\epsilon,\matholdcal M) &\mbox{ (\textup{by}~(\ref{eq5.3}))}\\\notag
			&=\omega_{(C,\matholdcal M)}\circ \tau_{(C,\matholdcal M)}\circ (\epsilon,\matholdcal M)\\\notag&= (\epsilon,\matholdcal M)
		\end{array}
	\end{equation*}
	Therefore, $\sigma$ satisfies (\ref{eq5.6}).
	
	\smallskip
	$(ii) \Rightarrow (i).$ Suppose we have $\sigma\in V_1$ that satisfies \eqref{eq5.6}. We set $\tau:=\chi(\sigma)$. Then, for any $(\matholdcal N,\pi_{\matholdcal N})\in~^{[C,-]}\mathfrak S$ we have 
	\begin{equation*}
		\begin{array}{lll}
			\omega_{\matholdcal N}\circ \tau_{\matholdcal N}&=\omega_{\matholdcal N}\circ (A,\pi_{\matholdcal N})\circ \sigma_{\matholdcal N} &\mbox{(\textup{by}~(\ref{eq5.2}))}\\\notag&=(\eta,\matholdcal N)\circ(A,\pi_{\matholdcal N})\circ \sigma_{\matholdcal N}\\\notag&=\pi_{\matholdcal N}\circ (\eta,C,\matholdcal N)\circ \sigma_{\matholdcal N}\\\notag&=\pi_{\matholdcal N}\circ(\epsilon,\matholdcal N)\\\notag
			&=id_{\matholdcal N} &\mbox{ (\textup{by}~(\ref{eq3.1}))}
		\end{array}
	\end{equation*}
	Since we know from (\cite[Proposition 1.1]{BCMZ}) that  $^{[C,-]}\mathcal T$ is separable if and only if there exists $\tau\in V=Nat(id_{^{[C,-]}\mathfrak S}, ^{[C,-]}\mathcal F~^{[C,-]}\mathcal T)$ such that $\omega\circ \tau=id_{^{[C,-]}\mathfrak S}$, the result follows.
\end{proof}
\smallskip
We now  give an outline of the condition for the separability of the functor $\mathcal T^C:\mathfrak S^C\longrightarrow \mathfrak S_A^C(\psi)$ in the case of entwined comodule objects in 
$\mathfrak S$ over $(A,C,\psi)$.
We set $V':=Nat(\mathcal F^C\mathcal T^C, id_{\mathfrak S^C})$. In this case, we consider the $k$-vector space $V_1'$ of natural transformations $\sigma:(-\otimes C\otimes A)\longrightarrow id_{\mathfrak S}$ such that
\begin{equation}\label{eq5.7}
	(\sigma_{\mathcal M}\otimes C)\circ(\mathcal M\otimes C\otimes \psi)\circ (\mathcal M\otimes \Delta\otimes A)=\sigma_{(\mathcal M\otimes C)}\circ (\mathcal M\otimes \Delta\otimes A) 
\end{equation}
for any $\mathcal M\in \mathfrak S$.
In a manner similar to Proposition \ref{P5.1}, it can be shown that there exists a map $\chi:V_1'\longrightarrow V'$ that assigns to each $\sigma \in V_1'$ the natural transformation $\chi(\sigma) := \tau \in V'$, determined by setting 
\begin{equation}\label{eq5.8}
	\tau_{\mathcal N}:\mathcal N\otimes A\xrightarrow{(\Delta_{\mathcal N}\otimes A)}\mathcal N\otimes C\otimes A\xrightarrow{\sigma_{\mathcal N}} \mathcal N
\end{equation}
for each $(\mathcal N, \Delta_{\mathcal N})\in\mathfrak S^C$. The inverse of $\chi:V_1'\longrightarrow V'$ is given by the map $\vartheta:V'\longrightarrow V_1'$ that sends each element $\tau \in V'$ to a natural transformation $\vartheta(\tau) := \sigma \in V_1'$ given by setting
\begin{equation}\label{eq5.9}
	\sigma_{\mathcal M}:\mathcal M\otimes C\otimes A\xrightarrow{\tau_{(\mathcal M\otimes C)}}\mathcal M\otimes C\xrightarrow{(\mathcal M\otimes \epsilon)}\mathcal M
\end{equation}
for any $\mathcal M\in\mathfrak S$. 
\begin{Thm}\label{T5.3}
	Let $(A,C,\psi)$ be an entwining structure  and $\mathfrak S$ be a $k$-linear Grothendieck category. Then, the following statements are equivalent.
	
	\smallskip
	(i) The functor $\mathcal T^C:\mathfrak S^C\longrightarrow \mathfrak S^C_A(\psi)$ is separable.
	
	\smallskip
	(ii) There exists $\sigma\in V_1'$ such that
	\begin{equation}\label{eq5.10}
		\sigma_{\mathcal M}\circ (\mathcal M\otimes C\otimes \eta)=\mathcal M\otimes \epsilon 
	\end{equation}
	for each $\mathcal M\in\mathfrak S$.
\end{Thm}
\begin{proof}
	The proof is similar to that of Theorem \ref{T5.2}
\end{proof}
\subsection{Separability of the functors $^{[C,-]}\mathcal F$ and $\mathcal F^C$}
We begin with the functor $^{[C,-]}\mathcal F:~_{\hspace{1em}A}^{[C,-]}\mathfrak S(\psi)\longrightarrow~ ^{[C,-]}\mathfrak S$ and we set $W:=Nat(^{[C,-]}\mathcal T^{[C,-]}\mathcal F, id_{~ _{\hspace{1em}A}^{[C,-]}\mathfrak S(\psi)})$. Let $W_1$ be the $k$-vector space of natural transformations $\rho:(A,A,-)\longrightarrow (C,-)$ 
between endofunctors on $\mathfrak S$ satisfying the commutative diagrams
\begin{equation}\label{eq5.11}
	\begin{tikzcd}[row sep=1.8em, column sep = 3.8em]
		(A,A,A,\matholdcal M) \arrow{rr}{\rho_{(A,\matholdcal M)}} 
		&&  (C,A,\matholdcal M)\arrow{d}{(\psi,\matholdcal M)}	\\
		(A,A,\matholdcal M) \arrow{r}
		{(\mu,A,\matholdcal M)}\arrow{u}{(A,\mu,\matholdcal M)} &(A,A,A,\matholdcal M)\arrow{r}{(A,\rho_{\matholdcal M})}&(A,C,\matholdcal M)		 			
	\end{tikzcd}
\end{equation}
and
\begin{equation}\label{eq5.12}
	\begin{tikzcd}[row sep=1.8em, column sep = 3.8em]
		(A,C,A,\matholdcal M) \arrow{r}{(A,\psi,\matholdcal M)}
		&  (A,A,C,\matholdcal M)\arrow{r}{\rho_{(C,\matholdcal M)}} &(C,C,\matholdcal M)\arrow{d}{(\Delta,\matholdcal M)}\\
		(C,A,A,\matholdcal M )\arrow{r}
		{(C,\rho_{\matholdcal M})} \arrow{u}{(\psi,A,\matholdcal M)}&(C,C,\matholdcal M)\arrow{r}{(\Delta,\matholdcal M)}&(C,\matholdcal M)
	\end{tikzcd}
\end{equation}
for any $\matholdcal M\in\mathfrak S$.
\begin{thm}\label{P5.4}
	Let $(A,C,\psi)$ be an entwining structure  and $\mathfrak S$ be a $k$-linear Grothendieck category. Then,
	
	\smallskip
	(i)	There exists a map $\chi':W_1\longrightarrow W$ that assigns to $\rho\in W_1$ the natural transformation $\kappa\in W$ given as follows: for any $(\matholdcal M,\pi_{\matholdcal M},\mu_{\matholdcal M})\in~ _{\hspace{1em}A}^{[C,-]}\mathfrak S(\psi)$, $\kappa_{\matholdcal M}:(A,\matholdcal M)\longrightarrow \matholdcal M$ is the composition given by
	\begin{equation}\label{eq5.13}
		(A,\matholdcal M)\xrightarrow{(A,\mu_{\matholdcal M})}(A,A,\matholdcal M)\xrightarrow{\rho_{\matholdcal M}}(C,\matholdcal M)\xrightarrow{\pi_{\matholdcal M}}\matholdcal M
	\end{equation}
	
	\smallskip
	(ii) There exists a map $\vartheta':W\longrightarrow W_1$ that assigns to $\kappa\in W$ the natural transformation $\rho\in W_1$ given as follows: for any $\matholdcal M\in\mathfrak S$, $\rho_{\matholdcal M}:(A,A,\matholdcal M)\longrightarrow (C,\matholdcal M)$ is the composition
	\begin{equation}\label{eq5.14}
		(A,A,\matholdcal M)\xrightarrow{(A,A,\epsilon,\matholdcal M)}(A,A,C,\matholdcal M)\xrightarrow{\kappa_{(A,C,\matholdcal M)}}(A,C,\matholdcal M)\xrightarrow{(\eta,C,\matholdcal M)}(C,\matholdcal M)
	\end{equation}
	
	\smallskip
	(iii) The morphisms $\chi':W_1\longrightarrow W$ and $\vartheta':W\longrightarrow W_1$ are mutually inverse isomorphisms.
\end{thm}
\begin{proof}
	(i) For any $(\matholdcal M,\pi_{\matholdcal M},\mu_{\matholdcal M})\in~ _{\hspace{1em}A}^{[C,-]}\mathfrak S(\psi)$, we first show that $\kappa_{\matholdcal M}$ is a morphism in $_A\mathfrak S$, i.e., the following diagram commutes
	\begin{equation}\label{eq5.15}
		\begin{tikzcd}[row sep=1.8em, column sep = 3.8em]
			(A,\matholdcal M) \arrow{r}{\kappa_{\matholdcal M}} \arrow{d}{\mu_{(A,\matholdcal M)}}
			&  \matholdcal M\arrow{d}{\mu_{\matholdcal M}}	\\
			(A,A,\matholdcal M) \arrow{r}
			{(A,\kappa_{\matholdcal M})} &(A,\matholdcal M) 		 			
		\end{tikzcd}
	\end{equation}
	where $\mu_{(A,\matholdcal M)}:=(\mu,\matholdcal M):(A,\matholdcal M)\longrightarrow(A,A,\matholdcal M)$.
	We note that
	\begin{equation}\label{eq5.16}
		\begin{array}{lll}
			(A,\kappa_{\matholdcal M})\circ \mu_{(A,\matholdcal M)}&=(A,\pi_{\matholdcal M})\circ(A,\rho_{\matholdcal M})\circ (A,A,\mu_{\matholdcal M})\circ (\mu,\matholdcal M)\\
			&=(A,\pi_{\matholdcal M})\circ(A,\rho_{\matholdcal M})\circ (\mu,A,\matholdcal M)\circ (A,\mu_{\matholdcal M})\\
			&=(A,\pi_{\matholdcal M})\circ(\psi,\matholdcal M)\circ \rho_{(A,\matholdcal M)}\circ (A,\mu,\matholdcal M)\circ(A,\mu_{\matholdcal M})&\mbox{(\textup{by}~(\ref{eq5.11}))}\\
			&=(A,\pi_{\matholdcal M})\circ(\psi,\matholdcal M)\circ \rho_{(A,\matholdcal M)}\circ(A,A,\mu_{\matholdcal M})\circ (A,\mu_{\matholdcal M})\\
			&=(A,\pi_{\matholdcal M})\circ(\psi,\matholdcal M)\circ (C,\mu_{\matholdcal M})\circ \rho_{\matholdcal M}\circ(A,\mu_{\matholdcal M})\\
			&=\mu_{\matholdcal M}\circ \pi_{\matholdcal M}\circ \rho_{\matholdcal M}\circ(A,\mu_{\matholdcal M})&\mbox{(\textup{by}~(\ref{eq3.9}))}\\
			&=\mu_{\matholdcal M}\circ \kappa_{\matholdcal M}&\mbox{(\textup{by}~(\ref{eq5.13}))}
		\end{array}
	\end{equation}
	Therefore diagram in (\ref{eq5.15}) commutes.
	To show that $\kappa_{\matholdcal M}$ is a morphism of $C$-contramodule objects in $\mathfrak S$, we now check that
	the following diagram commutes
	\begin{equation*}
		\begin{tikzcd}[row sep=1.8em, column sep = 3.8em]
			(C,A,\matholdcal M) \arrow{r}{(C,\kappa_{\matholdcal M})} \arrow{d}{\pi_{(A,\matholdcal M)}}
			& (C,\matholdcal M)\arrow{d}{\pi_{\matholdcal M}}	\\
			(A,\matholdcal M) \arrow{r}
			{\kappa_{\matholdcal M}} &\matholdcal M	 			
		\end{tikzcd}
	\end{equation*}
	where $\pi_{(A,\matholdcal M)}:(C.A,\matholdcal M)\xrightarrow{(\psi,\matholdcal M)}(A,C,\matholdcal M)\xrightarrow{(A,\pi_{\matholdcal M})}(A,\matholdcal M)$.
	We note that
	\begin{equation}\label{eq5.17}
		\begin{array}{lll}
			\kappa_{\matholdcal M}\circ \pi_{(A,\matholdcal M)}&=\pi_{\matholdcal M}\circ \rho_{\matholdcal M}\circ (A,\mu_{\matholdcal M})\circ (A,\pi_{\matholdcal M})\circ (\psi,\matholdcal M)\\
			&=\pi_{\matholdcal M}\circ \rho_{\matholdcal M}\circ (A,A,\pi_{\matholdcal M})\circ (A,\psi,\matholdcal M)\circ (A,C,\mu_{\matholdcal M})\circ (\psi,\matholdcal M)&\mbox{(\textup{by}~(\ref{eq3.9}))}\\
			&=\pi_{\matholdcal M}\circ \rho_{\matholdcal M}\circ (A,A,\pi_{\matholdcal M})\circ (A,\psi,\matholdcal M)\circ (\psi, A,\matholdcal M)\circ (C,A,\mu_{\matholdcal M})\\
			&=\pi_{\matholdcal M}\circ (C,\pi_{\matholdcal M})\circ \rho_{(C,\matholdcal M)}\circ (A,\psi,\matholdcal M)\circ (\psi, A,\matholdcal M)\circ (C,A,\mu_{\matholdcal M})\\
			&=\pi_{\matholdcal M}\circ (\Delta,\matholdcal M)\circ \rho_{(C,\matholdcal M)}\circ (A,\psi,\matholdcal M)\circ (\psi, A,\matholdcal M)\circ (C,A,\mu_{\matholdcal M})\\
			&=\pi_{\matholdcal M}\circ (\Delta,\matholdcal M)\circ(C,\rho_{\matholdcal M})\circ (C,A,\mu_{\matholdcal M})&\mbox{(\textup{by}~(\ref{eq5.12}))}\\
			&=\pi_{\matholdcal M}\circ (C,\pi_{\matholdcal M})\circ(C,\rho_{\matholdcal M})\circ (C,A,\mu_{\matholdcal M})\\
			&=\pi_{\matholdcal M}\circ (C,\kappa_{\matholdcal M})
		\end{array}
	\end{equation}
	It follows  that $\kappa_{\matholdcal M}$ is a morphism in $_{\hspace{1em}A}^{[C,-]}\mathfrak S(\psi)$. Further, for any morphism $\phi:\matholdcal M\longrightarrow \matholdcal N$ in $_{\hspace{1em}A}^{[C,-]}\mathfrak S(\psi)$, we can verify that $\kappa_{\matholdcal N}\circ (A,\phi)=\phi\circ \kappa_{\matholdcal M}$. Therefore $\kappa$ is a natural transformation in $W$. 
	
	\smallskip
	(ii) Let $\matholdcal M\in\mathfrak S$. We have to show that $\rho_{\matholdcal M}$ satisfies (\ref{eq5.11}) and (\ref{eq5.12}). Since $\kappa_{(C,A,\matholdcal M)}$ is a morphism in $_A\mathfrak S$, we have the following commutative diagram
	\begin{equation}\label{eq5.18}\small
		\begin{tikzcd}[row sep=1.8em, column sep = 3.8em]
			(A,A,\matholdcal M) \arrow{r}{(A,\epsilon,A,\matholdcal M)} \arrow{d} {\mu_{(A,A,\matholdcal M)}}
			&  (A,C,A,\matholdcal M)\arrow{r}{\kappa_{(C,A,\matholdcal M})}\arrow{d}{\mu_{(A,C,A,\matholdcal M)}} &(C,A,\matholdcal M)\arrow{d}{\mu_{(C,A,\matholdcal M)}}\\
			(A,A,A,\matholdcal M) \arrow{r}
			{(A,A,\epsilon,A,\matholdcal M)} &(A,A,C,A,\matholdcal M)\arrow{r}{(A,\kappa_{(C,A,\matholdcal M)})}&(A,C,A,\matholdcal M)
		\end{tikzcd}
	\end{equation}
	Here, we note from Lemma \ref{L3.11} that $\mu_{(C,A,\matholdcal M)}:(C,A,\matholdcal M)\xrightarrow{(C,\mu,\matholdcal M)}(C,A,A,\matholdcal M)\xrightarrow{(\psi,A,\matholdcal M)}(A,C,A,\matholdcal M)$. Further as $\mu_{(A,A,\matholdcal M)}=(\mu,A,\matholdcal M)$, it follows from \eqref{eq5.18} that
	\begin{equation}\label{eq5.19}
		(A,\kappa_{(C,A,\matholdcal M)})\circ (A,A,\epsilon,A,\matholdcal M)\circ (\mu,A,\matholdcal M)=\mu_{(C,A,\matholdcal M)}\circ \kappa_{(C,A,\matholdcal M)}\circ(A,\epsilon,A,\matholdcal M)
	\end{equation}
	Since $\kappa$ is a natural transformation, we have the following commutative diagram
	\begin{equation}\label{eq5.20}\small
		\begin{tikzcd}[row sep=1.8em, column sep = 3.8em]
			(A,A,\matholdcal M) \arrow{r}{(A,\epsilon,A,\matholdcal M)} \arrow{d} {(A,\mu,\matholdcal M)}
			&  (A,C,A,\matholdcal M)\arrow{r}{\kappa_{(C,A,\matholdcal M})}\arrow{d}{(A,\mu_{(C,A,\matholdcal M)})} &(C,A,\matholdcal M)\arrow{d}{\mu_{(C,A,\matholdcal M)}}\\
			(A,A,A,\matholdcal M) \arrow{r}
			{(A,A,\epsilon,A,\matholdcal M)} &(A,A,C,A,\matholdcal M)\arrow{r}{\kappa_{(A,C,A,\matholdcal M)}}&(A,C,A,\matholdcal M)
		\end{tikzcd}
	\end{equation}
	It follows from \eqref{eq5.20} that 
	\begin{align}\label{eq5.21}
		 \kappa_{(A,C,A,\matholdcal M)}\circ (A,A,\epsilon,A,\matholdcal M)\circ (A,\mu,\matholdcal M)
		=\mu_{(C,A,\matholdcal M)}\circ \kappa_{(C,A,\matholdcal M)}\circ(A,\epsilon,A,\matholdcal M)
	\end{align}
	Again as $\kappa$ is a natural transformation and $(\psi,\matholdcal M)$ is a morphism in $ _{\hspace{1em}A}^{[C,-]}\mathfrak S(\psi)$, we note that $\kappa_{(A,C,\matholdcal M)}\circ (A,\psi,\matholdcal M)=(\psi,\matholdcal M)\circ \kappa_{(C,A,\matholdcal M)}$. Then,
	\begin{equation*}
		\begin{array}{lll}
			(A,\rho_{\matholdcal M})\circ (\mu,A,\matholdcal M)&=(A,\eta,C,\matholdcal M)\circ (A,\kappa_{(A,C,\matholdcal M)})\circ (A,A,A,\epsilon,\matholdcal M)\circ (\mu,A,\matholdcal M)\\\notag
			&=(A,\eta,C,\matholdcal M)\circ (A,\kappa_{(A,C,\matholdcal M)})\circ(A,A,\psi,\matholdcal M)\circ (A,A,\epsilon,A,\matholdcal M)\circ (\mu,A,\matholdcal M)&\mbox{(\textup{by}~(\ref{eq2.2}))}\\\notag
			&=(A,\eta,C,\matholdcal M)\circ(A,\psi,\matholdcal M)\circ (A,\kappa_{(C,A,\matholdcal M)})\circ (A,A,\epsilon,A,\matholdcal M)\circ (\mu,A,\matholdcal M)\\\notag
			&= (A,\eta,C,\matholdcal M)\circ(A,\psi,\matholdcal M)\circ\mu_{(C,A,\matholdcal M)}\circ \kappa_{(C,A,\matholdcal M)}\circ(A,\epsilon,A,\matholdcal M)&\mbox{(\textup{by}~(\ref{eq5.19}))}\\\notag
			&=(A,\eta,C,\matholdcal M)\circ(A,\psi,\matholdcal M)\circ(\psi,A,\matholdcal M)\circ (C,\mu,\matholdcal M)\circ \kappa_{(C,A,\matholdcal M)}\circ(A,\epsilon,A,\matholdcal M)\\
			&=(A,C,\eta,\matholdcal M)\circ (\psi,A,\matholdcal M)\circ (C,\mu,\matholdcal M)\circ \kappa_{(C,A,\matholdcal M)}\circ(A,\epsilon,A,\matholdcal M)&\mbox{(\textup{by}~(\ref{eq2.1}))}\\\notag
			&=(\psi,\matholdcal M)\circ (C,A,\eta,\matholdcal M)\circ (C,\mu,\matholdcal M)\circ \kappa_{(C,A,\matholdcal M)}\circ(A,\epsilon,A,\matholdcal M)\\
			&=(\psi,\matholdcal M)\circ (C,\eta,A,\matholdcal M)\circ (C,\mu,\matholdcal M)\circ \kappa_{(C,A,\matholdcal M)}\circ(A,\epsilon,A,\matholdcal M)\\
			&=(\psi,\matholdcal M)\circ (\eta,C,A,\matholdcal M)\circ (\psi,A,\matholdcal M)\circ (C,\mu,\matholdcal M)\circ \kappa_{(C,A,\matholdcal M)}\circ(A,\epsilon,A,\matholdcal M)&\mbox{(\textup{by}~(\ref{eq2.1}))}\\\notag
			&=(\psi,\matholdcal M)\circ (\eta,C,A,\matholdcal M)\circ\mu_{(C,A,\matholdcal M)}\circ \kappa_{(C,A,\matholdcal M)}\circ(A,\epsilon,A,\matholdcal M)\\
			&=(\psi,\matholdcal M)\circ (\eta,C,A,\matholdcal M)\circ \kappa_{(A,C,A,\matholdcal M)}\circ (A,A,\epsilon,A,\matholdcal M)\circ (A,\mu,\matholdcal M)&\mbox{(\textup{by}~(\ref{eq5.21}))}\\\notag
			&=(\psi,\matholdcal M)\circ \rho_{(A,\matholdcal M)}\circ (A,\mu,\matholdcal M)&\mbox{(\textup{by}~(\ref{eq5.14}))}
		\end{array}
	\end{equation*} This shows that $\rho_{\matholdcal M}$ satisfies (\ref{eq5.11}).
	
	\smallskip
	Further, as 
	$\kappa_{(A,C,\matholdcal M)}$ is a morphism in $^{[C,-]}\mathfrak S$, we consider the following commutative diagram
	\begin{equation}\label{eq5.22}
		\begin{tikzcd}[row sep=1.8em, column sep = 3.8em]
			(C,A,A,\matholdcal M) \arrow{r}{(C,A,A,\epsilon,\matholdcal M)} \arrow{dr}[swap]{(A,\psi,\matholdcal M)\circ (\psi,A,\matholdcal M)} 
			&  (C,A,A,C,\matholdcal M)\arrow{r}{(C,\kappa_{(A,C,\matholdcal M)})}\arrow{d}{\pi_{(A,A,C,\matholdcal M)}} &(C,A,C,\matholdcal M)\arrow{d}{\pi_{(A,C,\matholdcal M)}}\\
			&(A,A,C,\matholdcal M)\arrow{r}{\kappa_{(A,C,\matholdcal M)}}&(A,C,\matholdcal M)
		\end{tikzcd}
	\end{equation}
	where we know from Lemma \ref{L3.9} that $\pi_{(A,C,\matholdcal M)}:(C,A,C,\matholdcal M)\xrightarrow{(\psi,C,\matholdcal M)}(A,C,C,\matholdcal M)\xrightarrow{(A,\Delta,\matholdcal M)}(A,C,\matholdcal M)$ and that $\pi_{(A,A,C,\matholdcal M)}:(C,A,A,C,\matholdcal M)\xrightarrow{(\psi,A,C,\matholdcal M)}(A,C,A,C,\matholdcal M)\xrightarrow{(A,\pi_{(A,C,\matholdcal M)})}(A,A,C,\matholdcal M)$. Here, $\pi_{(A,A,C,\matholdcal M)}\circ (C,A,A,\epsilon,\matholdcal M)=(A,\psi,\matholdcal M)\circ (\psi,A,\matholdcal M)$ as 
	\begin{equation*}
		\begin{array}{lll}
			\pi_{(A,A,C,\matholdcal M)}\circ (C,A,A,\epsilon,\matholdcal M)&=(A,A,\Delta,\matholdcal M)\circ (A,\psi,C,\matholdcal M)\circ (\psi,A,C,\matholdcal M)\circ (C,A,A,\epsilon,\matholdcal M)\\\notag&=(A,A,\Delta,\matholdcal M)\circ (A,\psi,C,\matholdcal M)\circ(A,C,A,\epsilon,\matholdcal M)\circ (\psi,A,\matholdcal M)\\\notag&=(A,A,\Delta,\matholdcal M)\circ(A,A,C,\epsilon,\matholdcal M)\circ (A,\psi,\matholdcal M)\circ (\psi,A,\matholdcal M)
			\\\notag&=(A,\psi,\matholdcal M)\circ (\psi,A,\matholdcal M)
		\end{array}
	\end{equation*}
	Therefore, from (\ref{eq5.22}), we get
	\begin{equation}\label{eq5.23}
		\kappa_{(A,C,\matholdcal M)}\circ (A,\psi,\matholdcal M)\circ (\psi,A,\matholdcal M)=\pi_{(A,C,\matholdcal M)}\circ (C,\kappa_{(A,C,\matholdcal M)})\circ (C,A,A,\epsilon,\matholdcal M)
	\end{equation}
	Again as $\kappa$ is a natural transformation and $(A,\Delta,\matholdcal M)$  is a morphism in $ _{\hspace{1em}A}^{[C,-]}\mathfrak S(\psi)$, we have the following commutative diagram
	\begin{equation}\label{eq5.24}
		\begin{tikzcd}[row sep=1.8em, column sep = 5.2em]
			(C,A,A,\matholdcal M) \arrow{r}{(A,\psi,\matholdcal M)\circ (\psi,A,\matholdcal M)} 
			&  (A,A,C,\matholdcal M)\arrow{r}{(A,A,\epsilon,C,\matholdcal M)}\arrow{dr}{id} &(A,A,C,C,\matholdcal M)\arrow{d}{(A,A,\Delta,\matholdcal M)}\arrow{r}{\kappa_{(A,C,C,\matholdcal M)}}&(A,C,C,\matholdcal M)\arrow{d}{(A,\Delta,\matholdcal M)}\\
			&&(A,A,C,\matholdcal M) \arrow{r}
			{\kappa_{(A,C,\matholdcal M)}} &(A,C,\matholdcal M)
		\end{tikzcd}
	\end{equation}
	Therefore, we get 
	\begin{equation}\label{eq5.25}
		\kappa_{(A,C,\matholdcal M)}\circ (A,\psi,\matholdcal M)\circ (\psi,A,\matholdcal M)=(A,\Delta,\matholdcal M)\circ \kappa_{(A,C,C,\matholdcal M)}\circ (A,A,\epsilon,C,\matholdcal M)\circ (A,\psi,\matholdcal M)\circ (\psi,A,\matholdcal M)
	\end{equation}
	We now see
	\begin{equation}\label{eq5.26}
		\begin{array}{lll}
			(\Delta,\matholdcal M)\circ (C,\rho_{\matholdcal M}) &=(\Delta,\matholdcal M)\circ (C,\eta,C,\matholdcal M)\circ (C,\kappa_{(A,C,\matholdcal M)})\circ (C,A,A,\epsilon,\matholdcal M)\\\notag
			&=(\eta,C,\matholdcal M)\circ \pi_{(A,C,\matholdcal M)}\circ (C,\kappa_{(A,C,\matholdcal M)})\circ (C,A,A,\epsilon,\matholdcal M)\\\notag
			&=(\eta,C,\matholdcal M)\circ	\kappa_{(A,C,\matholdcal M)}\circ (A,\psi,\matholdcal M)\circ (\psi,A,\matholdcal M)&\mbox{(\textup{by}~(\ref{eq5.23}))}\\\notag
			&=(\eta,C,\matholdcal M)\circ (A,\Delta,\matholdcal M)\circ \kappa_{(A,C,C,\matholdcal M)}\circ (A,A,\epsilon,C,\matholdcal M)\circ (A,\psi,\matholdcal M)\circ (\psi,A,\matholdcal M)&\mbox{(\textup{by}~(\ref{eq5.25}))}\\\notag
			&=(\Delta,\matholdcal M)\circ (\eta,C,C,\matholdcal M)\circ \kappa_{(A,C,C,\matholdcal M)}\circ (A,A,\epsilon,C,\matholdcal M)\circ (A,\psi,\matholdcal M)\circ (\psi,A,\matholdcal M)\\\notag
			&=(\Delta,\matholdcal M)\circ \rho_{(C,\matholdcal M)}\circ (A,\psi,\matholdcal M)\circ (\psi,A,\matholdcal M)&\mbox{(\textup{by}~\eqref{eq5.14})}
		\end{array}
	\end{equation}
	This shows that $\rho_{\matholdcal M}$ satisfies (\ref{eq5.12}). Hence $\rho\in W_1$. 
	
	\smallskip
	(iii) We start by showing that $\vartheta'\circ \chi'=id$ We consider a natural transformation  $\rho\in W_1$ and set $\chi'(\rho):=\kappa$, $\vartheta'(\kappa):=\rho'$. Then, for any $\matholdcal M\in\mathfrak S$ we have
	\begin{equation*}
		\begin{array}{lll}
			\rho'_{\matholdcal M}&=(\eta,C,\matholdcal M)\circ \kappa_{(A,C,\matholdcal M)}\circ (A,A,\epsilon,\matholdcal M)&\mbox{(\textup{by}~(\ref{eq5.14}))}\\
			&=(\eta,C,\matholdcal M)\circ \pi_{(A,C,\matholdcal M)}\circ(\rho_{(A,C,\matholdcal M)})\circ(A,\mu_{(A,C\matholdcal M)})\circ (A,A,\epsilon,\matholdcal M)&\mbox{(\textup{by}~(\ref{eq5.13}))}\\
			&=(\eta,C,\matholdcal M)\circ (A,\Delta,\matholdcal M)\circ (\psi,C,\matholdcal M)\circ(\rho_{(A,C,\matholdcal M)})\circ(A,\mu,C,\matholdcal M)\circ (A,A,\epsilon,\matholdcal M)&\mbox{\((\mu_{(A,C\matholdcal M)}=(\mu,C,\matholdcal M))\)}\\
			&=(\eta,C,\matholdcal M)\circ (A,\Delta,\matholdcal M)\circ(A,\rho_{(C,\matholdcal M)})\circ (\mu,A,C,\matholdcal M)\circ(A,A,\epsilon,\matholdcal M)&\mbox{(\textup{by}~(\ref{eq5.11}))}\\
			&=(\Delta,\matholdcal M)\circ(\eta, C,C,\matholdcal M)\circ (A,\rho_{(C,\matholdcal M)})\circ (\mu,A,C,\matholdcal M)\circ(A,A,\epsilon,\matholdcal M)\\
			&=(\Delta,\matholdcal M)\circ \rho_{(C,\matholdcal M)}\circ (\eta,A,A,C,\matholdcal M)\circ (\mu,A,C,\matholdcal M)\circ(A,A,\epsilon,\matholdcal M)\\
			&=(\Delta,\matholdcal M)\circ \rho_{(C,\matholdcal M)}\circ(A,A,\epsilon,\matholdcal M)\\
			&=(\Delta,\matholdcal M)\circ (C,\epsilon,\matholdcal M)\circ \rho_{\matholdcal M}\\
			&=\rho_{\matholdcal M}
		\end{array}
	\end{equation*}
	It follows that $\vartheta'\circ \chi'=id.$ 
	As in the proof of Proposition \ref{P5.1}(iii), it is now enough to show that $\vartheta'$ is a monomorphism. For this, we assume that there exists a natural transformation $ 0\neq\kappa\in W$ such that $\vartheta'(\kappa)=\rho=0$. Let $\matholdcal M\in\mathfrak S$. Since $\kappa$ is a natural transformation and $(A,\Delta,\matholdcal M)$ is a morphism in
	 $ _{\hspace{1em}A}^{[C,-]}\mathfrak S(\psi)$, we obtain the following commutative diagram
	\begin{equation}
		\begin{tikzcd}[row sep=1.8em, column sep = 3.8em]
			(A,A,C,\matholdcal M)\arrow{r}{(A,A,\epsilon,C,\matholdcal M)}\arrow{dr}{id} &(A,A,C,C,\matholdcal M)\arrow{d}{(A,A,\Delta,\matholdcal M)}\arrow{r}{\kappa_{(A,C,C,\matholdcal M)}}&(A,C,C,\matholdcal M)\arrow{d}{(A,\Delta,\matholdcal M)}\\
			&(A,A,C,\matholdcal M) \arrow{r}
			{\kappa_{(A,C,\matholdcal M)}} &(A,C,\matholdcal M)\arrow{r}{(\eta, C,\matholdcal M)}&(C,\matholdcal M)
		\end{tikzcd}
	\end{equation}
	Then, we see that
	\begin{equation}\label{527tg}\small
		\begin{array}{lll}
			(\eta,C,\matholdcal M)\circ \kappa_{(A,C,\matholdcal M)}&=	(\eta,C,\matholdcal M)\circ(A,\Delta,\matholdcal M)\circ \kappa_{(A,C,C,\matholdcal M)}\circ (A,A,\epsilon,C,\matholdcal M)\\
			&=(\Delta,\matholdcal M)\circ (\eta,C,C,\matholdcal M)\circ \kappa_{(A,C,C,\matholdcal M)}\circ (A,A,\epsilon,C,\matholdcal M)\\
			&=(\Delta,\matholdcal M)\circ\rho_{(C,\matholdcal M)}=0\\
		\end{array}
	\end{equation}
	It follows from \eqref{527tg} that $(A,\eta,C,\matholdcal M)\circ (A,\kappa_{(A,C,\matholdcal M)})=0$. Further as $\kappa_{(A,C,\matholdcal M)}$ is a morphishm in $_A\mathfrak S$ and $\mu\circ (\eta\otimes A)=id=\mu\circ (A\otimes\eta)$, we have the following commutative diagram
	\begin{equation}\label{528tg}\small
		\begin{tikzcd}[row sep=1.8em, column sep = 3.8em]
			(A,A,C,\matholdcal M)\arrow{d}{\mu_{(A,A,C,\matholdcal  M)}}\arrow{r}{\kappa_{(A,C,\matholdcal M)}}&(A,C,\matholdcal M)\arrow{d}{(\mu,C,\matholdcal  M)}\arrow{r}{id}&(A,C,\matholdcal M)\\
			(A,A,A,C,\matholdcal M) \arrow{r}
			{(A,\kappa_{(A,C,\matholdcal M)})} &(A,A,C,\matholdcal M)\arrow{ur}[swap]{(A,\eta,C,\matholdcal M)}
		\end{tikzcd}
	\end{equation}
	It follows that $\kappa_{(A,C,\matholdcal M)}=(A,\eta,C,\matholdcal M)\circ (A,\kappa_{(A,C,\matholdcal M)})\circ\mu_{(A,A,C,\matholdcal  M)}=0$ for any $\matholdcal M\in\mathfrak S$. We now consider any $(\matholdcal N,\pi_{\matholdcal N})\in~^{[C,-]}\mathfrak S$. Since $\kappa$ is a natural transformation and $(A,\pi_{\matholdcal N})$ is a morphism in
	 $ _{\hspace{1em}A}^{[C,-]}\mathfrak S(\psi)$, we have the following commutative diagram 
	\begin{equation}\label{529tg}\small
		\begin{tikzcd}[row sep=1.8em, column sep = 3.8em]
			(A,A,C,\matholdcal N)\arrow{d}{(A,A,\pi_{\matholdcal  N})}\arrow{r}{\kappa_{(A,C,\matholdcal N)}}&(A,C,\matholdcal N)\arrow{d}{(A,\pi_{\matholdcal N})}\\
			(A,A,\matholdcal N) \arrow{r}
			{\kappa_{(A,\matholdcal N)}} &(A,\matholdcal N)
		\end{tikzcd}
	\end{equation}
	Since $\kappa_{(A,C,\matholdcal M)}=0$ for all $\matholdcal M\in\mathfrak S$, it follows from \eqref{529tg} that $\kappa_{(A,\matholdcal N)}\circ (A,A,\pi_{\matholdcal N})=(A,\pi_{\matholdcal N})\circ \kappa_{(A,C,\matholdcal N)}=0$. Since $(A,A,\pi_{\matholdcal N})\circ (A,A,\epsilon,\matholdcal N)=id$ in $\mathfrak S$ and cokernels in $_{\hspace{1em}A}^{[C,-]}\mathfrak S(\psi)$ are computed in $\mathfrak S$, it follows that 
	  $(A,A,\pi_{\matholdcal N})$ is also an epimorphism in 
	$_{\hspace{1em}A}^{[C,-]}\mathfrak S(\psi)$. Hence, $\kappa_{(A,\matholdcal N)}=0$. Now for any object $(\matholdcal P,\pi_{\matholdcal P},\mu_{\matholdcal P})\in~_{\hspace{1em}A}^{[C,-]}\mathfrak S(\psi)$, we have
	$\mu_{\matholdcal P}\circ \kappa_{\matholdcal P}=\kappa_{(A,\matholdcal P)}\circ (A,\mu_{\matholdcal P})=0.$ Since $\mu_{\matholdcal P}$ is a monomorphism in $_{\hspace{1em}A}^{[C,-]}\mathfrak S(\psi)$, it follows that $\kappa_{\matholdcal P}=0$ for every $\matholdcal P\in {_{\hspace{1em}A}^{[C,-]}\mathfrak S(\psi)}$. This shows $\kappa=0$, which is a contradiction.
\end{proof}
\begin{Thm}\label{T5.5}
	Let $(A,C,\psi)$ be an entwining structure  and $\mathfrak S$ be a $k$-linear Grothendieck category. Then, the following are equivalent.
	
	\smallskip
	(i) The functor $^{[C,-]}\mathcal F:~_{\hspace{1em}A}^{[C,-]}\mathfrak S(\psi)\longrightarrow~ ^{[C,-]}\mathfrak S$ is separable 
	
	\smallskip
	(ii) There exists $\rho\in W_1$ such that
	\begin{equation}\label{eq5.27}
		\rho_{\matholdcal M}\circ (\mu,\matholdcal M)=(\epsilon,\matholdcal M)\circ (\eta,\matholdcal M)
	\end{equation}
	for each $\matholdcal M\in\mathfrak S$.
\end{Thm}
\begin{proof}
	$(i) \Rightarrow (ii).$	We know that $(^{[C,-]}\mathcal F, ^{[C,-]}\mathcal T)$ is an adjoint pair. Let $\omega':id_{_{\hspace{1em}A}^{[C,-]}\mathfrak S(\psi)}\longrightarrow {^{[C,-]}\mathcal T}{^{[C,-]}\mathcal F}$ denote the unit of this adjunction, which is given by $\omega'_{\matholdcal N}:\matholdcal N\xrightarrow{\mu_{\matholdcal N}} (A,\matholdcal N)$ for every $(\matholdcal N,\pi_{\matholdcal N},\mu_{\matholdcal N})\in ~_{\hspace{1em}A}^{[C,-]}\mathfrak S(\psi)$. Since $^{[C,-]}\mathcal F$ is separable, by \cite[Proposition 1.1]{BCMZ}, there exists a natural transformation $\kappa\in W_1$ such that $\kappa\circ \omega'=id$. We set $\rho:=\vartheta'(\kappa)$. Then $\rho\in W_1$ and for every $\matholdcal M\in\mathfrak S$, we have
	\begin{equation*}
		\begin{array}{lll}
			\rho_{\matholdcal M}\circ (\mu,\matholdcal M)&=(\eta,C,\matholdcal M)\circ \kappa_{(A,C,\matholdcal M)}\circ (A,A,\epsilon,\matholdcal M)\circ (\mu,\matholdcal M)&\mbox{(\textup{by}~(\ref{eq5.14}))}\\
			&=(\eta,C,\matholdcal M)\circ \kappa_{(A,C,\matholdcal M)}\circ\mu_{(A,C,\matholdcal M)}\circ (A,\epsilon,\matholdcal M)\\
			&=(\eta,C,\matholdcal M)\circ \kappa_{(A,C,\matholdcal M)}\circ(\omega'_{(A,C,\matholdcal M)})\circ (A,\epsilon,\matholdcal M)\\
			&=(\eta,C,\matholdcal M)\circ(A,\epsilon,\matholdcal M)&\mbox{\((\textup{as}~(\kappa\circ \omega'=id))\)}\\\notag
			&=(\epsilon,\matholdcal M)\circ (\eta,\matholdcal M)
		\end{array}
	\end{equation*}
	This shows that $\rho$ satisfies (\ref{eq5.27}) for every $\matholdcal M\in\mathfrak S$. 
	
	\smallskip
	$(ii) \Rightarrow (i).$ We consider a natural transformation $\rho\in W_1$ that satisfies (\ref{eq5.27}).  We set $\kappa:=\chi(\rho) \in W$. For any $\matholdcal M\in~{_{\hspace{1em}A}^{[C,-]}\mathfrak S(\psi)}$ we see that
	\begin{equation*}
		\begin{array}{lll}
			\kappa_{\matholdcal M}\circ \omega'_{\matholdcal M}=\kappa_{\matholdcal M}\circ \mu_{\matholdcal M}&=\pi_{\matholdcal M}\circ \rho_{\matholdcal M}\circ (A,\mu_{\matholdcal M})\circ \mu_{\matholdcal M}&\mbox{\textup{by}~\eqref{eq5.13}}\\&=\pi_{\matholdcal M}\circ \rho_{\matholdcal M}\circ(\mu,\matholdcal M)\circ \mu_{\matholdcal M}\\
			&=\pi_{\matholdcal M}\circ(\epsilon,\matholdcal M)\circ(\eta,\matholdcal M)\circ\mu_{\matholdcal M}\\
			&=id_\matholdcal M
		\end{array}
	\end{equation*}
	Hence, $\kappa\circ \omega'=id$, and by \cite[Proposition 1.1]{BCMZ}, it follows that $^{[C,-]}\mathcal F$ is separable. This completes the proof.
\end{proof}

\smallskip
We now turn to the case of entwined comodule objects and consider the forgetful functor $\mathcal F^C:\mathfrak S_A^C(\psi)\longrightarrow \mathfrak S^C$. For this, we set $W':=Nat(id_{\mathfrak S_A^C(\psi)}, \mathcal T^C\mathcal F^C)$. Let $W_1'$ be the $k$-vector space of natural transformations $\rho:(-\otimes C)\longrightarrow (-\otimes A\otimes A)$ between endofunctors on $\mathfrak S$ such that for each $\mathcal M\in\mathfrak S$, we have commutative diagrams
\begin{equation}\label{eq5.28}
	\begin{tikzcd}[row sep=1.8em, column sep = 3.8em]
		\mathcal M\otimes C \arrow{r}{(\mathcal M\otimes \Delta)} \arrow{d} {(\mathcal M\otimes \Delta)}
		&  (\mathcal M\otimes C\otimes C)\arrow{r}{(\rho_{\mathcal M}\otimes C)} &(\mathcal M\otimes A\otimes A\otimes C)\\
		(\mathcal M\otimes C\otimes C) \arrow{r}
		{\rho_{(\mathcal M\otimes C)}} &(\mathcal M\otimes C\otimes A\otimes A)\arrow{r}{(\mathcal M\otimes \psi\otimes A)}&(\mathcal M\otimes A\otimes C\otimes A)\arrow{u}{(\mathcal M\otimes A\otimes \psi)}
	\end{tikzcd}
\end{equation} and 
\begin{equation}\label{eq5.29}
	\begin{tikzcd}[row sep=1.8em, column sep = 3.8em]
		(\mathcal M\otimes C\otimes A) \arrow{r}{(\rho_{\mathcal M}\otimes A)} \arrow{d} {(\mathcal M\otimes \psi)}
		&  (\mathcal M\otimes A\otimes A\otimes A)\arrow{r}{(\mathcal M\otimes A\otimes \mu)} &(\mathcal M\otimes A\otimes A)\\
		(\mathcal M\otimes A\otimes C) \arrow{rr}
		{\rho_{(\mathcal M\otimes A)}} &&(\mathcal M\otimes A\otimes A\otimes A)\arrow{u}{(\mathcal M\otimes \mu\otimes A)}
	\end{tikzcd}
\end{equation}
In that case, we define a map $\chi':W_1'\longrightarrow W'$, $\rho\mapsto \kappa$, where $\kappa$ is a natural transformation given by setting
\begin{equation}
	\kappa_{\mathcal N}:\mathcal N\xrightarrow{\Delta_{\mathcal N}}\mathcal N\otimes C\xrightarrow{\rho_{\mathcal N}}\mathcal N\otimes A\otimes A\xrightarrow{\mu_{\mathcal N}\otimes A}\mathcal N\otimes A
\end{equation}
for any $(\mathcal N,\Delta_{\mathcal N},\mu_{\mathcal N})\in\mathfrak S^C_A(\psi).$  Further,  we define a map $\vartheta':W'\longrightarrow W_1'$ $\kappa\mapsto \rho$   by setting
\begin{equation}
	\rho_{\mathcal M}:\mathcal M\otimes C\xrightarrow{\mathcal M\otimes C\otimes \eta}\mathcal M\otimes C\otimes A\xrightarrow{\kappa_{(\mathcal M \otimes C\otimes A)}}\mathcal M\otimes C\otimes A\otimes  A\xrightarrow{\mathcal M\otimes \epsilon\otimes A\otimes A}\mathcal M\otimes A\otimes A
\end{equation} for any $\mathcal M\in \mathfrak S$. 
Then, it may be verified that $\chi'$ and $\vartheta'$ are mutually inverse isomorphisms.
\begin{Thm}\label{Tlst}
	Let $(A,C,\psi)$ be an entwining structure and $\mathfrak S$ be a $k$-linear Grothendieck category. Then, the following are equivalent.
	
	\smallskip
	(i) The functor $\mathcal F^C:\mathfrak S^C_A(\psi)\longrightarrow \mathfrak S^C$ is separable.
	
	\smallskip
	(ii) There exists $\rho\in W_1'$ such that
	\begin{equation}
		(\mu_{\mathcal M}\otimes A)\circ \rho_{\mathcal M}=(\mathcal M\otimes \eta)\circ (\mathcal M\otimes \epsilon)
	\end{equation}
	for each $\mathcal M\in\mathfrak S$.
\end{Thm}
\begin{proof}
	The proof is similar to that  of Theorem \ref{T5.5}.
\end{proof}

	\section{Frobenius properties}
	%$\mathcal C \overset{L}{\underset{R}\rightleftarrows}\mathcal D$ be functors between the categories $\mathcal C$ and $\mathcal D$. Recall that a pair of adjoint functors $(L,R)$ is said to be a Frobenius pair if $(R,L)$ is also an adjoint pair. Given an entwined structure 
	We continue with $(A, C, \psi)$ being an entwining structure. In this section, we give conditions under which the adjoint pairs $(^{[C,-]}\mathcal F, ^{[C,-]}\mathcal T)$ and $(\mathcal T^C, \mathcal F^C)$ are Frobenius, i.e, when $(^{[C,-]}\mathcal T, ^{[C,-]}\mathcal F)$ and $(\mathcal F^C, \mathcal T^C)$ are also adjoint pairs. First, we consider the adjunction  $(^{[C,-]}\mathcal F, ^{[C,-]}\mathcal T)$ for entwined contramodule objects in 
$\mathfrak S$ over $(A,C,\psi)$. 
	\begin{lem}\label{L6.1}
		Suppose that the adjoint pair $(^{[C,-]}\mathcal F, ^{[C,-]}\mathcal T)$ is Frobenius. Then, there exist natural transformations $\sigma\in V_1$ and $\rho\in W_1$ such that 
		\begin{align}
			&(\epsilon,\matholdcal M)\circ (\eta,\matholdcal M)=(\Delta,\matholdcal M)\circ \rho_{(C,\matholdcal M)}\circ(A,\psi,\matholdcal M)\circ \sigma_{(A,\matholdcal M)}\label{eq6.1}\\
			&(\epsilon,\matholdcal M)\circ (\eta,\matholdcal M)=(\Delta,\matholdcal M)\circ \rho_{(C,\matholdcal M)}\circ (A,\sigma_{\matholdcal M})\label{eq6.2}
		\end{align}
		for any $\matholdcal M\in \mathfrak S$.
	\end{lem}
	\begin{proof}
		Since $(^{[C,-]}\mathcal F, ^{[C,-]}\mathcal T)$ is Frobenius, by \cite[\S 1]{BCMZ}, there exist natural transformations $\tau\in V=Nat(id_{^{[C,-]}\mathfrak S}, ^{[C,-]}\mathcal F~^{[C,-]}\mathcal T)$ and $\kappa\in W=Nat(^{[C,-]}\mathcal T^{[C,-]}\mathcal F, id_{~ _{\hspace{1em}A}^{[C,-]}\mathfrak S(\psi)})$ satisfying the following conditions:
		\begin{align}
			&^{[C,-]}\mathcal F(\kappa_{\matholdcal M})\circ \tau_{^{[C,-]}\mathcal F(\matholdcal M)}=id_{^{[C,-]}\mathcal F(\matholdcal M)}:\matholdcal M\xrightarrow{\tau_{^{[C,-]}\mathcal F(\matholdcal M)}}(A,\matholdcal M)\xrightarrow{^{[C,-]}\mathcal F(\kappa_{\matholdcal M})}\matholdcal M\label{eq6.3}\\
			&\kappa_{^{[C,-]}\mathcal T(\matholdcal N)}\circ ^{[C,-]}\mathcal T(\tau_{\matholdcal N})=id_{^{[C,-]}\mathcal T(\matholdcal N)}:(A,\matholdcal N)\xrightarrow{^{[C,-]}\mathcal T(\tau_{\matholdcal N})}(A,A,\matholdcal N)\xrightarrow{\kappa_{^{[C,-]}\mathcal T(\matholdcal N)}}(A,\matholdcal N)\label{eq6.4} 
		\end{align} 
		for any $(\matholdcal M,\pi_{\matholdcal M},\mu_{\matholdcal M})\in~ _{\hspace{1em}A}^{[C,-]}\mathfrak S(\psi)$ and $(\matholdcal N,\pi_{\matholdcal N})\in~^{[C,-]}\mathfrak S.$ 
		We know from Proposition \ref{P5.1} and Proposition \ref{P5.4} respectively that $V_1\cong V$ and $W_1\cong W$. Therefore, we can choose $\sigma\in V_1$ and $\rho\in W_1$ corresponding to $\tau$ and $\kappa$ respectively. For any object $\matholdcal M\in\mathfrak S$, we have $(A,C,\matholdcal M)\in~_{\hspace{1em}A}^{[C,-]}\mathfrak S(\psi)$ with $\pi_{(A,C,\matholdcal M)}=(A,\Delta,\matholdcal M)\circ (\psi,C,\matholdcal M)$ and $\mu_{(A,C\matholdcal M)}=(\mu,C,\matholdcal M)$. On applying (\ref{eq6.3}) to $(A,C,\matholdcal M)\in~_{\hspace{1em}A}^{[C,-]}\mathfrak S(\psi)$, we get 
		%the following commutative diagram
		%			 			\begin{equation}\label{eq6.5}
			%			 			\begin{tikzcd}[row sep=1.8em, column sep = 3.8em]
				%			 				(A,C,\matholdcal M) \arrow{r}{\sigma_{(A,C,\matholdcal M)}} \arrow{d} {id}
				%			 				&  (A,C,A,C\matholdcal M)\arrow{r}{(A,\pi_{(A,C\matholdcal M)})} &(A,A,C,\matholdcal M)\arrow{d}{(A,\mu_A,C,\matholdcal M)}\\
				%			 				(A,C,\matholdcal M) \arrow{r}
				%			 				{\pi_{(A,C,\matholdcal M)}} &(A,A,C,A,\matholdcal M)\arrow{r}{(A,\kappa_{(C,A,\matholdcal M)})}&(A,C,A,\matholdcal M)
				%			 			\end{tikzcd}
			%			 		\end{equation}
		\begin{equation}\label{eq6.5}
			\begin{array}{lll}
			id_{(A,C,\matholdcal M)}&=^{[C,-]}\mathcal F(\kappa_{(
				A,C,\matholdcal M)})\circ \tau_{^{[C,-]}\mathcal F(A,C,\matholdcal M)}\\
			&=\pi_{(A,C,\matholdcal M)}\circ \rho_{(A,C,\matholdcal M)}\circ (A,\mu,C,\matholdcal M)\circ (A,\pi_{(A,C,\matholdcal M)})\circ \sigma_{(A,C,\matholdcal M)}\\
			&=(A,\Delta,\matholdcal M)\circ (\psi,C,\matholdcal M)\circ \rho_{(A,C,\matholdcal M)}\circ (A,\mu,C,\matholdcal M)\circ (A,A,\Delta,\matholdcal M)\circ (A,\psi,C,\matholdcal M)\circ \sigma_{(A,C,\matholdcal M)}\\
			&=(A,\Delta,\matholdcal M)\circ (A,\rho_{(C,\matholdcal M)})\circ (\mu,A,C,\matholdcal M)\circ (A,A,\Delta,\matholdcal M)\circ (A,\psi,C,\matholdcal M)\circ \sigma_{(A,C,\matholdcal M)}&\mbox{(\textup{by}~(\ref{eq5.11}))}
		\end{array}
	\end{equation}
		By composing with $(\eta,C,\matholdcal M)$ on the left and $(A,\epsilon,\matholdcal M)$ on the right on both sides of \eqref{eq6.5}, we obtain
			\begin{equation*}
			\begin{array}{lll}
			&(\eta,C,\matholdcal M)\circ (A,\epsilon,\matholdcal M)\\\notag&=(\eta,C,\matholdcal M)\circ (A,\Delta,\matholdcal M)\circ (A,\rho_{(C,\matholdcal M)})\circ (\mu,A,C,\matholdcal M)\circ (A,A,\Delta,\matholdcal M)\circ (A,\psi,C,\matholdcal M)\circ \sigma_{(A,C,\matholdcal M)}\circ (A,\epsilon,\matholdcal M)\\
			&=(\eta,C,\matholdcal M)\circ (A,\Delta,\matholdcal M)\circ (A,\rho_{(C,\matholdcal M)})\circ (\mu,A,C,\matholdcal M)\circ (A,A,\Delta,\matholdcal M)\circ (A,\psi,C,\matholdcal M)\circ(A,C,A,\epsilon,\matholdcal M)\circ \sigma_{(A,\matholdcal M)}\\
			&=(\eta,C,\matholdcal M)\circ (A,\Delta,\matholdcal M)\circ (A,\rho_{(C,\matholdcal M)})\circ (\mu,A,C,\matholdcal M)\circ (A,A,\Delta,\matholdcal M)\circ (A,A,\epsilon,C,\matholdcal M)\circ(A,\psi,\matholdcal M)\circ \sigma_{(A,\matholdcal M)}\\
			&=(\eta,C,\matholdcal M)\circ (A,\Delta,\matholdcal M)\circ (A,\rho_{(C,\matholdcal M)})\circ (\mu,A,C,\matholdcal M)\circ (A,\psi,\matholdcal M)\circ \sigma_{(A,\matholdcal M)}\hspace{4cm}(\textup{as}~\Delta\circ (\epsilon,C)=id)\\
			&=(\Delta,\matholdcal M)\circ (\eta,C,C,\matholdcal M)\circ (A,\rho_{(C,\matholdcal M)})\circ (\mu,A,C,\matholdcal M)\circ (A,\psi,\matholdcal M)\circ \sigma_{(A,\matholdcal M)}\\
			&=(\Delta,\matholdcal M)\circ  (\rho_{(C,\matholdcal M)})\circ (\eta,A,A,C,\matholdcal M)\circ (\mu,A,C,\matholdcal M)\circ (A,\psi,\matholdcal M)\circ \sigma_{(A,\matholdcal M)}\\
			&=(\Delta,\matholdcal M)\circ  (\rho_{(C,\matholdcal M)})\circ(A,\psi,\matholdcal M)\circ \sigma_{(A,\matholdcal M)}\hspace{8.5cm}(\textup{as}~(\eta,A)\circ \mu=id)
	\end{array}
\end{equation*}
		Since $(\eta,C,\matholdcal M)\circ (A,\epsilon,\matholdcal M)=(\epsilon,\matholdcal M)\circ(\eta,\matholdcal M)$, we obtain (\ref{eq6.1}).
		
		\smallskip
		Further, on applying (\ref{eq6.4}) to $(C,\matholdcal M)\in~^{[C,-]}\mathfrak S$ we obtain
		\begin{equation}\label{eq6.6}
			\begin{array}{lll}
			id_{^{[C,-]}\mathcal T(C,\matholdcal M)}=id_{(A,C,\matholdcal M)}&=\kappa_{^{[C,-]}\mathcal T(C,\matholdcal M)}\circ ^{[C,-]}\mathcal T(\tau_{(C,\matholdcal M)})\\
			&=(A,\Delta,\matholdcal M)\circ (\psi,C,\matholdcal M)\circ \rho_{(A,C,\matholdcal M)}\circ (A,\mu,C,\matholdcal M)\circ (A,A,\Delta,\matholdcal M)\circ (A,\sigma_{(C,\matholdcal M)})
		\end{array}
	\end{equation}
	Again by composing with $(\eta,C,\matholdcal M)$ on the left and $(A,\epsilon,\matholdcal M)$ on the right on both sides of \eqref{eq6.6}, we arrive at (\ref{eq6.2}), i.e., we have
			\begin{equation*}
			\begin{array}{lll}
			&(\eta,C, \matholdcal M)\circ (A,\epsilon,\matholdcal M)\\\notag&=(\eta,C, \matholdcal M)\circ (A,\Delta,\matholdcal M)\circ (\psi,C,\matholdcal M)\circ \rho_{(A,C,\matholdcal M)}\circ (A,\mu,C,\matholdcal M)\circ (A,A,\Delta,\matholdcal M)\circ (A,\sigma_{(C,\matholdcal M)})\circ (A,\epsilon,\matholdcal M)\\
			&=	(\eta,C,\matholdcal M)\circ (A,\Delta,\matholdcal M)\circ (A,\rho_{(C,\matholdcal M)})\circ (\mu,A,C,\matholdcal M)\circ (A,A,\Delta,\matholdcal M)\circ (A,\sigma_{(C,\matholdcal M)})\circ (A,\epsilon,\matholdcal M)&\mbox{(\textup{by}~(\ref{eq5.11}))}\\
			&= (\eta,C,\matholdcal M)\circ (A,\Delta,\matholdcal M)\circ (A,\rho_{(C,\matholdcal M)})\circ (\mu,A,C,\matholdcal M)\circ (A,A,\Delta,\matholdcal M)\circ (A,A,C,\epsilon,\matholdcal M)\circ (A,\sigma_{\matholdcal M})\\
			&=(\eta,C,\matholdcal M)\circ (A,\Delta,\matholdcal M)\circ (A,\rho_{(C,\matholdcal M)})\circ (\mu,A,C,\matholdcal M)\circ  (A,\sigma_{\matholdcal M})&\mbox{\((\textup{as}~\Delta\circ (\epsilon,C)=id)\)}\\
			&=(\Delta,\matholdcal M)\circ (\eta,C,C,\matholdcal M)\circ (A,\rho_{(C,\matholdcal M)})\circ (\mu,A,C,\matholdcal M)\circ  (A,\sigma_{\matholdcal M})\\
			&=(\Delta,\matholdcal M)\circ\rho_{(C,\matholdcal M)}\circ (\eta,A,A,C,\matholdcal M)\circ (\mu,A,C,\matholdcal M)\circ  (A,\sigma_{\matholdcal M})\\
			&=(\Delta,\matholdcal M)\circ \rho_{(C,\matholdcal M)}\circ  (A,\sigma_{\matholdcal M})&\mbox{\((\textup{as}~(\eta,A)\circ \mu=id)\)}
			\end{array}
	\end{equation*}
	\end{proof}
	\begin{Thm}\label{T6.2}
		Let $(A,C,\psi)$ be an entwining structure and let $\mathfrak S$ be a $k$-linear Grothendieck category. Then, the adjoint pair $(^{[C,-]}\mathcal F, ^{[C,-]}\mathcal T)$ is Frobenius if and only if there exist natural transformations $\sigma\in V_1$ and $\rho\in W_1$ satisfying the  conditions in \eqref{eq6.1} and \eqref{eq6.2}.
	\end{Thm}	
	\begin{proof} The ``only if'' part follows from Lemma \ref{L6.1}. 
	For the converse,  we choose $\tau\in V\cong V_1$ and $\kappa\in W\cong W_1$ corresponding to $\sigma\in V_1$ and $\rho\in W_1$ respectively. To show  $(^{[C,-]}\mathcal F, ^{[C,-]}\mathcal T)$ is a Frobenius pair, by \cite[\S 1]{BCMZ}, it suffices to show that $\tau$ and $\kappa$ satisfy the conditions in (\ref{eq6.3}) and (\ref{eq6.4}). We consider an object $\matholdcal M\in ~_{\hspace{1em}A}^{[C,-]}\mathfrak S(\psi)$. On composing \eqref{eq6.1} on the right with $\mu_{\matholdcal M}$ we get
		\begin{equation*}
		\begin{array}{lll}
			(\epsilon,\matholdcal M)&=(\epsilon,\matholdcal M)\circ (\eta,\matholdcal M)\circ \mu_{\matholdcal M}\\&=(\Delta,\matholdcal M)\circ (\rho_{(C,\matholdcal M)})\circ (A,\psi,\matholdcal M)\circ \sigma_{(A,\matholdcal M)}\circ \mu_{\matholdcal M}\\
			&=(\Delta,\matholdcal M)\circ (\rho_{(C,\matholdcal M)})\circ (A,\psi,\matholdcal M)\circ(A,C,\mu_{\matholdcal M})\circ \sigma_{\matholdcal M}\\
			&=(\Delta,\matholdcal M)\circ (\rho_{(C,\matholdcal M)})\circ (A,\mu_{(C,\matholdcal M)})\circ \sigma_{\matholdcal M}
		\end{array}
	\end{equation*}
		Therefore, we obtain the following commutative diagram 
		\begin{equation}\label{eq6.7}
			\begin{tikzcd}[row sep=2.8em, column sep = 4.8em]
				\matholdcal M\arrow{r}{\sigma_{\matholdcal M}}\arrow{rddd}[swap]{(\epsilon,\matholdcal M)} &(A,C,\matholdcal M)\arrow{d}{(A,\mu_{(C,\matholdcal M)})}\arrow{r}{(A,\pi_{\matholdcal M})}&(A,\matholdcal M)\arrow{d}{(A,\mu_{\matholdcal M})}\\
				&(A,A,C,\matholdcal M) \arrow{r}
				{(A,A,\pi_{\matholdcal M})} \arrow{d}{\rho_{(C,\matholdcal M)}}&(A,A,\matholdcal M)\arrow{d}{\rho_{\matholdcal M}}\\
				&(C,C,\matholdcal M) \arrow{r}
				{(C,\pi_{\matholdcal M})}\arrow{d}{(\Delta,\matholdcal M)} &(C,\matholdcal M)\arrow{d}{\pi_{\matholdcal M}}\\
				&(C,\matholdcal M)\arrow{r}{\pi_{\matholdcal M}}&\matholdcal M
			\end{tikzcd} 
		\end{equation}
		Hence, the composition $^{[C,-]}\mathcal F(\kappa_{\matholdcal M})\circ \tau_{^{[C,-]}\mathcal F(\matholdcal M)}=\pi_{ \matholdcal M}\circ \rho_{\matholdcal M}\circ (A,\mu_{\matholdcal M})\circ (A,\pi_{ \matholdcal M})\circ \sigma_{\matholdcal M}=\pi_{ \matholdcal M}\circ (\epsilon,\matholdcal M)=id_{^{[C,-]}\mathcal F(\matholdcal M)}$, i.e., $\matholdcal M\in ~_{\hspace{1em}A}^{[C,-]}\mathfrak S(\psi)$ satisfies \eqref{eq6.3}. 
		
		\smallskip
		We now consider an object $(\matholdcal N,\pi_{\matholdcal N})\in~^{[C,-]}\mathfrak S$. Since $\sigma$ and $\rho$ also satisfy \eqref{eq6.2}, we have
		\begin{equation}\label{62e}
			(\epsilon,\matholdcal N)\circ (\eta,\matholdcal N)=(\Delta,\matholdcal N)\circ \rho_{(C,\matholdcal N)}\circ (A,\sigma_{\matholdcal N})
		\end{equation}
		On applying the functor $(A,-)$ to (\ref{62e}) and then composing on the right with $(\mu, \matholdcal N)$, we obtain
		\begin{equation*}
		\begin{array}{lll}
			(A,\epsilon,\matholdcal N)=(A,\epsilon,\matholdcal N)\circ (A,\eta,\matholdcal N)\circ (\mu,\matholdcal N)&=(A,\Delta,\matholdcal N)\circ (A,\rho_{(C,\matholdcal N)})\circ (A,A,\sigma_{\matholdcal N})\circ (\mu,\matholdcal N)\\
			&=(A,\Delta,\matholdcal N)\circ (A,\rho_{(C,\matholdcal N)})\circ (\mu,A,C,\matholdcal N)\circ (A,\sigma_{\matholdcal N})\\
			&=(A,\Delta,\matholdcal N)\circ(\psi,C,\matholdcal N)\circ \rho_{(A,C,\matholdcal N)}\circ (A,\mu,C,\matholdcal N)\circ (A,\sigma_{\matholdcal N})&\mbox{(\textup{by}~(\ref{eq5.11}))}
		\end{array} 	
	\end{equation*}
		Therefore, we have the following commutative diagram
		\begin{equation}\label{eq6.8}
			\begin{tikzcd}[row sep=2em, column sep = 4.8em]
				(A,\matholdcal N)\arrow{r}{(A,\sigma_{\matholdcal N})}\arrow{rdddd}[swap]{(A,\epsilon,\matholdcal N)} &(A,A,C,\matholdcal N)\arrow{d}{(A,\mu, C,\matholdcal N)}\arrow{r}{(A,A,\pi_{\matholdcal N})}&(A,A,\matholdcal N)\arrow{d}{(A,\mu,\matholdcal N)}\\
				&(A,A,A,C,\matholdcal N) \arrow{r}
				{(A,A,A,\pi_{\matholdcal N})} \arrow{d}{\rho_{(A,C,\matholdcal N)}}&(A,A,A,\matholdcal N)\arrow{d}{\rho_{(A,\matholdcal N)}}\\
				&(C,A,C,\matholdcal N) \arrow{r}
				{(C,A,\pi_{\matholdcal N})}\arrow{d}{(\psi,C,\matholdcal N)} &(C,A,\matholdcal N)\arrow{d}{(\psi,\matholdcal N)}\\
				&(A,C,C,\matholdcal N)\arrow{r}{(A,C,\pi_{\matholdcal N})}\arrow{d}{(A,\Delta,\matholdcal N)}&(A,C,\matholdcal N)\arrow{d}{(A,\pi_{\matholdcal N})}\\
				&(A,C,\matholdcal N)\arrow{r}{(A,\pi_{ \matholdcal N})}&(A,\matholdcal N)
			\end{tikzcd} 
		\end{equation}
		It now follows that the composition
		\begin{equation}\kappa_{^{[C,-]}\mathcal T(\matholdcal N)}\circ ^{[C,-]}\mathcal T(\tau_{\matholdcal N})=(A,\pi_{ \matholdcal N})\circ(\psi,\matholdcal N)\circ \rho_{(A,\matholdcal N)}\circ (A,\mu,\matholdcal N)\circ (A,A,\pi_{ \matholdcal N})\circ (A,\sigma_{\matholdcal N}) =id_{^{[C,-]}\mathcal T(\matholdcal N)}
\end{equation} In other words $\matholdcal N \in~^{[C,-]}\mathfrak S$ satisfies \eqref{eq6.4}. This proves the result.

	\end{proof}

	\smallskip
	We now come to the setting of entwined comodule objects in $\mathfrak S$ over $(A,C,\psi)$. We briefly discuss when the adjoint pair $(\mathcal T^C, \mathcal F^C)$ is Frobenius. By \cite[\S 1]{BCMZ}, the pair $(\mathcal T^C, \mathcal F^C)$ is Frobenius if and only if there exist natural transformations $\tau\in V'=Nat(\mathcal F^C\mathcal T^C, id_{\mathfrak S^C})$ and $\kappa\in W'=Nat(id_{\mathfrak S_A^C(\psi)}, \mathcal T^C\mathcal F^C)$ satisfying the following conditions:
	\begin{align}
		&\tau_{\mathcal F^C(\mathcal M)}\circ \mathcal F^C(\kappa_{\mathcal M})=id_{\mathcal F^C(\mathcal M)}:\mathcal M\xrightarrow{ \mathcal F^C(\kappa_{\mathcal M})}\mathcal M\otimes A\xrightarrow{\tau_{\mathcal F^C(\mathcal M)}}\mathcal M\label{eq6.9*}\\
		&\mathcal T^C(\tau_{\mathcal N})\circ \kappa_{\mathcal T^C(\mathcal N)}=id_{\mathcal T^C(\matholdcal N)}:\mathcal N\otimes A\xrightarrow{\kappa_{\mathcal T^C(\mathcal N)}}\mathcal N\otimes A\otimes A\xrightarrow{\mathcal T^C(\tau_{\mathcal N})}\mathcal N\otimes A\label{eq6.10*}
	\end{align} 
	for any $(\mathcal M,\Delta_{\mathcal M},\mu_{\mathcal M})\in \mathfrak S_A^C(\psi)$ and $(\mathcal N,\Delta_{\mathcal N})\in \mathfrak S^C$. Accordingly, we can prove the following result. 
	\begin{Thm}\label{tT6.3}
	Let $(A,C,\psi)$ be an entwining structure and let $\mathfrak S$ be a $k$-linear Grothendieck category. 	Then, the adjoint pair $(\mathcal T^C, \mathcal F^C)$ is Frobenius if and only if there exist natural transformations $\sigma\in V_1'$ and $\rho\in W_1'$ such that
		\begin{align}
			&(\mathcal M\otimes \eta)\circ (\mathcal M\otimes \epsilon)=\sigma_{(\mathcal M\otimes A)}\circ (\mathcal M\otimes\psi\otimes A)\circ \rho_{(\mathcal M\otimes C)}\circ (\mathcal M\otimes \Delta)\label{eq6.11}\\
			&(\mathcal M\otimes \eta)\circ (\mathcal M\otimes \epsilon)=(\sigma_{\mathcal M}\otimes A)\circ \rho_{(\mathcal M\otimes C)}\circ (\mathcal M\otimes \Delta)\label{eq6.12}
		\end{align}
		for any $\mathcal M\in\mathfrak S$.
	\end{Thm}
	\begin{proof}
		If the pair $(\mathcal T^C, \mathcal F^C)$ is Frobenius, there exists $\tau\in V'$ and $\kappa\in W'$ satisfying (\ref{eq6.9*}) and (\ref{eq6.10*}). From Section 5, we have the $V'\cong V_1'$ and $W'\cong W_1'$, and we may choose $\sigma\in V_1'$ and $\rho\in W_1'$ corresponding to $\tau$ and $\kappa$ respectively. Then, (\ref{eq6.11}) is obtained by applying (\ref{eq6.9*}) to the object $\mathcal M \otimes C \otimes A \in \mathfrak S_A^C(\psi)$, while (\ref{eq6.12}) follows by applying \eqref{eq6.10*} to the object \((\mathcal M \otimes C) \in \mathfrak S^C\) for any $\mathcal M \in \mathfrak S$. The proof of the converse is similar to the proof of Theorem \ref{T6.2}.
	\end{proof}
	\section{Semisimple and Maschke functors}
	In this section, we discuss semisimple and Maschke type properties in the setting of both entwined contramodule objects and entwined comodule objects in $\mathfrak S$. We continue with an entwining structure $(A,C,\psi)$. In this section, we will always assume that the algebra $A$ is finite dimensional over $k$. As a consequence,   there exists a coevaluation map $coev_A:k\longrightarrow A\otimes A^{\ast}$ given by
	$r\mapsto r\underset{i\in I}{\sum }a_i\otimes a_1^{\ast}$  where $\{a_i\}_{i\in I}$ is a basis for $A$ over $k$, and $\{a_i^\ast\}_{i\in I}$ is its dual basis. We now recall the notion of a normalized cointegral from \cite[Definition 4.5]{Tb*}.
	\begin{defn}\label{D7.1}
		Let $(A,C,\psi)$ be an entwining structure. A $k$-linear map $\varphi:A^{\ast}\otimes C\longrightarrow A$ is said to be a normalized cointegral for $(A,C,\psi)$ if the following conditions hold.
		
		\smallskip
		(i) $(A\otimes \psi)\circ (\psi\otimes\varphi)\circ (C\otimes coev_A\otimes C)\otimes \Delta=(A\otimes \varphi\otimes C)\circ (coev_A\otimes \Delta)$
		
		\smallskip
		(ii) $(A\otimes \mu)\circ (A\otimes \varphi\otimes A)\circ (coev_A\otimes C\otimes A)=(\mu\otimes \varphi)\circ (A\otimes coev_A\otimes C)\circ \psi$
		
		\smallskip
		(iii) $\mu\circ (A\otimes \varphi)\circ (coev_A\otimes C)=\eta\circ \epsilon$		 		
	\end{defn}
The following is now an entwined  contramodule counterpart for the arguments of Brzezi\'nski \cite{Tb*} with entwined modules (see also work by Caenepeel, Militiaru and Zhu \cite{CMZ} with Doi-Hopf modules). 
	\begin{lem}\label{L7.2}
	Let $\mathfrak S$ be a $k$-linear Grothendieck category.	Let  $(A,C,\psi)$ be an entwining structure and let $\varphi:A^*\otimes C\longrightarrow A$ be a normalized cointegral for $(A,C,\psi)$. Let $(\matholdcal M,\pi_{\matholdcal M},\mu_{\matholdcal M}),(\matholdcal N,\pi_{\matholdcal N},\mu_{\matholdcal N})\in$ $_{\hspace{1em}A}^{[C,-]}\mathfrak S(\psi)$ and $\xi:\matholdcal M\longrightarrow \matholdcal N$ be a morphism in $^{[C,-]}\mathfrak S$. Then, there exists a morphism $\tilde{\xi}:\matholdcal M\longrightarrow \matholdcal N$ in $_{\hspace{1em}A}^{[C,-]}\mathfrak S(\psi)$ given by the composition
		\begin{equation}\label{eq7.1}
			\tilde{\xi}=\pi_{ \matholdcal N}\circ (C,coev_A,\matholdcal N)\circ (\varphi, A,\matholdcal N)\circ (A,\mu_{\matholdcal N})\circ (A,\xi)\circ \mu_{\matholdcal M}
		\end{equation}
	\end{lem}
	\begin{proof}
		We begin by verifying that $\tilde{\xi}$ is a morphism in $_A\mathfrak S$. Since $ (A, \xi) \circ \mu_{\matholdcal M}$ is clearly a morphism in $_A\mathfrak S$, it remains to prove that the composition $\pi_{\matholdcal N}\circ(C, \mathrm{coev}_A, \matholdcal N) \circ (\varphi, A, \matholdcal N)\circ (A,\mu_{\matholdcal N})$ is also a morphism in $_A\mathfrak S$. For this, we consider the following diagram.
		\begin{equation}\label{eq7.2}
		\small
			\begin{tikzcd}[row sep=2.8em, column sep = 4.8em]
				(A,\matholdcal N)\arrow{r}{(A,\mu_{\matholdcal N})}\arrow{dd}{(\mu,\matholdcal N)}&(A,A,\matholdcal N)\arrow{r}{(\varphi,A,\matholdcal N)}\arrow{d}{(A,A,\mu_{\matholdcal N})} &(C,A^{\ast},A,\matholdcal N)\arrow{d}{(C,A^{\ast},A,\mu_{\matholdcal N})}\arrow{r}{(C,coev_A,\matholdcal N)}&(C,\matholdcal N)\arrow{d}{(C, \mu_{\matholdcal N})}\arrow{r}{\pi_{\matholdcal N}}&\matholdcal N \arrow{dd}{\mu_{\matholdcal N}}\\
				&(A,A,A,\matholdcal N)\arrow{r}{(\varphi,A,A,\matholdcal N)}&(C,A^{\ast},A,A,\matholdcal N)\arrow{r}{(C,coev_A,A,\matholdcal N)}&(C,A,\matholdcal N)\arrow{d}{(\psi,\matholdcal N)}\\
			%	&(A,A,\matholdcal N)\arrow{d}{(\mu,A,\matholdcal N)}\\
				(A,A,\matholdcal N)\arrow{r}{(A,A,\mu_{\matholdcal N})}&(A,A,A,\matholdcal N)\arrow{r}{(A,\varphi,A,\matholdcal N)}&(A,C,A^{\ast},A,\matholdcal N)\arrow{r}{(A,C,coev_A,\matholdcal N)}&(A,C,\matholdcal N)\arrow{r}{(A,\pi_{\matholdcal N})}&(A,\matholdcal N)
			\end{tikzcd}
		\end{equation}
		We see that
		\begin{equation}\label{eq7.3}
		\begin{array}{lll}
			&\mu_{\matholdcal N}\circ \pi_{\matholdcal N}\circ (C,coev_A,\matholdcal N)\circ (\varphi,A,\matholdcal N)\circ (A,\mu_{\matholdcal N})&\\ 
			&=(A,\pi_{\matholdcal N})\circ(\psi,\matholdcal N)\circ (C,\mu_{\matholdcal N})\circ (C,coev_A,\matholdcal N)\circ (\varphi,A,\matholdcal N)\circ (A,\mu_{\matholdcal N})& 
			\mbox{(by \eqref{eq3.9})} \\
			&=(A,\pi_{\matholdcal N})\circ(\psi,\matholdcal N)\circ (C,coev_A,A,\matholdcal N)\circ (\varphi,A,A,\matholdcal N)\circ (A,A,\mu_{\matholdcal N})\circ (A,\mu_{\matholdcal N})&\\
			&=(A,\pi_{\matholdcal N})\circ(\psi,\matholdcal N)\circ (C,coev_A,A,\matholdcal N)\circ (\varphi, \mu,\matholdcal N)\circ (A,\mu_{\matholdcal N})&\\ 
			&=(A,\pi_{\matholdcal N})\circ(A,C,coev_A,\matholdcal N)\circ (A,\varphi,A,\matholdcal N)\circ (\mu,A,\matholdcal N)\circ (A,\mu_{\matholdcal N})
			&\mbox{(by Definition \ref{D7.1})}\\ 
			&=(A,\pi_{\matholdcal N})\circ(A,C,coev_A,\matholdcal N)\circ (A,\varphi,A,\matholdcal N)\circ (A,A,\mu_{\matholdcal N})\circ (\mu, \matholdcal N)
			\end{array}
		\end{equation}
		Hence the diagram \eqref{eq7.1} in commutes. 
		
		\smallskip
		We now prove that $\tilde{\xi}$ is a morphism in $^{[C,-]}\mathfrak S$. Since $(A,\mu_{\matholdcal N}), (A,\xi)$ and $\mu_{\matholdcal M}$ are already the morphisms in $^{[C,-]}\mathfrak S$, it is enough to check that the following diagram commutes.
	
		\begin{equation}\label{eq7.4}
			\begin{tikzcd}[row sep=2.8em, column sep = 4.8em]
				(C,A,A,\matholdcal N) \arrow{r}{(C,\varphi, A,\matholdcal N)} \arrow{d} {(\psi, A,\matholdcal N)}
				&  (C,C,A^{\ast},A,\matholdcal N)\arrow{r}{(C,C,coev_A,\matholdcal N)} &(C,C,\matholdcal N)\arrow{r}{(C,\pi_{\matholdcal N})}&(C,\matholdcal N)\arrow{dd}{id}\\
				(A,C,A,\matholdcal N)\arrow{d}{(A,\psi,\matholdcal N)}\\(A,A,C,\matholdcal N)\arrow{r}{(\varphi,A,C,\matholdcal N)}\arrow{d}{(A,A,\pi_{\matholdcal N})}&(C,A^{\ast},A,C,\matholdcal N)\arrow{r}{(C,coev_A,C,\matholdcal N)}\arrow{d}{(C,A^{\ast},A,\pi_{ \matholdcal N})}&(C,C,\matholdcal N)\arrow{r}{(\Delta,\matholdcal N)}\arrow{d}{(C,\pi_{ \matholdcal N})}&(C,\matholdcal N)\arrow{d}{\pi_{ \matholdcal N}}\\
				(A,A,\matholdcal N) \arrow{r}
				{(\varphi,A,\matholdcal N)} &(C,A^{\ast},A,\matholdcal N)\arrow{r}{(C,coev_A,\matholdcal N)}&(C,\matholdcal N)\arrow{r}{\pi_{ \matholdcal N}}&\matholdcal N
			\end{tikzcd}
		\end{equation}
		We note that 
		\begin{equation}
		\begin{array}{lll}
			&\pi_{ \matholdcal N}\circ (C,coev_A,\matholdcal N)\circ(\varphi,A,\matholdcal N)\circ (A,A,\pi_{ \matholdcal N})\circ (A,\psi,\matholdcal N)\circ (\psi,A,\matholdcal N)&\\\
			&=\pi_{ \matholdcal N}\circ (\Delta,\matholdcal N)\circ (C,coev_A,C,\matholdcal N)\circ (\varphi,A,C,\matholdcal N)\circ  (A,\psi,\matholdcal N)\circ (\psi,A,\matholdcal N)&\\ 
			&=\pi_{ \matholdcal N}\circ (\Delta,\matholdcal N)\circ (C,coev_A,C,\matholdcal N)\circ(\varphi,\psi,\matholdcal N)\circ (\psi, A,\matholdcal N)&\\
			&=\pi_{ \matholdcal N}\circ(\Delta, coev_A,\matholdcal N)\circ (C,\varphi, A,\matholdcal N)& \mbox{(by Definition \ref{D7.1}))} \\
			&=\pi_{\matholdcal N}\circ (\Delta,\matholdcal N)\circ (C,C,coev_A,\matholdcal N)\circ (C,\varphi,A,\matholdcal N)&\\
			&=\pi_{ \matholdcal M}\circ (C,\pi_{ \matholdcal N})\circ (C,C,coev_A,\matholdcal N)\circ (C,\varphi,A,\matholdcal N) & \\
			\end{array}
		\end{equation}
		Hence, diagram in (\ref{eq7.4}) commutes. This completes the proof.
	\end{proof}
	\begin{thm}\label{P7.3}
		Let  $(A,C,\psi)$ be an entwining structure and $\varphi$ be a normalized cointegral for $(A,C,\psi)$. Then, a morphism in $_{\hspace{1em}A}^{[C,-]}\mathfrak S(\psi)$ has a section (resp. retraction) in $_{\hspace{1em}A}^{[C,-]}\mathfrak S(\psi)$ if and only if it has a section (resp. retraction) in $^{[C,-]}\mathfrak S$.
	\end{thm}
	\begin{proof}
		We begin with a morphism $\xi:\matholdcal M\longrightarrow \matholdcal N$ in $_{\hspace{1em}A}^{[C,-]}\mathfrak S(\psi)$. In that case, we note that
		\begin{equation}
		\begin{array}{lll}
			\tilde{\xi}&=\pi_{ \matholdcal N}\circ (C,coev_A,\matholdcal N)\circ (\varphi,A,\matholdcal N)\circ (A,\mu_{\matholdcal N})\circ (A,\xi)\circ \mu_{\matholdcal M}&\\
			&=\pi_{ \matholdcal N}\circ (C,coev_A,\matholdcal N)\circ (\varphi,A,\matholdcal N)\circ (A,\mu_{\matholdcal N})\circ\mu_{\matholdcal N}\circ \xi &\\
			&=\pi_{ \matholdcal N}\circ (C,coev_A,\matholdcal N)\circ (\varphi,A,\matholdcal N)\circ (\mu,\matholdcal N)\circ\mu_{\matholdcal N}\circ \xi &\\
			&=\pi_{ \matholdcal N}\circ (\epsilon,\matholdcal N)\circ (\eta,\matholdcal N)\circ \mu_{\matholdcal N}\circ \xi
			&\mbox{(by Definition \ref{D7.1})} \\ 
			&=\xi &
			\end{array}
		\end{equation}
		We now consider a morphism $\zeta:\matholdcal N\longrightarrow \matholdcal M$ such that $\zeta$ is a section of $\xi$ in $^{[C,-]}\mathfrak S$, i.e., $\xi\circ \zeta=id$. By Lemma \ref{L7.2}, we obtain a morphism $\tilde{\zeta}:\matholdcal N\longrightarrow \matholdcal M$ in $_{\hspace{1em}A}^{[C,-]}\mathfrak S(\psi)$. Then, it is easy to see that $\xi\circ \tilde{\zeta}=\widetilde{\xi\circ \zeta}=id.$ This shows that $\tilde{\zeta}:\matholdcal N\longrightarrow\matholdcal M$ is a section of $\xi$ in $_{\hspace{1em}A}^{[C,-]}\mathfrak S(\psi)$. 
		The case for retraction can also be proved similarly.
	\end{proof}
	We now recall the notions of semisimple and Maschke functors from \cite[\S 3]{CM}. Let $\mathcal T:\mathcal E\longrightarrow \mathcal D$ be a functor between abelian categories. The functor $\mathcal T$ is said to be \textit{semisimple} if it reflects split exact sequences, i.e., a short exact sequence $0\longrightarrow \matholdcal N\longrightarrow \matholdcal M\longrightarrow \matholdcal P\longrightarrow 0$ is split exact in $\mathcal E$   if $0\longrightarrow \mathcal T(\matholdcal N)\longrightarrow \mathcal T(\matholdcal M)\longrightarrow \mathcal T(\matholdcal P)\longrightarrow 0$ is split exact in $\mathcal D$.
	The functor $\mathcal T$ is said to be \textit{Maschke} if for any morphisms $i:\matholdcal M'\longrightarrow \matholdcal M$ and $f:\matholdcal M'\longrightarrow \matholdcal N$ in $\matholdcal E$ such that $\mathcal T(\iota):\mathcal T(\matholdcal M')\longrightarrow \mathcal T(\matholdcal M)$ is a split monomorphism in $\mathcal D$, then there exists a morphism $g:\matholdcal M\longrightarrow \matholdcal N$ such that $g\circ i=f$. Additionally, \cite[Proposition 3.7]{CM} shows that if $\mathcal T$ is semisimple and reflects monomorphisms (or epimorphisms), then $\mathcal T$ is a Maschke functor.
	\begin{Thm}\label{T7.4}
		Let $\mathfrak S$ be a $k$-linear Grothendieck category. Let  $(A,C,\psi)$ be an entwining structure and $\varphi:A^\ast\otimes C\longrightarrow A$ be a normalized cointegral for $(A,C,\psi)$. Then,
		$^{[C,-]}\mathcal F:~_{\hspace{1em}A}^{[C,-]}\mathfrak S(\psi)\longrightarrow~ ^{[C,-]}\mathfrak S$ is a semisimple and Maschke functor.
		
	\end{Thm}
	\begin{proof}
		We consider an exact sequence $0\longrightarrow \matholdcal N\longrightarrow \matholdcal M\longrightarrow \matholdcal P\longrightarrow 0$ in $~_{\hspace{1em}A}^{[C,-]}\mathfrak S(\psi)$ that splits in $^{[C,-]}\mathfrak S$. By  Proposition \ref{P7.3}, it is clear that $0\longrightarrow \matholdcal N\longrightarrow \matholdcal M\longrightarrow \matholdcal P\longrightarrow 0$ also splits in $~_{\hspace{1em}A}^{[C,-]}\mathfrak S(\psi)$. This shows that $^{[C,-]}\mathcal F$ is a semisimple functor. Since kernels in $~_{\hspace{1em}A}^{[C,-]}\mathfrak S(\psi)$ and $^{[C,-]}\mathfrak S$ are both computed in $\mathfrak S$, it is clear that the forgetful functor $^{[C,-]}\mathcal F$ preserves and reflects monomorphisms. It now follows from \cite[Proposition 3.7]{CM} that the functor $^{[C,-]}\mathcal F$ is Maschke.
	\end{proof}
	We conclude by giving a brief sketch of the corresponding results in the case of entwined comodule objects in $\mathfrak S$ over $(A,C,\psi)$. Let $(\mathcal M,\Delta_{\mathcal M},\mu_{\mathcal M}),$ $(\mathcal N,\Delta_{\mathcal N},\mu_{\mathcal N})\in ~\mathfrak S_A^C(\psi)$ and suppose $\xi:\mathcal M\longrightarrow \mathcal N$ is a morphism in $\mathfrak S^C$. In a manner similar to Brzezi\'nski \cite[Lemma 4.7]{Tb*},  there exists a morphism $\hat{\xi}:\mathcal M\longrightarrow \mathcal N$ in $\mathfrak S_A^C(\psi)$ given by the composition
	\begin{equation}
		\hat{\xi}=\mu_{\mathcal N}\circ (\xi\otimes A)\circ (\mu_{\mathcal M}\otimes A)\circ (\mathcal M\otimes A\otimes \varphi)\circ (\mathcal M\otimes coev_A\otimes C)\circ \Delta_{\mathcal M}
	\end{equation}
	Accordingly, it may be checked that a morphism in $\mathfrak S_A^C(\psi)$ has a section (resp. retraction) in $\mathfrak S_A^C(\psi)$ if and only if it has a section (resp. retraction) in $\mathfrak S^C$. Moreover, the following result holds for the forgetful functor $\mathcal F^C:\mathfrak S_A^C(\psi)\longrightarrow \mathfrak S^C$.
	\begin{Thm}\label{tT7.5}
		Let $\mathfrak S$ be a $k$-linear Grothendieck category. Let  $(A,C,\psi)$ be an entwining structure and $\varphi:A^\ast\otimes C\longrightarrow A$ be a normalized cointegral for $(A,C,\psi)$. Then,
		$\mathcal F^C:\mathfrak S_A^C(\psi)\longrightarrow~ \mathfrak S^C$ is a semisimple and Maschke functor.
		
	\end{Thm}
	\begin{proof}
		The proof of this is similar to that of Theorem \ref{T7.4}.
	\end{proof}

\end{document}